\documentclass[11pt]{article}
\usepackage{amsfonts, amsmath, amssymb, latexsym, eucal, amscd,cite} 
\usepackage[all]{xy}
\newtheorem{theorem}{Theorem}[section]
\newtheorem{lemma}[theorem]{Lemma}
\newtheorem{proposition}[theorem]{Proposition}
\newtheorem{corollary}[theorem]{Corollary}
\newtheorem{exAux}[theorem]{Example}

\newtheorem{Def}[theorem]{Definition}
\newenvironment{definition}{\begin{Def} \rm}{\end{Def}}
\newtheorem{Note}[theorem]{Note}
\newenvironment{note}{\begin{Note} \rm}{\end{Note}}
\newtheorem{Problem}[theorem]{Problem}

\newtheorem{Rem}[theorem]{Remark}
\newenvironment{remark}{\begin{Rem} \rm}{\end{Rem}}
\newtheorem{Not}[theorem]{Notation}
\newenvironment{notation}{\begin{Not} \rm}{\end{Not}}
\newtheorem{Conj}[theorem]{Conjecture}

\newtheorem{Ass}[theorem]{Assumption}

\newenvironment{proof}{\medskip\noindent{\bf Proof.\ }}{\qed\medskip}
\newenvironment{proofof}[1]{\medskip\noindent{\bf Proof  of {#1}.\ 
}}{\qed\medskip}
\newcommand{\qed}{\hfill\mbox{$\Box$\qquad\qquad}}

\newcommand{\F}{\mathbb{F}}

\newcommand{\vphi}{\varphi}
\renewcommand{\th}{\theta}

\newcommand{\Hq}{\hat{H}_q}
\newcommand{\I}{\mathbb{I}}

\newcommand{\Z}{\mathbb{Z}}

\newcommand{\A}{\mathbb{A}}
\newcommand{\B}{\mathbb{B}}
\newcommand{\C}{\mathbb{C}}
\newcommand{\V}{\mathcal{V}}
\newcommand{\W}{\mathcal{W}}

\newcommand{\ve}{\varepsilon}

\newif\ifDRAFT
\DRAFTfalse

\begin{document}

\title{The universal DAHA of type $(C_1^\vee,C_1)$ and \\
Leonard pairs of $q$-Racah type}

\author{Kazumasa Nomura and Paul Terwilliger}
\date{}
\maketitle

{
\small

\begin{quote}
\begin{center}
{\bf Abstract}
\end{center}
A Leonard pair is a pair of diagonalizable linear transformations
of a finite-dimensional vector space, 
each of which acts in an irreducible tridiagonal fashion on an
eigenbasis for the other one.
Let $\F$ denote an algebraically closed field,
and fix a nonzero $q \in \F$ that is not a root of unity.
The universal double affine Hecke algebra (DAHA) $\Hq$ of type 
$(C_1^\vee,C_1)$ is the associative $\F$-algebra defined by
generators $\{t_i^{\pm 1}\}_{i=0}^3$ and relations
(i) $t_it_i^{-1}=t_i^{-1}t_i=1$;
(ii) $t_i+t_i^{-1}$ is central;
(iii) $t_0t_1t_2t_3 = q^{-1}$.
We consider the elements $X=t_3t_0$ and $Y=t_0t_1$ of $\Hq$.
Let $\V$ denote a finite-dimensional irreducible $\Hq$-module
on which each of $X$, $Y$ is diagonalizable and $t_0$ has
two distinct eigenvalues.
Then $\V$ is a direct sum of the two eigenspaces of $t_0$.
We show that the pair $X+X^{-1}$, $Y+Y^{-1}$ acts on each eigenspace 
as a Leonard pair, 
and each of these Leonard pairs falls into a class said to have $q$-Racah type.
Thus from $\V$ we obtain a pair of Leonard pairs of $q$-Racah type.
It is known that a Leonard pair of $q$-Racah type is determined up to isomorphism by
a parameter sequence $(a,b,c,d)$ called its Huang data.
Given a pair of Leonard pairs of $q$-Racah type, we find
necessary and sufficient conditions on their Huang data for that pair to come from 
the above construction.
\end{quote}
}

\section{Introduction}
\label{sec:intro}

Throughout the paper $\F$ denotes an algebraically closed field.
Fix a nonzero $q \in \F$ that is not a root of unity.
An $\F$-algebra is meant to be associative and have a $1$.

We begin by recalling the notion of a Leonard pair.
We use the following terms.
A square matrix is said to be {\em tridiagonal} 
whenever each nonzero entry lies on either the diagonal,
the subdiagonal, or the superdiagonal.
A tridiagonal matrix is said to be {\em irreducible}
whenever each entry on the subdiagonal is nonzero and each entry on the
superdiagonal is nonzero.

\begin{definition} {\rm (See \cite[Definition 1.1]{T:Leonard}.)}  
\label{def:LP}     \samepage
\ifDRAFT {\rm def:LP}. \fi
Let $V$ denote a vector space over $\F$ with finite positive dimension.
By a {\em Leonard pair} on $V$ we mean an ordered pair of $\F$-linear transformations
$A: V \to V$ and $A^* : V \to V$ that satisfy (i) and (ii) below:
\begin{itemize}
\item[\rm (i)]
there exists a basis for $V$ with respect to which the matrix representing $A$
is irreducible tridiagonal and the matrix representing $A^*$ is diagonal;
\item[\rm (ii)]
there exists a basis for $V$ with respect to which the matrix representing $A^*$
is irreducible tridiagonal and the matrix representing $A$ is diagonal.
\end{itemize}
We say that $A,A^*$ is {\em over $\F$}.
By the {\em diameter} of $A,A^*$ we mean $\dim V-1$.
\end{definition}

\begin{note}
According to a common notational convention, 
$A^*$ denotes the conjugate-trans\-pose of $A$.
We are not using this convention.
In a Leonard pair $A,A^*$ the $\F$-linear transformations $A$ and $A^*$ are arbitrary 
subject to (i) and (ii) above.
\end{note}

We refer the reader to \cite{Cur:SpinLP, Hu,  NT:both, T:Leonard, T:survey, TV}
for background on Leonard pairs.

We recall the notion of an isomorphism of Leonard pairs.
Let $V$ (resp. $V'$) denote a vector space over $\F$ with finite positive dimension, and 
let $A,A^*$ (resp.\ $A',A^{*\prime}$) denote a Leonard pair on $V$ (resp. $V'$).
By an {\em isomorphism of Leonard pairs} from $A,A^*$ to $A',A^{*\prime}$
we mean an $\F$-linear bijection $f: V \to V'$ such that both
$f A = A' f$ and $f A^* = A^{*\prime} f$.

We recall some facts about the eigenvalues of a Leonard pair.
We use the following terms.
Let $V$ denote a vector space over $\F$ with finite positive dimension
and let $A: V \to V$ denote an $\F$-linear transformation.
Then $A$ is said to be {\em diagonalizable} whenever $V$ is spanned by
the eigenspaces of $A$.
We say $A$ is {\em multiplicity-free} whenever $A$ is diagonalizable
and each eigenspace of $A$ has dimension one.
Let $A,A^*$ denote a Leonard pair on $V$.
Then each of $A,A^*$ is multiplicity-free (see \cite[Lemma 1.3]{T:Leonard}).
Let $\{\th_r\}_{r=0}^d$ denote an ordering of the eigenvalues of $A$.
For $0 \leq r \leq d$ let $0 \neq v_r \in V$ denote an
eigenvector of $A$ associated with $\th_r$.
Observe that $\{v_r\}_{r=0}^d$ is a basis for $V$.
The ordering $\{\th_r\}_{r=0}^d$ is said to be {\em standard}
whenever the basis $\{v_r\}_{r=0}^d$ satisfies Definition \ref{def:LP}(ii).
A standard ordering of the eigenvalues of $A^*$ is similarly defined.
Let $\{\th_r\}_{r=0}^d$ denote a standard ordering of the eigenvalues of $A$.
Then the ordering $\{\th_{d-r}\}_{r=0}^d$ is also standard and no further
ordering is standard. A similar result applies to $A^*$. 

\begin{definition}    \label{def:qRacah}     \samepage
\ifDRAFT {\rm def:qRacah}. \fi
Let $d \geq 0$ denote an integer and let $\{\th_r\}_{r=0}^d$ denote a
sequence of scalars in $\F$.
The sequence $\{\th_r\}_{r=0}^d$ is said to be {\em $q$-Racah}
whenever there exists a nonzero $\alpha \in \F$ such that
$\th_r = \alpha q^{2r-d} + \alpha^{-1} q^{d-2r}$ for $0 \leq r \leq d$.
The scalar $\alpha$ is uniquely determined by $\{\th_r\}_{r=0}^d$
if $d \geq 1$, and determined up to inverse if $d=0$.
We call $\alpha$ the {\em parameter} of the $q$-Racah sequence.
\end{definition}

For a $q$-Racah sequence $\{\th_r\}_{r=0}^d$ with parameter $\alpha$,
the inverted sequence $\{\th_{d-r}\}_{r=0}^d$ is $q$-Racah 
with parameter $\alpha^{-1}$.

\begin{definition}  \label{def:LPqRacah}   \samepage
\ifDRAFT {\rm def:LPqRacah}. \fi
Let $A,A^*$ denote a Leonard pair over $\F$ with diameter $d$.
Then $A,A^*$ is said to have {\em $q$-Racah type} whenever
for each of $A,A^*$ 
a standard ordering of the eigenvalues forms a $q$-Racah sequence.
\end{definition} 

Referring to Definition \ref{def:LPqRacah},
assume that $A,A^*$ has $q$-Racah type.
Let $\{\th_r\}_{r=0}^d$ (resp.\ $\{\th^*_r\}_{r=0}^d$) 
denote a standard ordering of the eigenvalues of $A$ (resp.\ $A^*$).
Let $a$ (resp.\ $b$) denote the parameter of the $q$-Racah 
sequence $\{\th_r\}_{r=0}^d$ (resp.\ $\{\th^*_r\}_{r=0}^d$).
It is known that $A,A^*$ is determined up to isomorphism by
$a$, $b$, $d$ and one more nonzero scalar $c \in \F$.
The sequence $(a,b,c,d)$ is called a {\em Huang data} of $A,A^*$.
The scalar $c$ is defined up to inverse if $d \geq 1$,
and arbitrary if $d=0$
(see Section \ref{sec:pre}).
For a Huang data $(a,b,c,d)$ of $A,A^*$,
each of $(a^{\pm 1}, b^{\pm 1}, c^{\pm 1}, d)$ is a Huang data of $A,A^*$.
Moreover $A,A^*$ has no further Huang data, provided that $d \geq 1$.

\medskip

Next we recall the universal double affine Hecke algebra $\Hq$.
The double affine Hecke algebra (DAHA) was introduced by Cherednik \cite{Chered}.
The DAHA of type $(C_1^\vee,C_1)$ was studied by
Macdonald \cite[Ch.\ 6]{Mac},
Noumi-Stokman \cite{Noumi},
Sahi  \cite{Sahi1, Sahi2},
Koornwinder  \cite{Koor1, Koor2}, 
and  Ito-Terwilliger \cite{IT:DAHA}.
The algebra $\Hq$ was introduced by the second author 
as a generalization of the DAHA of type $(C_1^\vee, C_1)$.
We now recall the definition of $\Hq$.
For notational convenience define $\I = \{0,1,2,3\}$.

\begin{definition} {\rm (See \cite[Definition 3.1]{T:AWDAHA}.)}  \label{def:Hq}  \samepage
\ifDRAFT {\rm def:Hq}. \fi
Let $\Hq$ denote the $\F$-algebra defined by 
generators $\{t^{\pm 1}_i\}_{i \in \I}$ and relations
\begin{align} 
 & t_it_i^{-1}=t_i^{-1}t_i=1  && i \in \I,            \label{eq:defHq1}
\\
 & \text{$t_i+t_i^{-1}$ is central} && i \in \I,      \label{eq:defHq2}
\\
 & t_0t_1t_2t_3 = q^{-1}.                             \label{eq:defHq3}
\end{align}
The algebra $\Hq$ is called the {\em universal DAHA of type $(C_1^\vee,C_1)$}.
\end{definition}

Referring to Definition \ref{def:Hq},
for $i \in \I$ define 
\[
   T_i = t_i + t_i^{-1}.
\]
Note that $T_i$ is central in $\Hq$.

\begin{definition}   \label{def:param}
\ifDRAFT {\rm def:param.} \fi
Let $\V$ denote a finite-dimensional irreducible $\Hq$-module.
By Schur's lemma each $T_i$ acts on $\V$ as a scalar.
Write this scalar as $k_i + k_i^{-1}$ with $0 \neq k_i \in \F$.
Thus
\begin{equation}          \label{eq:ki}
   T_i = k_i + k_i^{-1}   \qquad \text{ on $\V$}.
\end{equation}
We refer to $\{k_i\}_{i \in \I}$ as a {\em parameter sequence} of $\V$.
\end{definition}

Referring to Definition \ref{def:param}, 
note that each $k_i$ is defined up to inverse.
So each of the $16$ sequences $\{k_i^{\pm 1}\}_{i \in \I}$ is a parameter sequence of $\V$,
and $\V$ has no further parameter sequence.
Observe by \eqref{eq:ki} that $(t_i-k_i)(t_i-k_i^{-1})\V=0$.
Thus the eigenvalues of $t_i$ are among $k_i$, $k_i^{-1}$.

We consider the following elements of $\Hq$:
\begin{align}
 X &= t_3t_0, & Y &= t_0t_1, &
 \A &= Y+Y^{-1}, & \B &= X+X^{-1}.                  \label{eq:defXYAB}
\end{align}
It is known that each of $\A$, $\B$ commutes with $t_0$ (see Lemma \ref{lem:titj}).
By an {\em XD} (resp.\ {\em YD}) {\em $\Hq$-module} we mean a finite-dimensional
irreducible $\Hq$-module on which $X$ (resp.\ $Y$) is diagonalizable.
An $\Hq$-module is said to be {\em feasible} whenever
(i) it is both XD and YD;
(ii) $t_0$ has two distinct eigenvalues.
Let $\V$ denote a feasible $\Hq$-module with parameter sequence $\{k_i\}_{i \in \I}$.
Then $k_0^2 \neq 1$. Moreover $t_0$ is diagonalizable on $\V$ with eigenvalues $k_0$ and $k_0^{-1}$.
Observe that $\V$ is a direct sum of the two eigenspaces of $t_0$, and each eigenspace
is invariant under $\A$, $\B$.
We remark that $t_0$ does not commute with $X$, $Y$,
and so the eigenspaces of $t_0$  may not  be invariant under $X$, $Y$.

We now state our first main result.

\begin{theorem}   \label{thm:main1}   \samepage
\ifDRAFT {\rm thm:main1}. \fi
Let $\V$ denote a feasible $\Hq$-module.
Then $\A,\B$ act on each eigenspace of $t_0$ as a Leonard pair of $q$-Racah type.
\end{theorem}

Let $\V$ denote a feasible $\Hq$-module.
By Theorem \ref{thm:main1} we obtain a pair of Leonard pairs of $q$-Racah type.
In order to describe how these Leonard pairs
are related, we use the following notion.
Let $A,A^*$ denote a Leonard pair on $V$ and let 
$A',A^{*\prime}$ denote a Leonard pair on $V'$.
We say that these Leonard pairs are {\em linked} whenever 
the direct sum $V \oplus V'$ supports a feasible $\Hq$-module structure such that
$V$, $V'$ are the eigenspaces of $t_0$ and
\begin{align}
 \A|_{V} &= A, &  \B|_{V} &= A^*, & 
 \A|_{V'} &= A', & \B|_{V'} &= A^{*\prime}.       \label{eq:linked}
\end{align}
We now state our second main result.

\begin{theorem}  \label{thm:main} \samepage
\ifDRAFT {\rm thm:main}. \fi
Suppose we are given two Leonard pairs $A,A^*$ and 
$A',A^{*\prime}$ over $\F$ that have $q$-Racah type.
Then these Leonard pairs are linked if and only if 
there exist a Huang data $(a,b,c,d)$ of $A,A^*$ 
and a Huang data $(a',b',c',d')$ of $A',A^{*\prime}$ that satisfy
one of {\rm (i)--(vii)} below:
\[
\begin{array}{r|c|c|c|c|ccc}
\text{\rm Case} & d'-d & a'/a & b'/b & c'/c &  & \text{\rm Inequalities} &  
\\ \hline 
\rule{0mm}{4.3mm}
\text{\rm (i)} & -2 & 1 & 1 & 1
\\
\rule{0mm}{4mm}
\text{\rm (ii)} & -1 & q & q & q & \quad a^2 \neq q^{-2d} & \quad  b^2 \neq q^{-2d}
\\
\rule{0mm}{4mm}
\text{\rm (iii)} & 0 & q^2 & 1 & 1 & & \quad b^2 \neq q^{\pm 2d} & \quad a^2 \neq q^{-2}
\\
\rule{0mm}{4mm}
\text{\rm (iv)} & 0 & 1 & q^2 & 1  & \quad a^2 \neq q^{\pm 2d} & & \quad b^2 \neq q^{-2}
\\
\rule{0mm}{4mm}
\text{\rm (v)} & 0 & 1 & 1 & q^2 & \quad a^2 \neq q^{\pm 2d} & \quad b^2 \neq q^{\pm 2d } & \quad c^2 \neq q^{-2} 
\\
\rule{0mm}{4mm}
\text{\rm (vi)} & 1 & q^{-1} & q^{-1} & q^{-1} & \quad a^2 \neq q^{-2d} & \quad b^2 \neq q^{-2d}
\\
\rule{0mm}{4mm}
\text{\rm (vii)} & 2 & 1 & 1 & 1
\end{array}
\]
\end{theorem}

\begin{remark}    \label{rem:thmmain}
\ifDRAFT {\rm rem:thmmain}. \fi 
Referring to Theorem \ref{thm:main},
in each of (ii)--(vi) there appear some inequalities.
For each of these inequalities
we explain the role in our construction of a feasible $\Hq$-module.
The inequalities in the first column are needed to make $Y$ diagonalizable.
The inequalities in the second column are needed to
make $X$ diagonalizable.
The inequalities in the third column
are needed to make $t_0$ have two distinct eigenvalues.
\end{remark}

\begin{remark}    \label{rem:deleted}
\ifDRAFT {\rm rem:deleted}. \fi
Referring to Theorem \ref{thm:main},
for $d\geq 2$ we have $a^2 \not=q^{-2}$ and
$b^2 \not=q^{-2}$ by Lemma \ref{lem:condabc}(i), so these
inequalities can be deleted from the table.
\end{remark}

\begin{remark}    \label{rem:exchange}
\ifDRAFT {\rm rem:exchange}. \fi
Suppose we exchange our two Leonard pairs in Theorem \ref{thm:main}.
In the conditions (i)--(vii), 
the Huang data $(a,b,c,d)$ and $(a',b',c',d')$ are exchanged.
After this exchange, the conditions (i), (ii), (vi), (vii) become the original conditions
(vii), (vi), (ii), (i) respectively.
Concerning the conditions (iii)--(v), we replace each of $a$, $b$, $c$, $a'$, $b'$, $c'$ with its inverse
after the above exchange.
Then the conditions (iii)--(v) become the original conditions (iii)--(v) respectively.
\end{remark}

\smallskip

In a moment we will summarize the proof of Theorems \ref{thm:main1} and \ref{thm:main}.
Prior to that we explain the significance of the cases
that show up in Theorem \ref{thm:main}.
Let $\V$ denote an XD $\Hq$-module.
For $\mu \in \F$, let the subspace $\V_X(\mu)$ consist of the vectors $v \in \V$
such that $Xv=\mu v$.
Thus $\V_X(\mu)$ is nonzero if and only if $\mu$ is an eigenvalue of $X$,
and in this case $\V_X(\mu)$ is the corresponding eigenspace.
As we will see in Lemmas \ref{lem:t0t2inv} and \ref{lem:Hqinv}, for $\mu$, $\nu \in \F$
with $\mu \nu = 1$  (resp.\ $\mu\nu=q^{-2}$)
the subspace  $\V_X(\mu) + \V_X(\nu)$ is invariant under $t_0$, $t_3$
(resp.\ $t_1$, $t_2$).
Motivated by this, we consider the following diagram.

Let $\mu$, $\nu \in \F$.
We say $\mu$, $\nu$ are {\em $1$-adjacent} (resp. {\em $q$-adjacent})
whenever their product is $1$ (resp.\ $q^{-2}$).
For a subset $\cal M$ of $\F$ we define the {\em diagram} of $\cal M$,
that has vertex set $\cal M$, and $\mu$, $\nu \in {\cal M}$ are connected
by a single bond (resp.\ double bond) whenever $\mu$, $\nu$ are $1$-adjacent 
(resp. $q$-adjacent).
If $\mu=\nu$ then the single bond (resp.\ double bond) becomes a single loop
(resp.\ double loop).
The {\em reduced diagram} of $\cal M$ is obtained by deleting all the loops in the
diagram of $\cal M$.

Let $\V$ denote an XD $\Hq$-module.
By the {\em $X$-diagram} of $\V$ we mean the diagram of $\cal M$,
where $\cal M$ is the set of eigenvalues of $X$ on $\V$.
As we will see in Lemma \ref{lem:diagram}, the reduced $X$-diagram is a path.
So the reduced $X$-diagram has one of the following types:
\begin{align}
 \textup{\sf DS}: & \qquad\qquad
 \begin{xy}
  \ar @{=} (-10,0) *\cir<2pt>{};
            (0,0)  *\cir<2pt>{}="B"
  \ar @{-} "B" ; (10,0)  *\cir<2pt>{}="C"
  \ar @{=} "C" ; (20,0) *\cir<2pt>{}="D"
  \ar @{-} "D" ; (30,0) *{\cdots\cdots}="E"
  \ar @{=} "E" ; (40,0)  *\cir<2pt>{}="F"
  \ar @{-} "F" ; (50,0)  *\cir<2pt>{};
 \end{xy}                                                                \notag
\\
 \textup{\sf DD}:
 &  \qquad \qquad
 \begin{xy}
  \ar @{=} (-10,0)  *\cir<2pt>{}; 
            (0,0)   *\cir<2pt>{}="B"
  \ar @{-} "B" ; (10,0)  *\cir<2pt>{}="C"
  \ar @{=} "C" ; (20,0)  *\cir<2pt>{}="D"
  \ar @{-} "D" ; (30,0) *{\cdots\cdots}="E"
  \ar @{-} "E" ; (40,0)  *\cir<2pt>{}="F"
  \ar @{=} "F" ; (50,0)  *\cir<2pt>{};
 \end{xy}                                                                \label{eq:diag}
\\
 \textup{\sf SS}:
 & \qquad\qquad
 \begin{xy}
  \ar @{-} (-10,0)  *\cir<2pt>{}; 
            (0,0)   *\cir<2pt>{}="B"
  \ar @{=} "B" ; (10,0)  *\cir<2pt>{}="C"
  \ar @{-} "C" ; (20,0) *\cir<2pt>{}="D"
  \ar @{=} "D" ; (30,0) *{\cdots\cdots}="E"
  \ar @{=} "E" ; (40,0)  *\cir<2pt>{}="F"
  \ar @{-} "F" ; (50,0)   *\cir<2pt>{};
 \end{xy}                                                                   \notag
\end{align}
If the above diagram has only one vertex, we interpret it to be {\sf DS}. 
For each of \eqref{eq:diag} we choose an ordering of the eigenvalues 
$\{\mu_r\}_{r=0}^n$ of $X$ as follows:
\begin{align}
 \textup{\sf DS:} & \qquad\qquad
 \begin{xy}
  \ar @{=} (-10,0) *++!D{\mu_0} *\cir<2pt>{}; 
            (0,0)   *++!D{\mu_1} *\cir<2pt>{}="B"
  \ar @{-} "B" ; (10,0) *++!D{\mu_2} *\cir<2pt>{}="C"
  \ar @{=} "C" ; (20,0) *++!D{\mu_3} *\cir<2pt>{}="D"
  \ar @{-} "D" ; (30,0) *{\cdots\cdots}="E"
  \ar @{=} "E" ; (40,0) *++!D{\mu_{n-1}} *\cir<2pt>{}="F"
  \ar @{-} "F" ; (50,0) *++!D{\mu_n}  *\cir<2pt>{};
 \end{xy}
 & \qquad                                               \notag
\\
 \textup{\sf DD:}
 &  \qquad \qquad
 \begin{xy}
  \ar @{=} (-10,0) *++!D{\mu_0} *\cir<2pt>{}; 
            (0,0)   *++!D{\mu_1} *\cir<2pt>{}="B"
  \ar @{-} "B" ; (10,0) *++!D{\mu_2} *\cir<2pt>{}="C"
  \ar @{=} "C" ; (20,0) *++!D{\mu_3} *\cir<2pt>{}="D"
  \ar @{-} "D" ; (30,0) *{\cdots\cdots}="E"
  \ar @{-} "E" ; (40,0) *++!D{\mu_{n-1}} *\cir<2pt>{}="F"
  \ar @{=} "F" ; (50,0) *++!D{\mu_n}  *\cir<2pt>{};
 \end{xy}                                                         \label{eq:standard}
\\
 \textup{\sf SS:}
 & \qquad\qquad
 \begin{xy}
  \ar @{-} (-10,0) *++!D{\mu_0} *\cir<2pt>{}; 
            (0,0)   *++!D{\mu_1} *\cir<2pt>{}="B"
  \ar @{=} "B" ; (10,0) *++!D{\mu_2} *\cir<2pt>{}="C"
  \ar @{-} "C" ; (20,0) *++!D{\mu_3} *\cir<2pt>{}="D"
  \ar @{=} "D" ; (30,0) *{\cdots\cdots}="E"
  \ar @{=} "E" ; (40,0) *++!D{\mu_{n-1}} *\cir<2pt>{}="F"
  \ar @{-} "F" ; (50,0) *++!D{\mu_n}  *\cir<2pt>{};
 \end{xy}                                                                 \notag
\end{align}
It turns out that each eigenspace of $X$ has dimension one (see Proposition \ref{prop:mfree}).
For each end-vertex $\mu$ of the diagram,
we consider the action of $\{t_i\}_{i \in \I}$ on $\V_X(\mu)$.
As we will see in Lemma \ref{lem:endvertex},
\begin{equation}                              \label{eq:end}
\begin{array}{c|c|c}
\text{Case}
 & \text{$\V_X(\mu_0)$ is invariant under} 
 & \text{$\V_X(\mu_n)$ is invariant under}
\\ \hline
\text{\sf DS} \rule{0mm}{4mm}  
  & t_0,\qquad t_3 & t_1,\qquad t_2 \\
\text{\sf DD}  & t_0,\qquad t_3 & t_0,\qquad t_3 \\
\text{\sf SS}  & t_1,\qquad t_2 & t_1,\qquad t_2
\end{array}
\end{equation}
For each $t_i$, consider the action of $t_i$ on $\V$.
For this action any two distinct eigenvalues are reciprocals.
For the reduced $X$-diagram {\sf DD} one of the following cases 
occurs (see Lemma \ref{lem:typesDD}):
\begin{equation}                 \label{eq:DDaDDb}
\begin{array}{c|c|c}
 \text{Case} 
 & 
   \begin{array}{c}
     \text{Eigenvalues of $t_0$ on} \\
     \text{$\V_X(\mu_0)$ and $\V_X(\mu_n)$}
   \end{array}
 &
   \begin{array}{cc}
     \text{Eigenvalues of $t_3$ on} \\
     \text{$\V_X(\mu_0)$ and $\V_X(\mu_n)$}
   \end{array}
\\ \hline
\text{\sf DDa} \rule{0mm}{4mm}
  & \text{same} & \text{reciprocals} \\
\text{\sf DDb} & \text{reciprocals} & \text{same}
\end{array}
\end{equation}
Similarly, for the reduced $X$-diagram {\sf SS} one of the following cases 
occurs (see Lemma \ref{lem:typesSS}):
\begin{equation}                  \label{eq:SSaSSb}
\begin{array}{c|c|c}
 \text{Case} 
 & 
   \begin{array}{c}
     \text{Eigenvalues of $t_1$ on} \\
     \text{$\V_X(\mu_0)$ and $\V_X(\mu_n)$}
   \end{array}
 &
   \begin{array}{cc}
     \text{Eigenvalues of $t_2$ on} \\
     \text{$\V_X(\mu_0)$ and $\V_X(\mu_n)$}
   \end{array}
\\ \hline
\text{\sf SSa} \rule{0mm}{4mm}
 & \text{same} & \text{reciprocals} \\
\text{\sf SSb} & \text{reciprocals} & \text{same}
\end{array}
\end{equation}
We refer to each case \textup{\sf DS},
\textup{\sf DDa}, \textup{\sf DDb}, \textup{\sf SSa}, \textup{\sf SSb}
as the {\em $X$-type of $\V$}.
An XD $\Hq$-module is determined up to isomorphism by its dimension, its parameter sequence, and
its $X$-type (see Note \ref{note:iso}).
The cases (i)--(v) in Theorem \ref{thm:main} correspond to the $X$-types as follows:
\begin{equation}           \label{eq:corres}
\begin{array}{ccccc}
 \text{(i)} & \text{(ii)} & \text{(iii)} & \text{(iv)} & \text{(v)}
\\ \hline
 \textup{\sf DDa} & \textup{\sf DS} & \textup{\sf SSa} & \textup{\sf DDb} & \textup{\sf SSb}
\end{array}
\end{equation}
The cases (vi) and (vii) are reduced to (ii) and (i) by exchanging our two Leonard pairs
(see Remark \ref{rem:exchange}).
 
Recall that $\V$ has $16$ parameter sequences $\{k_i^{\pm 1}\}_{i \in \I}$.
In view of \eqref{eq:end}--\eqref{eq:SSaSSb} we adopt the following convention
for most of the paper:
\begin{equation}     \label{eq:rule0}
\begin{array}{c|c}
\text{Case} & \text{Rule}
\\ \hline
\text{\sf DS} \rule{0mm}{6.5mm} &
\begin{array}{l}
\text{$k_0$ (resp.\ $k_3$) is the eigenvalue of $t_0$ (resp.\ $t_3$) on $\V_X(\mu_0)$} \\
\text{$k_1$ (resp.\ $k_2$) is the eigenvalue of $t_1$ (resp.\ $t_2$) on $\V_X(\mu_n)$}
\end{array}
\\ \hline
\text{\sf DD} \rule{0mm}{4.2mm} & 
\begin{array}{l}
\text{$k_0$ (resp.\ $k_3$) is the  eigenvalue of $t_0$ (resp.\ $t_3$) on $\V_X(\mu_0)$}
\end{array}
\\ \hline
\text{\sf SS} \rule{0mm}{4.2mm} &
\begin{array}{l}
\text{$k_1$ (resp. $k_2$) is the eigenvalue of $t_1$ (resp.\ $t_2$) on $\V_X(\mu_0)$}
\end{array}
\end{array}
\end{equation}
Under this convention, the following equation holds (see Lemma \ref{lem:typekr0}):
\begin{equation}                           \label{eq:equation}
\begin{array}{c|c}
\text{\rm $X$-type of $\V$}  & \text{\rm Equation}
\\ \hline 
\textup{\sf DS}  & k_0 k_1 k_2 k_3 = q^{-n-1}   \rule{0mm}{3ex}
\\
 \textup{\sf DDa} &   k_0^2 = q^{-n-1}   \rule{0mm}{2.5ex}
\\
 \textup{\sf DDb} &   k_3^2 = q^{-n-1}  \rule{0mm}{2.5ex}
\\
 \textup{\sf SSa} &   k_1^2 = q^{-n-1}   \rule{0mm}{2.5ex}
\\
 \textup{\sf SSb} &   k_2^2 = q^{-n-1}  \rule{0mm}{2.5ex}
\end{array}
\end{equation}

Below we summarize our proof of the main theorems.
Until starting the proof summary of Theorem \ref{thm:main},
the following notation is in effect.
Let $\V$ denote an XD $\Hq$-module with dimension $n+1$, $n \geq 1$.
We consider the reduced $X$-diagram of $\V$.
Let $\{\mu_r\}_{r=0}^n$ denote the eigenvalues of $X$ on $\V$, as shown in \eqref{eq:standard}.
Choose a parameter sequence $\{k_i\}_{i \in \I}$ of $\V$ that satisfies \eqref{eq:rule0}.
Assume that $k_0 \neq k_0^{-1}$.
Let $\V(k_0)$ denote the subspace of $\V$ consisting of $v \in \V$
such that $t_0 v = k_0 v$.
The subspace $\V(k_0^{-1})$ is similarly defined.

Towards Theorem \ref{thm:main1},
we make the following observations (A)--(C).

(A)
The eigenvalues  $\{\mu_r\}_{r=0}^n$ are determined by
$\{k_i\}_{i \in \I}$, using \eqref{eq:rule0} and the shape of the diagram.

(B)
We indicated earlier that for each single bond $\mu$, $\nu$ 
the subspace $\V_X(\mu) + \V_X(\nu)$ is invariant under $t_0$.
It turns out that the intersections of this subspace with
$\V(k_0)$ and $\V(k_0^{-1})$ each have dimension one.
Call these intersections {\em bond subspaces}.
For an endvertex $\mu$ that is incident to a double bond,
by \eqref{eq:end} $\V_X(\mu)$ is invariant under $t_0$, 
so it is contained in one of $\V(k_0)$, $\V(k_0^{-1})$.
Call $\V_X(\mu)$ a {\em bond subspace}.
Note that each of $\V(k_0)$, $\V(k_0^{-1})$ is a direct sum of its bond subspaces.
Also note that $t_0$ has only one eigenvalue on $\V$
if and only if $n=1$ and $\V$ has $X$-type \text{\sf DDa}.

(C)
For each vertex $\mu$
the subspace $\V_X(\mu)$ is an eigenspace of $X$ with eigenvalue $\mu$.
Since $\B = X + X^{-1}$,
$\V_X(\mu)$ is invariant under $\B$ with eigenvalue $\mu + \mu^{-1}$.
For each single bond $\mu$, $\nu$
we have $\mu \nu = 1$, so on $\V_X(\mu)$ and $\V_X(\nu)$
the eigenvalue of $\B$ is the same.
Therefore $\B$ acts on $\V_X(\mu) + \V_X(\nu)$ as $\mu + \nu$ times the identity.

We summarize our proof of Theorem \ref{thm:main1}.
Assume that $\V$ is feasible.
First consider the action of $\A,\B$ on $\V(k_0)$.
We pick a nonzero vector in each bond subspace of $\V(k_0)$.
By (B) these vectors form a basis for $\V(k_0)$.
We order these vectors along with the ordering $\{\mu_r\}_{r=0}^n$, and denote them 
by $\{w_r\}_{r=0}^d$.
By (C) the vectors $\{w_r\}_{r=0}^d$ are eigenvectors of $\B$.
By (A) the corresponding eigenvalues are represented in terms of $\{k_i\}_{i \in \I}$,
and we find that these eigenvalues form a $q$-Racah sequence.
We indicated earlier that for each single (resp.\ double) bond $\mu$, $\nu$
the subspace $\V_X(\mu) + \V_X(\nu)$ is invariant under $t_0$ (resp.\ $t_1$).
By this we see that the matrix representing $\A$ with respect to $\{w_r\}_{r=0}^d$
is tridiagonal, and it turns out that this tridiagonal matrix is irreducible.
We have described the action of $\A,\B$ on $\V(k_0)$,
and the action of $\A,\B$ on $\V(k_0^{-1})$ is similar.
So far for each of $\V(k_0^{\pm 1})$  there exists a basis
with respect to which the matrix representing
$\A$ is irreducible tridiagonal and the matrix representing $\B$ is diagonal
whose diagonal entries form a $q$-Racah sequence.
There is an automorphism $\sigma$ of $\Hq$ that fixes $t_0$
and swaps $\A$, $\B$.
Applying the above fact to the twisted $\Hq$-module $\V^\sigma$,
we find that for each of $\V(k_0^{\pm 1})$ there exists a basis with respect to which the matrix
representing $\B$ is irreducible tridiagonal and the matrix representing $\A$ is diagonal
whose diagonal entries form a $q$-Racah sequence.
Therefore $\A,\B$ act on each $\V(k_0^{\pm 1})$ as a Leonard pair of $q$-Racah type.

Towards Theorem \ref{thm:main}, 
we do the following (D)--(G).

(D)
We construct a certain basis $\{u_r\}_{r=0}^n$ for $\V$ such that
for $1 \leq r \leq n$ we have
$Y^e u_{r-1} \in \F (u_{r-1} - u_r)$,
where $e=1$ if $\mu_{r-1}$, $\mu_r$ are $1$-adjacent,
and $e=-1$  if $\mu_{r-1}$, $\mu_r$ are $q$-adjacent.

(E)
We obtain the action of $X^{\pm 1}$ and $Y^{\pm 1}$ on the basis $\{u_r\}_{r=0}^n$.
The results show that
with respect to the basis $\{u_r\}_{r=0}^n$ the matrices representing
$X^{\pm 1}$,  $\B$ are lower tridiagonal,
and the matrices representing   $Y^{\pm 1}$, $\A$ are upper tridiagonal.

(F)
We obtain some inequalities for $\{k_i\}_{i \in \I}$ 
from the facts that $X$ is diagonalizable on $\V$,
$\{\mu_r\}_{r=0}^n$ are mutually distinct,  and
each $\V_X(\mu_r)$ has dimension one (see Lemma \ref{lem:typekr}).
Also, using the  basis $\{u_r\}_{r=0}^n$ we 
obtain necessary and sufficient conditions on $\{k_i\}_{i \in \I}$ for $Y$ to be
diagonalizable on $\V$ (see Corollary \ref{cor:Ydiagonalizable}).

(G)
Assume that $\V$ is feasible, and 
consider the two Leonard pairs of $q$-Racah type  from Theorem \ref{thm:main1}.
We represent a Huang data of each Leonard pair in terms of $\{k_i\}_{i \in \I}$
as follows.
First consider the Leonard pair $\A,\B$ on $\V(k_0)$.
Using the basis $\{u_r\}_{r=0}^n$ from (D) we construct a basis for $\V(k_0)$ 
with respect to which the matrix representing $\A$ (resp.\ $\B$) is lower bidiagonal 
(resp.\ upper bidiagonal).
By construction the diagonal entries of this matrix give an ordering of the
eigenvalues of $\A$ (resp.\ $\B$) on $\V(k_0)$,
and it turns out (see Lemma \ref{lem:LBUB}) that this ordering
is standard.
This standard ordering forms a $q$-Racah sequence.
Let $a$ (resp.\ $b$) denote the parameter of this $q$-Racah sequence.
We represent $a$ (resp.\ $b$) in terms of $\{k_i\}_{i \in \I}$.
We remark that the above standard ordering of the eigenvalues of $\B$ coincides with the one
from the proof summary of Theorem \ref{thm:main1}
(see Proposition \ref{prop:ABaction} and Corollary \ref{cor:Beigen}).
We define a certain nonzero $c \in \F$ in terms of $\{k_i\}_{i \in \I}$.
Using the equitable Askey-Wilson relations (see Lemma \ref{lem:Z4AW}),
we show that $(a,b,c,d)$ is a Huang data of the Leonard pair $\A,\B$ on $\V(k_0)$.
So far, we have represented a Huang data of the Leonard pair $\A,\B$ 
on $\V(k_0)$ in terms of $\{k_i\}_{i \in \I}$.
We similarly represent a Huang data of the Leonard pair $\A,\B$ on $\V(k_0^{-1})$
in terms of $\{k_i\}_{i \in \I}$.

We now summarize our proof of  Theorem \ref{thm:main}.
First consider the  ``only if''  direction.
Suppose we are given two Leonard pairs $A,A^*$ on $V$ and
$A',A^{*\prime}$ on $V'$ that have $q$-Racah type.
Assume that these Leonard pairs are linked.
So we have a feasible $\Hq$-module structure on $\V := V \oplus V'$ such that
$V$, $V'$ are the eigenspaces of $t_0$ and \eqref{eq:linked} holds.
Let $\{k_i\}_{i \in \I}$ denote a parameter sequence of $\V$ that
satisfies \eqref{eq:rule0}.
First assume that $V=\V(k_0)$ and $V' = \V(k_0^{-1})$.
In (G) we described a Huang data $(a,b,c,d)$ of $\A,\B$ on $\V(k_0)$ and a Huang data
$(a',b',c',d')$ of  $\A,\B$ on $\V(k_0^{-1})$.
These Huang data are displayed in Proposition \ref{prop:Huangdata}.
For these Huang data the value of $d'-d$ and the ratios   $a'/a$, $b'/b$, $c'/c$ are
as follows.
\[
\begin{array}{c|c|c|c|c}
\text{$X$-type of $\V$} & d'-d & a'/a & b'/b & c'/c
\\ \hline
\textup{\sf DS} & -1 & q & q & q
\\
\textup{\sf DDa} & -2 & 1 & 1 & 1
\\
\textup{\sf DDb} & 0 & 1 & q^2 & 1
\\
\textup{\sf SSa} & 0 & q^2 & 1 & 1
\\
\textup{\sf SSb} & 0 & 1 & 1 & q^2
\end{array}
\]
Using (F), \eqref{eq:equation} and $k_0^2 \neq 1$ we get the following inequalities:
\[
\begin{array}{c|ccc}
\text{\rm $X$-type of $\V$} & & \text{\rm inequalities} 
\\ \hline
\textup{\sf DS} & a^2 \neq q^{-2d} & b^2 \neq q^{-2d}   \rule{0mm}{2.8ex}
\\
\textup{\sf DDa}
\\
\textup{\sf DDb} & a^2 \neq q^{\pm 2d} & & b^2 \neq q^{-2}
\\
\textup{\sf SSa} & & b^2 \neq q^{\pm 2d} & a^2 \neq q^{-2}
\\
\textup{\sf SSb} & a^2 \neq q^{\pm 2d} & b^2 \neq q^{\pm 2d} & c^2 \neq q^{-2}
\end{array}
\]
Now we find that  the following case in Theorem \ref{thm:main} occurs:
\[
\begin{array}{c|ccccc}
 \textup{\rm $X$-type of $\V$} & 
 \textup{\sf DS} & \textup{\sf DDa} & \textup{\sf DDb} & \textup{\sf SSa} &  \textup{\sf SSb}
\\ \hline
 \text{Case} & \text{(ii)} & \text{(i)} & \text{(iv)} & \text{(iii)} & \text{(v)}    \rule{0mm}{2.5ex}
\end{array}
\]
For the case $V=\V(k_0^{-1})$ and $V' = \V(k_0)$ the argument is similar.

Next consider the ``if'' direction.
Suppose we are given two Leonard pairs $A,A^*$ on $V$ and $A',A^{*\prime}$ on $V'$
that have $q$-Racah type.
Assume that there exist a Huang data $(a,b,c,d)$ of $A,A^*$ and a Huang data $(a',b',c',d')$ of $A',A^{*\prime}$
that satisfy one of the conditions (i)--(vii).
We may assume that the Huang data satisfy one of (i)--(v) by exchanging our two Leonard pairs
if necessary.
For each case (i)--(v),
we define scalars $\{k_i\}_{i \in \I}$ as follows:
\begin{equation}                   \label{eq:defki0}
\begin{array}{c|ccccc}
\text{Case} & k_0 & k_1 & k_2 & k_3 
\\ \hline
\rule{0mm}{4.5mm}
\text{\rm (i)}  & q^{-d} &  a & c & b
\\
\rule{0mm}{4mm}
\text{\rm (ii)}  
 & (abcq^{1-d})^{1/2} & aq^{-d}k_0^{-1} & cq^{-d}k_0^{-1} & bq^{-d}k_0^{-1}
\\
\rule{0mm}{4mm}
\text{\rm (iii)}  & aq & q^{-d-1} & b & c
\\
\rule{0mm}{4mm}
\text{\rm (iv)}  &  bq & c & a & q^{-d-1}
\\
\text{\rm (v)}  & cq & b & q^{-d-1} & a
\end{array}
\end{equation}
In case (ii), take any one of the square roots of $a b c q^{1-d}$
for the value of $k_0$.
We construct an XD $\Hq$-module $\V$ with dimension $d+d'+1$
and parameter sequence $\{k_i\}_{i \in \I}$ in the following way.
Set $n=d+d'+1$ and let $\V$ denote a vector space over $\F$ with basis $\{v_r\}_{r=0}^n$.
We define scalars $\{\mu_r\}_{r=0}^n$ as follows:
\[
\begin{array}{c|c}
\text{Case} & \text{Definition of $\mu_r$} 
\\ \hline
\text{(i), (ii), (iv)} 
& \mu_r = k_0 k_3 q^r \;\; \text{ if $r$ is even},  
\qquad \mu_r = \big( k_0 k_3 q^{r+1} \big)^{-1} \;\;  \text{ if $r$ is odd}
\\
\text{(iii), (v)}
& \mu_r = \big( k_1 k_2 q^{r+1} \big)^{-1} \;\; \text{ if $r$ is even},
\qquad \mu_r = k_1 k_2 q^r \;\; \text{ if $r$ is odd}
\end{array}
\]
We find that the reduced diagram of $\{\mu_r\}_{r=0}^n$ is 
as in  \eqref{eq:standard} with
\[
\begin{array}{ccccc}
 \text{(i)} & \text{(ii)} & \text{(iii)} & \text{(iv)} & \text{(v)}
\\ \hline
 \textup{\sf DD} & \textup{\sf DS} & \textup{\sf SS} & \textup{\sf DD} & \textup{\sf SS}
\end{array}
\]
Using this diagram
we define the action of $\{t_i\}_{i \in \I}$ on $\{v_r\}_{r=0}^n$ as follows.
For $0 \leq r \leq n-1$ such that $\mu_r$, $\mu_{r+1}$ are $1$-adjacent
(resp.\ $q$-adjacent),
we define the action of $t_0$, $t_3$ (resp.\ $t_1$, $t_2$)
on $\F v_r + \F v_{r+1}$ by \eqref{eq:t0vr}--\eqref{eq:t3vr+1}
(resp.\eqref{eq:t1vr}--\eqref{eq:t2vr+1}).
The remaining actions are defined by \eqref{eq:tiv0vn}.
It turns out that these actions give an $\Hq$-module structure of $\V$.
Moreover,
this $\Hq$-module $\V$ is XD and has $X$-type
\begin{equation}           \label{eq:corres2}
\begin{array}{ccccc}
 \text{(i)} & \text{(ii)} & \text{(iii)} & \text{(iv)} & \text{(v)}
\\ \hline
 \textup{\sf DDa} & \textup{\sf DS} & \textup{\sf SSa} & \textup{\sf DDb} & \textup{\sf SSb}
\end{array}
\end{equation}
The inequalities in Theorem \ref{thm:main} yield
some inequalities for $\{k_i\}_{i \in \I}$ (see Lemma \ref{lem:restrictions}).
By these inequalities we find that $\{k_i\}_{i \in \I}$ satisfy the conditions in (F)
that make $Y$ diagonalizable on $\V$,
so $\V$ is YD.
We also find that $t_0$ has two distinct eigenvalues on $\V$, so $\V$ is feasible.
By Theorem \ref{thm:main1} $\A,\B$ act on each eigenspace of $t_0$
as a Leonard pair of $q$-Racah type.
These Leonard pairs have Huang data that are displayed in Proposition \ref{prop:Huangdata}.
Using \eqref{eq:defki0} we find that these Huang data coincide with $(a,b,c,d)$ and $(a',b',c',d')$.
Thus the Leonard pair $\A,\B$ on $\V(k_0)$ (resp.\ $\V(k_0^{-1})$)
and the Leonard pair $A,A^*$ (resp.\ $A', A^{* \prime}$)
have a common Huang data.
So the Leonard pair $\A,\B$ on $\V(k_0)$ (resp.\ $\V(k_0^{-1})$)
and the Leonard pair $A,A^*$ on $V$ (resp.\ $A', A^{* \prime}$ on $V'$) are isomorphic.
Let $f : \V(k_0) \to V$ (resp.\ $f' : \V(k_0^{-1}) \to V'$)
denote the corresponding isomorphism of Leonard pairs.
Consider the $\F$-linear bijection $f \oplus f' : \V = \V(k_0) + \V(k_0^{-1}) \to V \oplus V'$. 
We define an $\Hq$-module structure on $V \oplus V'$ so that 
$f \oplus f'$ is an isomorphism of $\Hq$-modules.
By construction the $\Hq$-module $V \oplus V'$ is feasible,
the subspaces $V$ and $V'$ are the eigenspaces of $t_0$,  and \eqref{eq:linked} holds.
Thus the Leonard pairs $A,A^*$ and $A',A^{*\prime}$ are linked.

The paper is organized as follows.
In Section \ref{sec:pre} we recall some materials concerning Leonard pairs.
In Section \ref{sec:auto} we recall some basic facts about $\Hq$.
In Sections \ref{sec:G0G2} and \ref{sec:G0G2action} we consider certain elements of $\Hq$ 
and study their properties.
In Section \ref{sec:diagram} we study the $X$-diagram.
In Section \ref{sec:mfree} we show that $X$ is multiplicity-free on an XD $\Hq$-module.
In Section \ref{sec:types} we study the cases {\sf DS}, {\sf DDa}, {\sf DDb}, {\sf SSa}, {\sf SSb},
and we do (A).
In Section \ref{sec:structure} we investigate the structure of an XD $\Hq$-module.
In Section \ref{sec:eigent0} we investigate the eigenspaces of $t_0$,
and we do (B).
In Section \ref{sec:proofmain1} we prove Theorem \ref{thm:main1}.
In Section \ref{sec:restrictions} we obtain some (in)equalities for $\{k_i\}_{i \in \I}$.
In Sections \ref{sec:ur} we do (D).
In Section \ref{sec:eigenY} we obtain the action of $Y^{\pm 1}$ on the basis $\{u_r\}_{r=0}^n$,
and we do (F).
In Section \ref{sec:tiaction} we obtain the action of $\{t_i\}_{i \in \I}$ on the basis $\{u_r\}_{r=0}^n$.
In Section \ref{sec:Xaction} we obtain the action of $X^{\pm 1}$  on the basis $\{u_r\}_{r=0}^n$.
In Sections \ref{sec:basisVpm}--\ref{sec:Huang} we do (G).
In Section \ref{sec:onlyif} we prove the ``only if'' direction of Theorem \ref{thm:main}.
In Sections \ref{sec:const} and \ref{sec:if} we prove the ``if'' direction
of Theorem \ref{thm:main}.

\section{Preliminaries for Leonard pairs}
\label{sec:pre}

In this section we recall some materials concerning Leonard pairs.
We first recall the notion of a parameter array,
and next recall the notion of a Huang data.

Fix an integer $d \geq 0$.
Let $M$ denote a $(d+1) \times (d+1)$ matrix with all entries in $\F$.
We index the rows and columns by  $0,1,\ldots,d$.
Let $V$ denote a vector space over $\F$ with basis $\{v_r\}_{r=0}^d$, 
and consider a linear transformation $A: V \to V$.
We say {\em $M$ represents $A$ with respect to $\{v_r\}_{r=0}^d$}
whenever $A v_s = \sum_{r=0}^d M_{r,s} v_r$ for $0 \leq s \leq d$.

\begin{lemma} {\rm (See \cite[Theorem 3.2]{T:Leonard}.)}     \label{lem:split}   \samepage
\ifDRAFT {\rm lem:split}. \fi
Let $V$ denote a vector space over $\F$ with dimension $d+1$,
and let $A,A^*$ denote a Leonard pair on $V$.
Let $\{\th_r\}_{r=0}^d$ (resp.\ $\{\th^*_r\}_{r=0}^d$) 
denote a standard ordering of the eigenvalues of $A$ (resp.\ $A^*$).
Then there exists a basis for $V$ with respect to which
the matrices representing $A,A^*$ are
\begin{align*} 
A &: \;\; 
 \begin{pmatrix}
  \th_0  & & & & & \text{\bf 0} \\
  1 & \th_1  \\
   & 1 & \th_2 \\
   & & \cdot & \cdot \\
   & & & \cdot & \cdot \\
  \text{\bf 0}  & & & & 1 & \th_{d}
 \end{pmatrix},
& 
A^*  &: \;\;
 \begin{pmatrix}
  \th^*_0 & \vphi_1 & & & & \text{\bf 0} \\
     & \th^*_1 & \vphi_2  \\
   &  & \th^*_2 & \cdot  \\
   & &    & \cdot & \cdot \\
   & & &    & \cdot & \vphi_d \\
  \text{\bf 0} & & & &  & \th^*_d
 \end{pmatrix}
\end{align*}
for some scalars $\{\vphi_r\}_{r=1}^d$ in $\F$.
The sequence $\{\vphi_r\}_{r=1}^d$ is uniquely determined by the ordering
$(\{\th_r\}_{r=0}^d, \{\th^*_r\}_{r=0}^d)$.
Moreover $\vphi_r \neq 0$ for $1 \leq r \leq d$.
\end{lemma}

With reference to Lemma \ref{lem:split}
we refer to $\{\vphi_r\}_{r=1}^d$ as the {\em first split sequence} of $A,A^*$
associated with the ordering $(\{\th_r\}_{r=0}^d, \{\th^*_r\}_{r=0}^d)$.
By the {\em second split sequence} of $A,A^*$ associated with
the ordering $(\{\th_r\}_{r=0}^d, \{\th^*_r\}_{r=0}^d)$ we mean the first split sequence
of $A,A^*$ associate with the ordering
 $(\{\th_{d-r}\}_{r=0}^d, \{\th^*_r\}_{r=0}^d)$.

\begin{definition}    \label{def:parray}    \samepage
\ifDRAFT {\rm def:parray}. \fi
Let $A,A^*$ denote a Leonard pair over $\F$.
By a {\em parameter array} of $A,A^*$ we mean a sequence
\begin{equation}   \label{eq:parray}
  (\{\th_r\}_{r=0}^d, \{\th^*_r\}_{r=0}^d, \{\vphi_r\}_{r=1}^d, \{\phi_r\}_{r=1}^d),
\end{equation}
such that $\{\th_r\}_{r=0}^d$ is a standard ordering of the eigenvalues of $A$,
$\{\th^*_r\}_{r=0}^d$ is a standard ordering of the eigenvalues of $A^*$,
and $\{\vphi_r\}_{r=1}^d$ (resp.\ $\{\phi_r\}_{r=1}^d$) is the corresponding
first split sequence (resp.\ second split sequence) of $A,A^*$.
\end{definition}

Let $A,A^*$ denote a Leonard pair over $\F$ with diameter $d$,
and let \eqref{eq:parray} denote a parameter array of $A,A^*$.
Then by \cite[Theorem 1.11]{T:Leonard} the parameter arrays of $A,A^*$ are
as follows:
\begin{align*}
 & (\{\th_{r}\}_{r=0}^d,  \{\th^*_{r}\}_{r=0}^d,  \{\vphi_{r}\}_{r=1}^d,  \{\phi_{r}\}_{r=1}^d), 
\\
 & (\{\th_{r}\}_{r=0}^d,  \{\th^*_{d-r}\}_{r=0}^d,  \{\phi_{d-r+1}\}_{r=1}^d,  \{\vphi_{d-r+1}\}_{r=1}^d),
\\
 &  (\{\th_{d-r}\}_{r=0}^d,  \{\th^*_{r}\}_{r=0}^d,  \{\phi_{r}\}_{r=1}^d,  \{\vphi_{r}\}_{r=1}^d),
\\
 & (\{\th_{d-r}\}_{r=0}^d, \{\th^*_{d-r}\}_{r=0}^d,  \{\vphi_{d-r+1}\}_{r=1}^d,  \{\phi_{d-r+1}\}_{r=1}^d).
\end{align*}

\begin{lemma}  {\rm (See \cite[Theorem 1.9]{T:Leonard}.)}
\label{lem:characterize}   \samepage
\ifDRAFT {\rm lem:characterize}. \fi
Consider two Leonard pairs over $\F$.
Assume that these Leonard pairs have a common parameter array.
Then these Leonard pairs are isomorphic.
\end{lemma}

A square matrix is said to be {\em upper bidiagonal}
whenever each nonzero entry lies on the diagonal or the superdiagonal,
and {\em lower bidiagonal} whenever its transpose is upper bidiagonal.

\begin{lemma}   {\rm (See  \cite[Corollary 7.6]{T:LBUB}.)}     \samepage
\label{lem:LBUB} 
\ifDRAFT {\rm lem:LBUB}. \fi
Let $V$ denote a vector space over $\F$ with dimension $d+1$,
and let $A,A^*$ denote a Leonard pair on $V$.
Assume that there exists a basis for $V$ with respect to which the matrix
representing $A$ (resp.\ $A^*$) is lower bidiagonal (resp.\ upper bidiagonal)
with $(r,r)$-entry $\th_r$ (resp.\ $\th^*_r$) for $0 \leq r \leq d$.
Then $\{\th_r\}_{r=0}^d$ (resp.\ $\{\th^*_r\}_{r=0}^d$) is a standard ordering
of the eigenvalues of $A$ (resp.\ $A^*$).
\end{lemma}

\begin{lemma}  {\rm (See \cite[Corollary 6.4]{Hu}.)}  \label{lem:qRacah} \samepage
\ifDRAFT {\rm lem:qRacah}. \fi
Let $A,A^*$ denote a Leonard pair over $\F$ with diameter $d$, 
and let \eqref{eq:parray} denote a parameter array of $A,A^*$.
Assume that $A,A^*$ has $q$-Racah type, and
let $a$, $b$ denote nonzero scalars in $\F$ such that 
\begin{align}    \label{eq:thrthsr}
 \th_r  &= a q^{2r-d} + a^{-1} q^{d-2r},  &
 \th^*_r &= b q^{2r-d} + b^{-1}q^{d-2r} && (0 \leq r \leq d). 
\end{align}
Then there exists a nonzero $c \in \F$ such that 
{\small
\begin{align}
  \vphi_r &= a^{-1}b^{-1} q^{d+1}(q^r-q^{-r})(q^{r-d-1}-q^{d-r+1})
                        (q^{-r}-a b c q^{r-d-1})(q^{-r}-a b c^{-1} q^{r-d-1}),    \label{eq:vphir}
\\
  \phi_ r &= a b^{-1} q^{d+1}(q^r-q^{-r})(q^{r-d-1}-q^{d-r+1})
               (q^{-r}-a^{-1} b c q^{r-d-1})(q^{-r}-a^{-1} b c^{-1} q^{r-d-1})   \label{eq:phir}
\end{align}
}
for $1 \leq r \leq d$.
Moreover $c$ is unique up to inverse, provided that $d \geq 1$.
\end{lemma}

With reference to Lemma \ref{lem:qRacah},
for the case $d=0$
the conditions \eqref{eq:vphir} and \eqref{eq:phir} become vacuous,
so any nonzero $c \in \F$ satisfies these conditions.
For the case $d \geq 1$,
the scalars $a$, $b$, $c$ are determined up to inverse of $c$ by
the ordering $(\{\th_r\}_{r=0}^d, \{\th^*_r\}_{r=0}^d)$,
since the parameter array \eqref{eq:parray} is determined by 
the ordering $(\{\th_r\}_{r=0}^d, \{\th^*_r\}_{r=0}^d)$.

\begin{definition}   \label{def:Huangdata}   \samepage
\ifDRAFT {\rm def:Huangdata}. \fi
Let $A,A^*$ denote a Leonard pair over $\F$ with diameter $d$ that has $q$-Racah type,
and let \eqref{eq:parray} denote a parameter array of $A,A^*$.
Let $a$, $b$, $c$ denote nonzero scalars in $\F$ that satisfy
\eqref{eq:thrthsr}--\eqref{eq:phir}.
Then the sequence $(a,b,c,d)$ is called a {\em Huang data} of $A,A^*$ corresponding
to the ordering $(\{\th_r\}_{r=0}^d, \{\th^*_r\}_{r=0}^d)$.
\end{definition}

With reference to Definition \ref{def:Huangdata}, assume $d \geq 1$.
Let $(a,b,c,d)$ denote a Huang data of $A,A^*$ corresponding to the ordering
$(\{\th_r\}_{r=0}^d, \{\th^*_r\}_{r=0}^d)$.
Then by \cite[Lemma 8.1]{Hu} for each standard ordering of the eigenvalues of $A,A^*$
the corresponding Huang data of $A,A^*$ are as follows:
\[
 \begin{array}{c|c}
   \text{\rm Standard ordering} & \text{\rm Huang data}
\\ \hline
  (\{\th_r\}_{r=0}^d, \{\th^*_r\}_{r=0}^d) & (a,b,c^{\pm 1} ,d)           \rule{0mm}{2.8ex}
\\
    (\{\th_r\}_{r=0}^d, \{\th^*_{d-r}\}_{r=0}^d) & (a,b^{-1},c^{\pm 1} ,d)          \rule{0mm}{2.8ex}
\\
  (\{\th_{d-r}\}_{r=0}^d, \{\th^*_r\}_{r=0}^d) & (a^{-1} , b, c^{\pm 1} ,d)          \rule{0mm}{2.8ex}
\\
  (\{\th_{d-r}\}_{r=0}^d, \{\th^*_{d-r}\}_{r=0}^d) & (a^{-1} , b^{-1}, c^{\pm 1} ,d)          \rule{0mm}{2.8ex}
 \end{array}
\]

\begin{lemma}   \label{lem:qRacahunique}
\ifDRAFT {\rm lem:qRacahunique}. \fi
Consider two Leonard pairs over $\F$ that have $q$-Racah type.
Assume that these Leonard pairs have a common Huang data.
Then these Leonard pairs are isomorphic.
\end{lemma}

\begin{proof}
Let $(a,b,c,d)$ denote a common Huang data of the two Leonard pairs.
By \eqref{eq:thrthsr}--\eqref{eq:phir} the Huang data $(a,b,c,d)$ determines a parameter array
of each Leonard pair.
Thus these Leonard pairs have a common parameter array.
By this and Lemma \ref{lem:characterize} these Leonard pairs are isomorphic.
\end{proof}

We mention a lemma for later use.

\begin{lemma}  {\rm (See \cite[Lemmas 7.2, 7.3]{Hu}.)} \label{lem:condabc} \samepage
\ifDRAFT {\rm lem:condabc}. \fi
Let $a,b,c$ denote nonzero scalars in $\F$.
Then the sequence $(a,b,c,d)$ is a Huang data of a Leonard pair over $\F$ of
$q$-Racah type if and only if 
the following {\rm (i)} and {\rm (ii)} hold:
\begin{itemize}
\item[\rm (i)]
neither of $a^2$, $b^2$ is among
$q^{2d-2},q^{2d-4},\ldots,q^{2-2d}$;
\item[\rm (ii)]
none of $abc$, $a^{-1}bc$, $ab^{-1}c$, $abc^{-1}$ is among
$q^{d-1},q^{d-3},\ldots,q^{1-d}$.
\end{itemize}
\end{lemma}

\section{Basic facts about $\Hq$}
\label{sec:auto}

In this section we collect some basic facts about $\Hq$.
For more background information we refer the reader to \cite{T:AWDAHA}.
The following two lemmas are immediate from Definition \ref{def:Hq}.

\begin{lemma}   \label{lem:Z4pre}  \samepage
\ifDRAFT {\rm lem:Z4pre}. \fi
In the algebra $\Hq$ the scalar $q^{-1}$ is equal to each of
\begin{align*}
 & t_0t_1t_2t_3, && t_1t_2t_3t_0, && t_2t_3t_0t_1, && t_3t_0t_1t_2.
\end{align*}
\end{lemma}

\begin{lemma}  \label{lem:Z4} \samepage
\ifDRAFT {\rm lem:Z4}. \fi
There exists an automorphism $\varrho$ of $\Hq$ that sends
\[
  t_0 \mapsto t_1 \mapsto t_2 \mapsto t_3 \mapsto t_0.
\]
\end{lemma}

\begin{lemma}   \label{lem:sigma} \samepage
\ifDRAFT {\rm lem:sigma}. \fi
There exists an automorphism $\sigma$ of $\Hq$ that sends
\begin{align*}
 t_0 &\mapsto t_0, &
 t_1 &\mapsto t_0^{-1}t_3t_0, &
 t_2 &\mapsto t_1t_2t_1^{-1}, &
 t_3 &\mapsto t_1.
\end{align*}
\end{lemma}

\begin{proof}
This is routinely checked, or see \cite[Lemma 4.2]{IT:DAHA}.
\end{proof}

The following lemmas are routinely checked.

\begin{lemma}  \label{lem:Z4XY}  \samepage
\ifDRAFT {\rm lem:Z4XY}. \fi
The automorphism $\varrho$ from Lemma {\rm \ref{lem:Z4}} sends
\[
  X \mapsto Y \mapsto q^{-1}X^{-1} \mapsto q^{-1}Y^{-1} \mapsto X.
\]
\end{lemma}

\begin{lemma}  \label{lem:sigmaXY}  \samepage
\ifDRAFT {\rm lem:sigmaXY}. \fi
The automorphism $\sigma$ from Lemma {\rm \ref{lem:sigma}} sends
\begin{align*}
  X &\mapsto t_0^{-1} Y t_0,   &
  Y &\mapsto X,  &
  \A &\mapsto \B, &
  \B &\mapsto \A,
\\
 T_0 &\mapsto T_0, &
 T_1 &\mapsto T_3, &
 T_2 &\mapsto T_2, &
 T_3 &\mapsto T_1.
\end{align*}
\end{lemma}

We collect some identities in $\Hq$.

\begin{lemma}  {\rm (See \cite[Corollary 6.2]{T:AWDAHA}.)}    \label{lem:titj}  \samepage
\ifDRAFT {\rm lem:titj}. \fi
For $i,j \in \I$ the following hold:
\begin{itemize}
\item[\rm (i)]
$t_it_j + (t_it_j)^{-1} = t_jt_i + (t_jt_i)^{-1}$;
\item[\rm (ii)]
$t_it_j + (t_it_j)^{-1}$ commutes with each of $t_i$, $t_j$.
\end{itemize}
\end{lemma}

\begin{lemma}   \label{lem:Xt0}   \samepage
\ifDRAFT {\rm lem:Xt0}. \fi
We have
\begin{align}
  X t_0 - t_0 X^{-1} &= X T_0 - T_3,        \label{eq:Xt0}  \\
  q X t_2 - q^{-1} t_2 X^{-1} 
        &=  T_1 - q^{-1} X^{-1} T_2.   \label{eq:Xt2}
\end{align}
\end{lemma}

\begin{proof}
To verify \eqref{eq:Xt0}, eliminate $T_0$ (resp.\ $T_3$) using $T_0=t_0+t_0^{-1}$ (resp.\ $T_3=t_3+t_3^{-1}$)
and simplify the result using  $t_3 = X t_0^{-1}$.
To obtain \eqref{eq:Xt2}, apply $\varrho^2$ to each side of \eqref{eq:Xt0} and use Lemma \ref{lem:Z4XY} to find
\[
   q^{-1} X^{-1} t_2 - q t_2 X = q^{-1} X^{-1} T_2 - T_1.
\]
In this equation, multiply each term on the left by $X$ and on the right by $X^{-1}$.
This yields \eqref{eq:Xt2}.
\end{proof}

\section{The elements $G_i$}
\label{sec:G0G2}

We consider the following elements of $\Hq$ that will help us investigate $\Hq$-modules.

\begin{definition}        \label{def:G0G2}
\ifDRAFT {\rm def:G0G2}. \fi
Define
\begin{align*} 
 G_0 &= t_0 - t_3 t_0 t_3^{-1},   &
 G_1 &= t_1 - t_0 t_1 t_0^{-1},   \\
 G_2 &= t_2 - t_1 t_2 t_1^{-1},   &
 G_3 &= t_3 - t_2 t_3 t_2^{-1}.
\end{align*}
\end{definition}

In this and the next section, we describe the $G_i$.
We focus on how $G_0$, $G_2$ interact with $X$;
similar results describe how $G_1$, $G_3$ interact with $Y$.
Our results are summarized as follows.
Let $\V$ denote an XD $\Hq$-module with parameter sequence $\{k_i\}_{i \in \I}$.
We show that each of $G_0^2$, $G_2^2$ acts on $\V$ as a Laurent polynomial in $X$.
In particular $G_0^2$, $G_2^2$ act on each $X$-eigenspace in $\V$ as a scalar multiple
of the identity.
We compute these scalars in terms of $\{k_i\}_{i \in \I}$.

\begin{lemma} \label{lem:R4} \samepage
\ifDRAFT {\rm lem:R4}. \fi
We have
\begin{align*}
 X G_0 &= G_0 X^{-1}, &
 X^{-1}G_0 &= G_0 X,  \\
 X G_2 &= q^{-2} G_2 X^{-1}, &
 X^{-1}G_2 &= q^2 G_2 X.
\end{align*}
\end{lemma}

\begin{proof}
Using $X=t_3t_0$ and $G_0=t_0-t_3t_0t_3^{-1}$ one verifies
\[
  G_0 t_3 - X^{-1}G_0X^{-1} t_3 = t_0t_3 + (t_0t_3)^{-1} - t_3t_0 - (t_3t_0)^{-1}.
\]
In this equation, the right-hand side is zero by Lemma \ref{lem:titj}(i).
So $G_0t_3 = X^{-1}G_0X^{-1}t_3$.
In this equation, multiply each side on the left by X and on the right by $t_3^{-1}$
to get $X G_0=G_0 X^{-1}$.
In this equation, multiply each side on the left by $X^{-1}$ and on the right by $X$
to get $X^{-1}G_0=G_0X$.
In this equation, apply $\varrho^2$ to each side and simplify the result using Lemma \ref{lem:Z4XY}
to get $X G_2=q^{-2} G_2 X^{-1}$.
In this equation, multiply each side on the left by $X^{-1}$
and on the right by $X$ to get $X^{-1}G_2 = q^2 G_2 X$.
\end{proof}

\begin{lemma} \label{lem:R11} \samepage
\ifDRAFT {\rm lem:R11}. \fi
We have
\begin{align}
 G_0 &= t_0(1-X^{-2}) + T_3 X^{-1} - T_0,     \label{eq:R11no1} \\
 G_0 &= t_3(X^{-1}-X) + T_3X - T_0,           \label{eq:R11no2} \\
 G_2 &= t_2(1-q^2X^2) + q T_1X-T_2,           \label{eq:R11no3} \\
 G_2 &= t_1(qX-q^{-1}X^{-1}) + q^{-1} T_1X^{-1} - T_2.
                                              \label{eq:R11no4}
\end{align}
\end{lemma}

\begin{proof}
In \eqref{eq:Xt0},
multiply each side on the right by $X^{-1}$
and use $Xt_0X^{-1}=t_0-G_0$ to get \eqref{eq:R11no1}.
In \eqref{eq:R11no1},
eliminate $t_0$ using $t_0 = T_3 X - t_3 X$ to get \eqref{eq:R11no2}.
Concerning  \eqref{eq:R11no3},  apply $\varrho^2$ to
\eqref{eq:R11no1} and use Lemma \ref{lem:Z4XY}.
The line \eqref{eq:R11no4} is similarly obtained from \eqref{eq:R11no2}.
\end{proof}

\begin{proposition} \label{prop:R7} \samepage
\ifDRAFT {\rm prop:R7}. \fi
We have
\begin{align}
 G_0^2 
 &= (X+X^{-1})^2 - (X+X^{-1})T_0T_3 + T_0^2+T_3^2 - 4, \label{eq:G02} \\
 G_2^2 
 &= (qX+q^{-1}X^{-1})^2 - (qX+q^{-1}X^{-1})T_1T_2+T_1^2+T_2^2-4.   \label{eq:G22}
\end{align}
\end{proposition}

\begin{proof}
We first show \eqref{eq:G02}.
We claim
\begin{equation}   \label{eq:R7aux1}
 X^{-1}-X = t_0 X^{-1}G_0 + (T_3 X^{-1}-T_0) t_3.
\end{equation}
In \eqref{eq:R7aux1},
represent each side  in terms of $t_0$, $t_3$, $T_0$, $T_3$
using $X=t_3t_0$, $X^{-1}=t_0^{-1}t_3^{-1}$, $G_0=t_0-t_3t_0t_3^{-1}$,
$t_0^{-1}=T_0-t_0$, $t_3^{-1}=T_3-t_3$, and simplify the results using the fact that $T_0$, $T_3$
are central.
This gives \eqref{eq:R7aux1}.
In \eqref{eq:R7aux1}, multiply each side on the right by $X^{-1}-X$, and simplify the results
using Lemma \ref{lem:R4} to find
\begin{equation}     \label{eq:R7aux2}
 (X^{-1}-X)^2 = t_0(1-X^{-2})G_0 + (T_3 X^{-1}-T_0)t_3(X^{-1}-X).
\end{equation}
By \eqref{eq:R11no1}
$t_0(1-X^{-2}) = G_0 - (T_3 X^{-1}-T_0)$.
By \eqref{eq:R11no2}
$t_3(X^{-1}-X) = G_0-(T_3X-T_0)$.
In \eqref{eq:R7aux2} eliminate $t_0(1-X^{-2})$ and $t_3(X^{-1}-X)$ using these comments,
and simplify the result to find
\[
 (X^{-1}-X)^2 = G_0^2 - T_0^2 - T_3^2 + (X+X^{-1})T_0T_3.
\]
This implies \eqref{eq:G02}.
Concerning \eqref{eq:G22}, apply $\varrho^2$ to each side of \eqref{eq:G02} and use
Lemma \ref{lem:Z4XY}.
\end{proof}

\section{The action of $G_i$ on the eigenspaces of $X$}
\label{sec:G0G2action}

Let $\V$ denote an XD $\Hq$-module.
In this section we describe how the elements $\{t_i\}_{i \in \I}$,
$G_0$, $G_2$ act on the eigenspaces of $X$.

\begin{lemma} \label{lem:R5} \samepage
\ifDRAFT {\rm lem:R5}. \fi
For $0 \neq \mu \in \F$
\begin{align*}
 G_0 \V_X(\mu) & \subseteq \V_X(\mu^{-1}), &
 G_2 \V_X(\mu) & \subseteq \V_X(q^{-2} \mu^{-1}).
\end{align*}
\end{lemma}

\begin{proof}
We first show $G_0 \V_X(\mu) \subseteq \V_X(\mu^{-1})$.
Pick any $v \in \V_X(\mu)$.
Then $X^{-1} v = \mu^{-1} v$.
By Lemma \ref{lem:R4} $X G_0 = G_0 X^{-1}$.
Using these comments we argue
$X G_0 v = G_0 X^{-1}v = \mu^{-1} G_0 v$.
So $G_0 v \in \V_X(\mu^{-1})$.
We have shown  $G_0 \V_X(\mu) \subseteq \V_X(\mu^{-1})$.
The proof of $G_2 \V_X(\mu)  \subseteq \V_X(q^{-2} \mu^{-1})$ is similar.
\end{proof}

For indeterminates $\lambda,s,t$ define
\begin{equation}                  \label{eq:defG}
 G(\lambda,s,t) =
\lambda^{-2}
   (\lambda-s t) (\lambda- st^{-1}) (\lambda-s^{-1}t) (\lambda-s^{-1}t^{-1}).
\end{equation}
The following lemma is routinely verified.

\begin{lemma}    \label{lem:G}
\ifDRAFT {\rm lem:G}. \fi
We have
\[
 G(\lambda,s,t) =
 (\lambda + \lambda^{-1})^2  - (\lambda +\lambda^{-1})(s+s^{-1})(t+t^{-1})
   + (s+s^{-1})^2 + (t+t^{-1})^2 - 4.
\]
Moreover $G(\lambda,s,t)$ is equal to each of the following:
\begin{align*}
 & G(\lambda^{-1},s,t),  &&  G(\lambda,s^{-1},t),  && G(\lambda,s,t^{-1}).
\end{align*}
\end{lemma}

The action of $G_0^2$, $G_2^2$ on $\V$ is given as follows.
Let $\{k_i\}_{i \in \I}$ denote a parameter sequence of $\V$.

\begin{lemma} \label{lem:R9pre}
\ifDRAFT {\rm lem:R9pre}. \fi
The following hold on $\V$:
\begin{align}
 G_0^2 &= G(X,k_0,k_3),  & 
 G_2^2 &= G(q X,k_1,k_2).                           \label{eq:R9pre}
\end{align}
\end{lemma}

\begin{proof}
We first show the equation on the left in \eqref{eq:R9pre}.
By Definition \ref{def:param} $T_0 = k_0 + k_0^{-1}$ and $T_3 = k_3 + k_3^{-1}$ on $\V$.
By this and \eqref{eq:G02}
\[
 G_0^2 = (X+X^{-1})^2 - (X+X^{-1}) (k_0 + k_0^{-1})(k_3 + k_3^{-1})
             + (k_0 + k_0^{-1})^2 + (k_3 + k_3^2)^2 - 4
\]
on $\V$.
In the above line, the the right-hand side is equal to
$G(X,k_0,k_3)$ by Lemma \ref{lem:G}.
Thus the equation on the left in \eqref{eq:R9pre} holds on $\V$.
The proof of the equation on the right in \eqref{eq:R9pre} is similar.
\end{proof}

The following corollary is immediate from Lemma \ref{lem:R9pre}.

\begin{corollary} \label{cor:R9} \samepage
\ifDRAFT {\rm cor:R9}. \fi
Let $0 \neq \mu \in \F$.
\begin{itemize}
\item[\rm (i)]
$G_0^2$ acts on $\V_X(\mu)$ as $G(\mu,k_0,k_3)$ times the identity.
\item[\rm (ii)]
$G_2^2$ acts on $\V_X(\mu)$ as $G(q \mu,k_1,k_2)$
times the identity.
\end{itemize}
\end{corollary}

We mention some lemmas for later use.

\begin{lemma}   \label{lem:t0t3action}   \samepage
\ifDRAFT {\rm lem:t0t3action}. \fi
Let $\mu$ denote a vertex of the $X$-diagram of $\V$ that is
not incident to a single loop.
Then $\mu \neq \mu^{-1}$ and the following hold on $\V_X(\mu)$:
\begin{align}
  t_0 &=  \frac{\mu T_0 - T_3 + \mu G_0}
                   {\mu - \mu^{-1}},      
& t_0^{-1} &= \frac{\mu^{-1} T_0 - T_3 + \mu G_0}
                          {\mu^{-1} - \mu},                       \label{eq:t0}
\\
  t_3 &= \frac{T_0 - \mu T_3 + G_0}
                  {\mu^{-1} - \mu},
& t_3^{-1} &= \frac{T_0 - \mu^{-1} T_3 + G_0}
                          {\mu - \mu^{-1}}.                        \label{eq:t3}
\end{align}
\end{lemma}

\begin{proof}
We have $\mu \neq \mu^{-1}$ by the definition of a single loop.
Now the formula for $t_0$ follows from \eqref{eq:R11no1},
and the formula for $t_3$ follows from \eqref{eq:R11no2}.
Concerning $t_0^{-1}$ and $t_3^{-1}$, use $t_0^{-1} = T_0 - t_0$
and $t_3^{-1} = T_3 - t_3$. 
\end{proof}

\begin{lemma}   \label{lem:t1t2action}   \samepage
\ifDRAFT {\rm lem:t1t2action}. \fi
Let $\mu$ denote a vertex of the $X$-diagram of $\V$ that is
not incident to a double loop.
Then $q \mu \neq q^{-1} \mu^{-1}$ and the following hold on $\V_X(\mu)$:
\begin{align}
  t_1 &=  \frac{q^{-1} \mu^{-1} T_1 - T_2 - G_2}
                   {q^{-1} \mu^{-1} - q \mu},     
& t_1^{-1} &= \frac{q \mu T_1 - T_2 - G_2}
                          {q \mu - q^{-1} \mu^{-1} },                \label{eq:t1}
\\
  t_2 &= \frac{T_1 - q^{-1} \mu^{-1} T_2 - q^{-1} \mu^{-1} G_2}
                  { q \mu - q^{-1}\mu^{-1}},
& t_2^{-1} &= \frac{T_1 - q \mu T_2 - q^{-1}\mu^{-1} G_2}
                         {q^{-1} \mu^{-1} - q \mu}.                     \label{eq:t2}
\end{align}
\end{lemma}

\begin{proof}
Similar to the proof of Lemma \ref{lem:t0t3action}.
\end{proof}

\section{The $X$-diagram}
\label{sec:diagram}

Let $\V$ denote an XD $\Hq$-module.
In this section we prove that the reduced $X$-diagram of $\V$ is a path.
We then introduce the notion of  a standard ordering
of the eigenvalues of $X$.
We also describe some basic facts concerning these notions.

\begin{lemma}    \label{lem:t0t2inv}   \samepage
\ifDRAFT {\rm lem:t0t2inv}. \fi
Let $\W$ denote a subspace of $\V$ that is invariant under $X$.
\begin{itemize}
\item[\rm (i)]
$\W$ is invariant under $t_0$ if and only if $\W$ is invariant under $t_3$.
\item[\rm (ii)]
$\W$ is invariant under $t_2$ if and only if $\W$ is invariant under $t_1$.
\item[\rm (iii)]
$\W$ is an $\Hq$-submodule of $\V$ if and only if $\W$ is invariant
under each of $t_0$, $t_2$.
\end{itemize}
\end{lemma}

\begin{proof}
We have $X \W \subseteq \W$.
This forces $X \W = \W$ since $X$ is invertible and $\W$ is finite-dimensional.

(i):
First assume $t_0 \W \subseteq \W$.
This forces $t_0 \W = \W$, and so $t_0^{-1} \W = \W$.
By this and $t_3 = X t_0^{-1}$ we get $t_3 \W = \W$.
Next assume $t_3 \W \subseteq \W$.
We get $t_0 \W = \W$ in a similar way.

(ii):
Similar to the proof of (i) using $t_1 = q^{-1} X^{-1} t_2^{-1}$.

(iii):
Follows from (i) and (ii).
\end{proof}

\begin{lemma}  \label{lem:Hqinv}  \samepage
\ifDRAFT {\rm lem:Hqinv}. \fi
Let $\mu$, $\nu \in \F$.
\begin{itemize}
\item[\rm (i)]
Assume that $\mu$, $\nu$ are $1$-adjacent.
Then $\V_X(\mu)+\V_X(\nu)$ is invariant under $t_0$.
\item[\rm (ii)]
Assume that $\mu$, $\nu$ are $q$-adjacent.
Then $\V_X(\mu) + \V_X(\nu)$ is invariant under $t_2$.
\end{itemize}
\end{lemma}

\begin{proof}
(i):
For notational convenience set $\W = \V_X(\mu) + \V_X(\nu)$.
We first show $t_0 \V_X(\mu) \subseteq \W$.
Pick any $w \in \V_X(\mu)$. 
We show $t_0 w \in \W$.
Multiply each side of \eqref{eq:Xt0} on the left by $X-\mu$,
and apply the result to $w$.
Simplifying the result using $Xw=\mu w$, $X^{-1}w = \nu w$ and the fact that
$T_0$, $T_3$ commute with $X$,
we find $(X-\mu)(X-\nu)t_0 w=0$.
By this and since $X$ is diagonalizable we obtain $t_0 w \in \W$.
We have shown $t_0 \V_X(\mu) \subseteq \W$.
Interchanging $\mu$ and $\nu$ in the above arguments,
we obtain $t_0 \V_X(\nu) \subseteq \W$.
Therefore $t_0 \W \subseteq \W$.

(ii):
Similar to the proof of (i), but use \eqref{eq:Xt2} instead of \eqref{eq:Xt0}.
\end{proof}

Consider the $X$-diagram of $\V$,
which we are not assuming is reduced.

\begin{lemma}    \label{lem:G0G2V0}   \samepage
\ifDRAFT {\rm lem:G0G2V0}. \fi
Let $\mu$ denote a vertex of the $X$-diagram.
\begin{itemize}
\item[\rm (i)]
Assume that $\mu$ is not incident to a single bond.
Then $G_0 \V_X(\mu)=0$.
\item[\rm (ii)]
Assume that $\mu$ is not incident to a double bond.
Then $G_2 \V_X(\mu)=0$.
\end{itemize}
\end{lemma}

\begin{proof}
(i):
By Lemma \ref{lem:R5} $G_0 \V_X(\mu) \subseteq \V_X(\mu^{-1})$.
We have $\V_X(\mu^{-1})=0$ since $\mu$ is not incident to a single bond.
By these comments $G_0 \V_X(\mu)=0$.

(ii):
Similar.
\end{proof}

\begin{lemma}  \label{lem:vG0vG2v2}  \samepage
\ifDRAFT {\rm lem:vG0vG2v2}. \fi
Let $\mu$ denote a vertex of the $X$-diagram and let $v \in \V_X(\mu)$.
\begin{itemize}
\item[\rm (i)]
Assume that $\mu$ is not incident to a single loop.
Then $\F v + \F G_0 v$ is invariant under $t_0$.
\item[\rm (ii)]
Assume that $\mu$ is not incident to a double loop.
Then $\F v + \F G_2 v$ is invariant under $t_2$.
\end{itemize}
\end{lemma}

\begin{proof}
(i):
We show each of $t_0 v$ and $t_0 G_0 v$ is contained in $\F v + \F G_0 v$.
Recall that $T_i$ $(i \in \I)$ acts on $\V$ as a scalar.
Applying \eqref{eq:t0} to $v$ we find that $t_0 v$ is contained in $\F v + \F G_0 v$.
Concerning $t_0 G_0 v$, we may assume $G_0 v \neq 0$.
By Lemma \ref{lem:R5} $G_0 v \in \V_X(\mu^{-1})$.
In particular $\V_X(\mu^{-1}) \neq 0$ and so $\mu^{-1}$ is a vertex of the $X$-diagram.
Now applying \eqref{eq:t0} to $G_0 v$ and using Corollary \ref{cor:R9}(i) we find that
$t_0 G_0 v$ is contained in $\F v + \F G_0 v$.
The result follows.

(ii):
Similar to the proof of (i).
\end{proof}

\begin{lemma}   \label{lem:endvertex3}   \samepage
\ifDRAFT {\rm lem:endvertex3}. \fi
Let $\mu$ denote a vertex of the $X$-diagram and let
$v \in \V_X(\mu)$.
\begin{itemize}
\item[\rm (i)]
Assume that $\mu$ is not incident to a single bond.
Then $\F v$ is invariant under $t_0$.
\item[\rm (ii)]
Assume that $\mu$ is not incident to a double bond.
Then $\F v$ is invariant under  $t_2$.
\end{itemize}
\end{lemma}

\begin{proof}
(i):
By Lemma \ref{lem:G0G2V0}(i) $G_0 v =0$.
By this and Lemma \ref{lem:vG0vG2v2}(i) $\F v$ is invariant under  $t_0$.

(ii):
Similar.
\end{proof}

\begin{lemma}    \label{lem:connected}   \samepage
\ifDRAFT {\rm lem:connected}. \fi
The $X$-diagram of $\V$ is connected.
\end{lemma}

\begin{proof}
Let $\cal C$ denote a connected component of the $X$-diagram,
and let $\W$ denote the sum of the subspaces $\V_X(\mu)$ for $\mu \in {\cal C}$.
By Lemmas \ref{lem:Hqinv} and \ref{lem:endvertex3}
$\W$ is invariant under each of $t_0$, $t_2$.
By this and Lemma \ref{lem:t0t2inv}(iii) $\W$ is an $\Hq$-submodule of $\V$,
and this forces $\W= \V$ since $\V$ is irreducible.
Therefore $\cal C$ contains every vertex, so the $X$-diagram is connected.
\end{proof}

\begin{lemma}  \label{lem:diagram}  \samepage
\ifDRAFT {\rm lem:diagram}. \fi
The reduced $X$-diagram of $\V$ is a path.
\end{lemma}

\begin{proof}
By Lemma \ref{lem:connected} the reduced $X$-diagram is connected.
By the construction, each vertex is $1$-adjacent
 (resp.\ $q$-adjacent) to at most one vertex. 
Moreover there is no cycle in the reduced $X$-diagram since $q$ is not a root of unity.
By these comments the reduced $X$-diagram is a path.
\end{proof}

By Lemma \ref{lem:diagram} the reduced $X$-diagram is one of the types \eqref{eq:diag}.
An ordering $\{\mu_r\}_{r=0}^n$ of the eigenvalues of $X$ is said to be {\em standard}
whenever they are attached to the reduced $X$-diagram as in \eqref{eq:standard}.
Assume for the moment that the reduced $X$-diagram of $\V$ is {\sf DD} or {\sf SS}.
Let $\{\mu_r\}_{r=0}^n$ denote a standard ordering of the eigenvalues of $X$.
Then the ordering  $\{\mu_{n-r}\}_{r=0}^n$ is also standard and no further ordering
is standard.

For the rest of this section, we fix a standard ordering $\{\mu_r\}_{r=0}^n$ of the
eigenvalues of $X$.
By the shape of the diagram the parity of $n$ is as follows:
\begin{equation}   \label{eq:parityn}
\begin{array}{c|c}
\text{$X$-diagram} & \text{Parity of $n$}
\\ \hline
\textup{\sf DS} & \text{even}  \rule{0mm}{4mm}
\\
\textup{\sf DD}, \textup{\sf SS} &  \text{odd}
\end{array}
\end{equation}
By the construction, for $0 \leq r \leq n$ the eigenvalue $\mu_r$ is as follows:
\begin{equation}          \label{eq:mur}
\begin{array}{c|c|c}
\text{$X$-diagram} & \text{$\mu_r$ for even $r$} & \text{$\mu_r$ for odd $r$} 
\\ \hline
\textup{\sf DS}, \textup{\sf DD} & q^r \mu_0 & q^{-r-1}\mu_0^{-1}  \rule{0mm}{4mm}
\\
\textup{\sf SS}                       & q^{-r} \mu_0 & q^{r-1} \mu_0^{-1}
\end{array}
\end{equation}
In particular, $\mu_n$ is as follows:
\begin{equation}   \label{eq:mun}
\begin{array}{c|cccc}
\text{$X$-diagram} & \textup{\sf DS} & \textup{\sf DD} & \textup{\sf SS}
\\ \hline
\mu_n & \;\; q^n \mu_0 \;\; & \; q^{-n-1} \mu_0^{-1} \; 
         & \; q^{n-1} \mu_0^{-1} \; \rule{0mm}{4mm}
\end{array}
\end{equation}

\begin{lemma}  \label{lem:loop}  \samepage
\ifDRAFT {\rm lem:loop}. \fi
Consider the $X$-diagram of $\V$.
\begin{itemize}
\item[\rm (i)]
Assume $n=0$.
Then $\mu_0$ cannot be incident to both a single loop and a double loop.
\item[\rm (ii)]
Assume $n\geq 1$.
Then at most one of $\mu_0$, $\mu_n$ is incident to a loop.
\end{itemize}
\end{lemma}

\begin{proof}
First consider the case that the $X$-diagram $\V$ is {\sf DS}.
By way of contradiction, assume that $\mu_0$ is incident to a single loop and $\mu_n$ is
incident to a double loop.
So $\mu_0^2 =1$ and $\mu_n^2 = q^{-2}$.
By \eqref{eq:mun} we have $\mu_n = q^n \mu_0$.
By these comments $q^{2n+2}=1$, contradicting the assumption that $q$ is not a
root of unity.
We have shown the assertion for the case {\sf DS}.
For the cases {\sf DD}, {\sf SS} the proof is similar, and omitted.
\end{proof}

\section{The eigenspaces of $X$}
\label{sec:mfree}

Let $\V$ denote an XD $\Hq$-module.
In this section we show that $X$ is multiplicity-free on $\V$.
We then introduce the notion of an $X$-standard basis.
Consider the $X$-diagram of $\V$,
which we are not assuming is reduced.
We use the following term.

\begin{definition}   \label{def:follows}   \samepage
\ifDRAFT {\rm def:follows}. \fi
Let $\mu$, $\nu$ denote distinct vertices of the $X$-diagram that are
$1$-adjacent or $q$-adjacent.
For $u \in \V_X(\mu)$ and $v \in \V_X(\nu)$,
we say that {\em $v$ follows $u$} whenever $v = Gu$,
where $G=G_0$ if $\mu$, $\nu$ are $1$-adjacent and
$G=G_2$ if $\mu$, $\nu$ are $q$-adjacent.
\end{definition}

For the rest of this section,
we fix a standard ordering $\{\mu_r\}_{r=0}^n$ of the eigenvalues of $X$.

\begin{definition}   \label{def:forwardchain}   \samepage
\ifDRAFT {\rm def:forward chain}. \fi
By a {\em forward chain} in $\V$ (with respect to the ordering $\{\mu_r\}_{r=0}^n$)
we mean a sequence of vectors $\{v_r\}_{r=0}^n$ in $\V$ such that
\begin{itemize}
\item[\rm (i)]
$v_r \in \V_X(\mu_r)$ for $0 \leq r \leq n$;
\item[\rm (ii)]
$v_0 \neq 0$;
\item[\rm (iii)]
$v_r$ follows $v_{r-1}$ for $1 \leq r \leq n$.
\end{itemize}
\end{definition}

\begin{definition}   \label{def:backwardchain}   \samepage
\ifDRAFT {\rm def:backwardchain}. \fi
By a {\em backward chain} in $\V$ (with respect to the ordering $\{\mu_r\}_{r=0}^n$)
we mean a sequence of vectors $\{v_r\}_{r=0}^n$ in $\V$ such that
\begin{itemize}
\item[\rm (i)]
$v_r \in \V_X(\mu_r)$ for $0 \leq r \leq n$;
\item[\rm (ii)]
$v_n \neq 0$;
\item[\rm (iii)]
$v_{r-1}$ follows $v_r$ for $1 \leq r \leq n$.
\end{itemize}
\end{definition}

\begin{definition}   \label{def:chain}   \samepage
\ifDRAFT {\rm def:chain}. \fi
A sequence of vectors $\{v_r\}_{r=0}^n$ in $\V$ is called a {\em chain}
whenever it is a forward or backward chain.
\end{definition}

\begin{lemma}    \label{lem:invert} \samepage
\ifDRAFT {\rm lem:invert}. \fi
Assume that the reduced $X$-diagram of $\V$ is of type {\sf DD} or {\sf SS}. 
Let $\{v_r\}_{r=0}^n$ denote a forward chain (resp.\ backward chain)
with respect to the ordering $\{\mu_r\}_{r=0}^n$.
Then $\{v_{n-r}\}_{r=0}^n$ is a backward chain (resp.\ forward chain)
with respect to the ordering $\{\mu_{n-r}\}_{r=0}^n$.
\end{lemma}

\begin{lemma}   \label{lem:chainunique}   \samepage
\ifDRAFT {\rm lem:chainunique}. \fi
The following hold.
\begin{itemize}
\item[\rm (i)]
For a nonzero $v \in \V_X(\mu_0)$,
there exists a unique forward chain $\{v_r\}_{r=0}^n$ such that $v_0 = v$.
\item[\rm (ii)]
For a nonzero $v \in \V_X(\mu_n)$,
there exists a unique backward chain $\{v_r\}_{r=0}^n$ such that $v_n=v$.
\end{itemize}
\end{lemma}

\begin{proof}
Routine using Lemma \ref{lem:R5}.
\end{proof}

\begin{lemma}    \label{lem:mu0mun}    \samepage
\ifDRAFT {\rm lem:mu0mun}. \fi
Let $\mu$ denote an endvertex of the $X$-diagram.
\begin{itemize}
\item[\rm (i)]
Assume that $\mu$ is incident to a double bond in the reduced $X$-diagram of $\V$.
Then $\V_X(\mu)$ is invariant under $t_0$.
\item[\rm (ii)]
Assume that $\mu$ is incident to a single bond in the reduced $X$-diagram of $\V$.
Then $\V_X(\mu)$ is invariant under $t_2$.
\end{itemize}
\end{lemma}

\begin{proof}
(i):
If $\mu$ is incident to a single loop, then $\mu^{-1} = \mu$.
If $\mu$ is not incident to a single loop, then $\V_X(\mu^{-1}) = 0$.
In either case $\V_X(\mu^{-1}) \subseteq \V_X(\mu)$.
By Lemma \ref{lem:Hqinv}(i) $\V_X(\mu) + \V_X(\mu^{-1})$ is invariant under $t_0$.
By these comments $\V_X(\mu)$ is invariant under $t_0$.

(ii):
Similar.
\end{proof}

\begin{proposition}   \label{prop:mfree}   \samepage
\ifDRAFT {\rm prop:mfree}. \fi
The following hold.
\begin{itemize}
\item[\rm (i)]
$X$ is multiplicity-free on $\V$.
\item[\rm (ii)]
Every chain in $\V$ forms a basis for $\V$.
\end{itemize}
\end{proposition}

\begin{proof}
We assume that the $X$-diagram of $\V$ is of type {\sf DS};
the proof is similar for the other types.
Note that $\V_X(\mu_0)$ (resp.\ $\V_X(\mu_n)$) is invariant under $t_0$ (resp. $t_2$)
by Lemma \ref{lem:mu0mun}.
We claim that a chain $\{v_r\}_{r=0}^n$ in $\V$ forms a basis for $\V$
provided that the following conditions hold:
\begin{align}
  & \text{ $\F v_0$ is invariant under $t_0$};                      \label{eq:condv0}
\\
  & \text{ $\F v_n$ is invariant under $t_2$}.                      \label{eq:condvn}
\end{align}
Let $\W$ denote the subspace of $\V$ spanned by $\{v_r\}_{r=0}^n$.
Note that $\W \neq 0$ by the definition of a chain.
Using Lemma \ref{lem:vG0vG2v2} and the conditions \eqref{eq:condv0}, \eqref{eq:condvn}
one finds that $\W$ is invariant under each of $t_0$, $t_2$.
By this and Lemma \ref{lem:t0t2inv}(iii) $\W$ is an $\Hq$-submodule of $\V$,
and this forces $\W=\V$ since $\V$ is irreducible.
Therefore $\{v_r\}_{r=0}^n$ forms a basis for $\V$.

(i):
We construct a chain $\{v_r\}_{r=0}^n$ that satisfies \eqref{eq:condv0} and \eqref{eq:condvn}.
By Lemma \ref{lem:loop}
either $\mu_0$ is not incident to a single loop or $\mu_n$ is not incident to a double loop.
First assume that $\mu_n$ is not incident to a double loop.
There exists an eigenvector $v$ of $t_0$ that is contained in $\V_X(\mu_0)$ since $\V_X(\mu_0)$
is invariant under $t_0$.
By Lemma \ref{lem:chainunique}(i) there exists a forward chain $\{v_r\}_{r=0}^n$ such that $v_0=v$.
This chain satisfies \eqref{eq:condv0} by the construction,
and satisfies \eqref{eq:condvn} by Lemma \ref{lem:endvertex3}(ii).
Next assume that $\mu_0$ is not incident to a single loop.
We can construct a backward chain $\{v_r\}_{r=0}^n$ that satisfies \eqref{eq:condv0} and \eqref{eq:condvn}
in a similar way as above.
We have constructed a chain $\{v_r\}_{r=0}^n$ that satisfies \eqref{eq:condv0} and \eqref{eq:condvn}.
By the claim this chain forms a basis for $\V$,
and so $\V_X(\mu_r)$ is spanned by $v_r$ for $0 \leq r \leq n$.
Thus $X$ is multiplicity-free on $\V$.

(ii):
Let $\{v_r\}_{r=0}^n$ denote a chain in $\V$.
This chain satisfies \eqref{eq:condv0} and \eqref{eq:condvn} by (i) and since
$\V_X(\mu_0)$ (resp.\ $\V_X(\mu_n)$) is invariant under $t_0$ (resp.\ $t_2$).
By the claim this chain forms a basis for $\V$.
\end{proof}

\begin{lemma}    \label{lem:G0G2V}   \samepage
\ifDRAFT {\rm lem:G0G2V}. \fi
Let $\mu$, $\nu$ denote distinct eigenvalues of $X$.
\begin{itemize}
\item[\rm (i)]
Assume that $\mu$, $\nu$ are $1$-adjacent.
Then $G_0 \V_X(\mu) = \V_X(\nu)$ and $G_0 \V_X(\nu)  = \V_X(\mu)$.
\item[\rm (ii)]
Assume that $\mu$, $\nu$ are $q$-adjacent.
Then $G_2 \V_X(\mu) = \V_X(\nu)$  and $G_2 \V_X(\nu) = \V_X(\mu)$.
\end{itemize}
\end{lemma}

\begin{proof}
Routine using Lemma \ref{lem:R5} and Proposition \ref{prop:mfree}(ii).
\end{proof}

\begin{definition}   \label{def:Xstandard}
\ifDRAFT {\rm def:Xstandard}. \fi
By an {\em $X$-standard basis} for $\V$ (corresponding to the ordering
$\{\mu_r\}_{r=0}^n$) we mean a basis $\{v_r\}_{r=0}^n$ for $\V$
that is a forward chain.
\end{definition}

We mention two corollaries for later use.
Let $\{k_i\}_{i \in \I}$ denote a parameter sequence of $\V$.

\begin{corollary}   \label{cor:nonzeropre} \samepage
\ifDRAFT {\rm cor:nonzeropre}. \fi
Let $\mu$, $\nu$ denote distinct eigenvalues of $X$.
\begin{itemize}
\item[\rm (i)]
Assume that $\mu$, $\nu$ are $1$-adjacent.
Then
$G(\mu,k_0,k_3) \neq 0$.
\item[\rm (ii)]
Assume that $\mu$, $\nu$ are $q$-adjacent.
Then
$G(q \mu, k_1, k_2) \neq 0$.
\end{itemize}
\end{corollary}

\begin{proof}
Follows from Corollary \ref{cor:R9} and Lemma \ref{lem:G0G2V}.
\end{proof}

\begin{corollary}   \label{cor:nonzero}   \samepage
\ifDRAFT {\rm cor:nonzero}. \fi
Let $\mu$, $\nu$ denote distinct eigenvalues of $X$.
\begin{itemize}
\item[\rm (i)]
Assume that $\mu$, $\nu$ are $1$-adjacent.
Then $\mu$ is not among
\[
 k_0k_3, \quad k_0k_3^{-1}, \quad k_0^{-1}k_3, \quad k_0^{-1}k_3^{-1}.
\] 
\item[\rm (ii)]
Assume that $\mu$, $\nu$ are $q$-adjacent.
Then $q \mu$ is not among
\[
    k_1k_2,\quad k_1k_2^{-1}, \quad k_1^{-1}k_2, \quad k_1^{-1}k_2^{-1}.
\]
\end{itemize}
\end{corollary}

\begin{proof}
Follows from Corollary \ref{cor:nonzeropre} and \eqref{eq:defG}.
\end{proof}

\section{The $X$-type of an XD $\Hq$-module}
\label{sec:types}

In Section 1 we discussed the five cases
{\sf DS}, {\sf DDa}, {\sf DDb}, {\sf SSa}, {\sf SSb} with reference to some lemmas.
In this section we prove these lemmas.
We then define the $X$-type of an XD $\Hq$-module,
and make a rule concerning the corresponding parameter sequence 
$\{k_i\}_{i \in \I}$.
Let $\V$ denote an XD $\Hq$-module and let $\{\mu_r\}_{r=0}^n$ denote a 
standard ordering of the eigenvalues of $X$.

\begin{lemma}  \label{lem:endvertex}   \samepage 
\ifDRAFT {\rm lem:endvertex}. \fi
The subspaces $\V_X(\mu_0)$ and $\V_X(\mu_n)$
are invariant under the following elements:
\[
\begin{array}{c|c|c}
\text{\rm $X$-diagram}
 & \text{\rm $\V_X(\mu_0)$ is invariant under} 
 & \text{\rm $\V_X(\mu_n)$ is invariant under}
\\ \hline
\textup{\sf DS} \rule{0mm}{4mm}  
  & t_0,\qquad t_3 & t_1,\qquad t_2 \\
\textup{\sf DD}  & t_0,\qquad t_3 & t_0,\qquad t_3 \\
\textup{\sf SS}  & t_1,\qquad t_2 & t_1,\qquad t_2
\end{array}
\]
\end{lemma}

\begin{proof}
Follows from Lemmas \ref{lem:t0t2inv} and \ref{lem:mu0mun}.
\end{proof}

\begin{lemma}   \label{lem:typesDD}  \samepage
\ifDRAFT {\rm lem:typesDD}. \fi
Assume that the reduced $X$-diagram of $\V$ is \textup{\sf DD}.
Then one of the following cases occurs:
\[
\begin{array}{c|c|c}
 \text{\rm Case} 
 & 
   \begin{array}{c}
     \text{\rm Eigenvalues of $t_0$ on} \\
     \text{\rm $\V_X(\mu_0)$ and $\V_X(\mu_n)$}
   \end{array}
 &
   \begin{array}{cc}
     \text{\rm Eigenvalues of $t_3$ on} \\
     \text{\rm $\V_X(\mu_0)$ and $\V_X(\mu_n)$}
   \end{array}
\\ \hline
\textup{\sf DDa} \rule{0mm}{4mm}
  & \text{\rm same} & \text{\rm reciprocals} \\
\textup{\sf DDb} & \text{\rm reciprocals} & \text{\rm same}
\end{array}
\]
\end{lemma}

\begin{proof}
We consider the subspaces $\V_X(\mu_0)$ and $\V_X(\mu_n)$.
First assume that the eigenvalues of $t_0$ on the subspaces are the same.
Then the eigenvalues of $t_3$ are reciprocals;
otherwise the eigenvalues of $X=t_3t_0$ are the same, and this forces $\mu_0 = \mu_n$.
Next assume that the eigenvalues of $t_0$ are reciprocals.
Then the eigenvalues of $t_3$ are the same;
otherwise the eigenvalues of $X$ are reciprocals and this forces
$q^{-n-1}=1$ since $\mu_0 \mu_n = q^{-n-1}$ by \eqref{eq:mun}.
Thus at least one of the cases {\sf DDa}, {\sf DDb} occurs.
These cases do not occur at the same time;
otherwise the eigenvalues of $t_0$, $t_3$ on the subspaces
are contained in $\{1, -1\}$ and this forces $q^{-n-1} = \pm 1$ since
$\mu_0 \mu_n = q^{-n-1}$ by \eqref{eq:mun}.
\end{proof}

\begin{lemma}   \label{lem:typesSS}  \samepage
\ifDRAFT {\rm lem:typesSS}. \fi
Assume that the reduced $X$-diagram of $\V$ is \text{\sf SS}.
Then one of the following cases occurs:
\[
\begin{array}{c|c|c}
 \text{\rm Case} 
 & 
   \begin{array}{c}
     \text{\rm Eigenvalues of $t_1$ on} \\
     \text{\rm $\V_X(\mu_0)$ and $\V_X(\mu_n)$}
   \end{array}
 &
   \begin{array}{cc}
     \text{\rm Eigenvalues of $t_2$ on} \\
     \text{\rm $\V_X(\mu_0)$ and $\V_X(\mu_n)$}
   \end{array}
\\ \hline
\textup{\sf SSa} \rule{0mm}{4mm}
 & \text{\rm same} & \text{\rm reciprocals} \\
\textup{\sf SSb} & \text{\rm reciprocals} & \text{\rm same}
\end{array}
\]
\end{lemma}

\begin{proof}
Similar to the proof of Lemma \ref{lem:typesDD} using $X = q^{-1} t_2^{-1} t_1^{-1}$.
\end{proof}

In view of Lemmas \ref{lem:typesDD} and \ref{lem:typesSS} we make some definitions.

\begin{definition}   \label{def:Xtypes}   \samepage
\ifDRAFT {\rm def:Xtypes}. \fi
We define the {\em $X$-type} of $\V$ as follows.
We say that $\V$ has $X$-type {\sf DS} whenever the reduced $X$-diagram of $\V$
is {\sf DS}.
We  say that $\V$ has $X$-type {\sf DDa} (resp.\ {\sf DDb})
whenever the reduced $X$-diagram of $\V$ is {\sf DD} and
case {\sf DDa} (resp.\ {\sf DDb}) occurs in Lemma \ref{lem:typesDD}.
We say that $\V$ has $X$-type {\sf SSa} (resp.\ {\sf SSb})
whenever the reduced $X$-diagram of $\V$ is {\sf SS} and
case {\sf SSa} (resp.\ {\sf SSb}) occurs in Lemma \ref{lem:typesSS}.
\end{definition}

\begin{definition}          \label{def:consistent}   \samepage
\ifDRAFT {\rm def:consistent}. \fi
A parameter sequence $\{k_i\}_{i \in \I}$ of $\V$ is said to be
{\em consistent with the ordering $\{\mu_r\}_{r=0}^n$}
whenever it follows the rule:
\begin{equation}     \label{eq:rule}
\begin{array}{c|c}
\text{$X$-type of $\V$} & \text{Rule}
\\ \hline
\text{\sf DS} \rule{0mm}{6.5mm} &
\begin{array}{l}
\text{$k_0$ (resp.\ $k_3$) is the eigenvalue of $t_0$ (resp.\ $t_3$) on $\V_X(\mu_0)$} \\
\text{$k_1$ (resp.\ $k_2$) is the eigenvalue of $t_1$ (resp.\ $t_2$) on $\V_X(\mu_n)$}
\end{array}
\\ \hline
\text{\sf DDa}, \text{\sf DDb} \rule{0mm}{4.2mm} & 
\begin{array}{l}
\text{$k_0$ (resp.\ $k_3$) is the  eigenvalue of $t_0$ (resp.\ $t_3$) on $\V_X(\mu_0)$}
\end{array}
\\ \hline
\text{\sf SSa}, \text{\sf SSb} \rule{0mm}{4.2mm} &
\begin{array}{l}
\text{$k_1$ (resp. $k_2$) is the eigenvalue of $t_1$ (resp.\ $t_2$) on $\V_X(\mu_0)$}
\end{array}
\end{array}
\end{equation}
\end{definition}

\begin{lemma}    \label{lem:tiv0vn}   \samepage
\ifDRAFT {\rm lem:tiv0vn}. \fi
Assume that the parameter sequence $\{k_i\}_{i \in \I}$ is consistent with
the ordering $\{\mu_r\}_{r=0}^n$.
Pick nonzero vectors $v_0 \in \V_X(\mu_0)$ and $v_n \in \V_X(\mu_n)$.
Then
\begin{equation}                 \label{eq:tiv0vn}
\begin{array}{c|llll}
\textup{\rm $X$-type of $\V$} & \multicolumn{4}{c}{\text{\rm Action}}
\\ \hline \rule{0mm}{4.3mm}
 \textup{\sf DS} &
 t_0 v_0 = k_0 v_0 \;\; & t_3 v_0 = k_3 v_0\;\;  & t_1 v_n = k_1 v_n \;\;& t_2 v_n = k_2 v_n  
\\  \rule{0mm}{4mm}
 \textup{\sf DDa} & 
 t_0 v_0 = k_0 v_0 & t_3 v_0 = k_3 v_0 & t_0 v_n = k_0 v_n & t_3 v_n = k_3^{-1} v_n
\\ \rule{0mm}{4mm}
 \textup{\sf DDb} & 
 t_0 v_0 = k_0 v_0 & t_3 v_0 = k_3 v_0 & t_0 v_n = k_0^{-1} v_n & t_3 v_n = k_3 v_n
\\ \rule{0mm}{4mm}
 \textup{\sf SSa} & 
 t_1 v_0 = k_1 v_0 & t_2 v_0 = k_2 v_0 & t_1 v_n = k_1 v_n & t_2 v_n = k_2^{-1} v_n
\\ \rule{0mm}{4mm}
 \textup{\sf SSb} & t_1 v_0 = k_1 v_0 & t_2 v_0 = k_2 v_0 & t_1 v_n = k_1^{-1} v_n & t_2 v_n = k_2 v_n
\end{array}
\end{equation}
\end{lemma}

\begin{proof}
Immediate from Definitions \ref{def:Xtypes} and \ref{def:consistent}.
\end{proof}

By the construction, among the parameter sequences $\{k_i^{\pm 1} \}_{i \in \I}$ of $\V$,
at least one sequence is consistent with the ordering $\{\mu_r\}_{r=0}^n$.
Fix a parameter sequence $\{k_i\}_{i \in \I}$ of $\V$ that is consistent with
the ordering $\{\mu_r\}_{r=0}^n$.
Then the parameter sequences of $\V$ that are consistent with the 
ordering  $\{\mu_r\}_{r=0}^n$ are as follows:
\[
\begin{array}{c|c}
  \text{\rm $X$-type of $\V$} & \text{\rm Parameter sequences}
\\ \hline
 \rule{0mm}{4mm}
 \textup{\sf DS} & (k_0,k_1,k_2,k_3)
\\
 \textup{\sf DDa}, \textup{\sf DDb} & (k_0, k_1^{\pm 1}, k_2^{\pm 1}, k_3)
\\
 \textup{\sf SSa}, \textup{\sf SSb} & (k_0^{\pm 1}, k_1,k_2, k_3^{\pm 1})
\end{array}
\]
Assume for the moment that  $\V$ has $X$-type among 
{\sf DDa}, {\sf DDb}, {\sf SSa}, {\sf SSb}.
Recall that the ordering $\{\mu_{n-r}\}_{r=0}^n$ is also standard.
The parameter sequences of $\V$ that are consistent with the 
ordering  $\{\mu_{n-r}\}_{r=0}^n$ are as follows:
\[
\begin{array}{c|c}
  \text{\rm $X$-type of $\V$} & \text{\rm Parameter sequences}
\\ \hline
 \rule{0mm}{4.5mm}
 \textup{\sf DDa} & (k_0, k_1^{\pm 1}, k_2^{\pm 1}, k_3^{-1})
\\
 \textup{\sf DDb} & (k_0^{-1}, k_1^{\pm 1}, k_2^{\pm 1}, k_3)
\\
  \textup{\sf SSa} & (k_0^{\pm 1}, k_1,k_2^{-1}, k_3^{\pm 1})
\\
  \textup{\sf SSb} & (k_0^{\pm 1}, k_1^{-1},k_2, k_3^{\pm 1})
\end{array}
\]
By the above observations, the parameter sequences of $\V$ that are 
consistent with a standard ordering  of the eigenvalues of $X$ are as follows:
\[
\begin{array}{c|c}
  \text{\rm $X$-type of $\V$} & \text{\rm Parameter sequences}
\\ \hline
 \rule{0mm}{4mm}
\textup{\sf DS} & (k_0,k_1,k_2,k_3)
\\
 \textup{\sf DDa} & (k_0, k_1^{\pm 1}, k_2^{\pm 1}, k_3^{\pm 1})
\\
 \textup{\sf DDb} & (k_0^{\pm 1}, k_1^{\pm 1}, k_2^{\pm 1}, k_3)
\\
  \textup{\sf SSa} & (k_0^{\pm 1}, k_1,k_2^{\pm 1}, k_3^{\pm 1})
\\
  \textup{\sf SSb} & (k_0^{\pm 1}, k_1^{\pm 1},k_2, k_3^{\pm 1})
\end{array}
\]
Note by this that a parameter sequence of $\V$ may be consistent 
with no standard ordering of the eigenvalues of $X$.

We mention two lemmas for later use.

\begin{lemma}   \label{lem:eigenvalues}  \samepage
\ifDRAFT {\rm lem:eignevalues}. \fi
Let $\{k_i\}_{i \in \I}$ denote a parameter sequence of $\V$ that is
consistent with the ordering $\{\mu_r\}_{r=0}^n$.
\begin{itemize}
\item[\rm (i)]
Assume that $\V$ has $X$-type among \textup{\sf DS}, \textup{\sf DDa}, \textup{\sf DDb}.
Then
\begin{align}       \label{eq:Dmur}
 \mu_r &= 
  \begin{cases}
    k_0k_3 q^{r}            &  \text{ if $r$ is even},  \\
    \frac{1}{k_0k_3q^{r+1} } & \text{ if $r$ is odd}
  \end{cases}
  && (0 \leq r \leq n).
\end{align}
\item[\rm (ii)]
Assume that $\V$ has $X$-type among \textup{\sf SSa}, \textup{\sf SSb}.
Then
\begin{align}       \label{eq:Smur}
 \mu_r &= 
  \begin{cases}
    \frac{1}{k_1k_2 q^{r+1} }  &  \text{ if $r$ is even},  \\
    k_1k_2 q^r                & \text{ if $r$ is odd}
  \end{cases}
  && (0 \leq r \leq n).
\end{align}
\end{itemize}
\end{lemma}

\begin{proof}
(i):
By \eqref{eq:rule} $t_0$ (resp.\ $t_3$) has eigenvalue $k_0$ (resp.\ $k_3$)
on $\V_X(\mu_0)$.
By this and $X=t_3t_0$ we get $\mu_0 = k_0k_3$.
Now \eqref{eq:Dmur} follows from \eqref{eq:mur}.

(ii):
Similar to the proof of (i) using $X=q^{-1}t_2^{-1}t_1^{-1}$.
\end{proof}

\begin{lemma}   \label{lem:G20}    \samepage
\ifDRAFT {\rm lem:G20}. \fi
In the reduced $X$-diagram of $\V$, consider an endvertex $\mu$.
\begin{itemize}
\item[\rm (i)]
Assume that $\mu$ is not incident to a double bond.
Then $G_2 \V_X(\mu)=0$.
\item[\rm (ii)]
Assume that $\mu$ is not incident to a single bond.
Then $G_0 \V_X(\mu)=0$.
\end{itemize}
\end{lemma}

\begin{proof}
(i):
Pick a nonzero $v \in \V_X(\mu)$.
Note that $\V_X(\mu) = \F v$.
By Lemma \ref{lem:endvertex} $\V_X(\mu)$ is invariant under $t_1$, $t_2$.
So $v$ is an eigenvector of $t_1$ and $t_2$.
By these comments $t_1$ and $t_2$ commute on $\V_X(\mu)$.
By this and Definition \ref{def:G0G2} we get $G_2 \V_X(\mu)=0$.

(ii):
Similar.
\end{proof}

\section{The action of $t_i$ on the $X$-standard basis}
\label{sec:structure}

Let $\V$ denote an XD $\Hq$-module that has parameter sequence 
$\{k_i\}_{i \in \I}$.
Let $\{\mu_r\}_{r=0}^n$ denote a standard ordering of the eigenvalues of $X$,
and let $\{v_r\}_{r=0}^n$ denote a corresponding $X$-standard basis for $\V$.
In this section we display the actions of $\{t_i\}_{i \in \I}$ on $\{v_r\}_{r=0}^n$.

\begin{lemma} \label{lem:R15}    \samepage
\ifDRAFT {\rm lem:R15}. \fi
For $0 \leq r \leq n-1$ such that $\mu_r$, $\mu_{r+1}$ are $1$-adjacent,
the elements $t_0$, $t_3$ act on $\F v_r + \F v_{r+1}$ as
\begin{align}
 t_0 v_r &=
  \frac{\mu_r (k_0+k_0^{-1}) - k_3-k_3^{-1}}
       {\mu_r - \mu_r^{-1}} v_r
 + \frac{\mu_r}
        {\mu_r - \mu_r^{-1}}  v_{r+1},              \label{eq:t0vr}
\\
 t_0 v_{r+1} &=
  \frac{G(\mu_r, k_0,k_3)}
       {\mu_r (\mu_r^{-1} - \mu_r)} v_{r} 
  + \frac{\mu_r^{-1} (k_0+k_0^{-1}) - k_3-k_3^{-1}}
         {\mu_r^{-1} - \mu_r} v_{r+1},              \label{eq:t0vr+1}
\\
 t_3 v_r &=
   \frac{\mu_r (k_3+k_3^{-1}) - k_0-k_0^{-1}}
        {\mu_r - \mu_r^{-1}} v_r
 +  \frac{1}
         {\mu_r^{-1} - \mu_r} v_{r+1},              \label{eq:t3vr}
\\
 t_3 v_{r+1} &=
     \frac{G(\mu_r, k_0,k_3)}
          {\mu_r - \mu_r^{-1}}  v_r
  +  \frac{\mu_r^{-1} (k_3+k_3^{-1}) - k_0-k_0^{-1}}
          {\mu_r^{-1} - \mu_r} v_{r+1}.              \label{eq:t3vr+1}
\end{align}
In \eqref{eq:t0vr}--\eqref{eq:t3vr+1} the denominators are nonzero by Lemma \ref{lem:t0t3action}.
\end{lemma}

\begin{proof}
Recall that for $i \in \I$ the element $T_i$ acts on $\V$ as the scalar $k_i + k_i^{-1}$.
Note that $v_{r+1} = G_0 v_r$ by the construction.
Now applying \eqref{eq:t0} to $v_r$ one finds \eqref{eq:t0vr}.
To get \eqref{eq:t0vr+1},
apply \eqref{eq:t0} to $G_0 v_r$ and use Corollary \ref{cor:R9}(i).
The lines \eqref{eq:t3vr} and \eqref{eq:t3vr+1} are similarly
obtained from \eqref{eq:t3}.
\end{proof}

\begin{lemma} \label{lem:R20} \samepage
\ifDRAFT {\rm lem:R20}. \fi
For $0 \leq r \leq n-1$ such that $\mu_r$, $\mu_{r+1}$ are $q$-adjacent,
the elements $t_1$, $t_2$ act on
$\F v_r + \F v_{r+1}$ as
\begin{align}
t_1 v_r &=
   \frac{q^{-1}\mu_r^{-1} (k_1+k_1^{-1}) - k_2-k_2^{-1}}
        {q^{-1}\mu_r^{-1} - q\mu_r} v_r
 + \frac{1}
        {q\mu_r - q^{-1}\mu_r^{-1}} v_{r+1},   \label{eq:t1vr}
\\
t_1 v_{r+1} &=
   \frac{G(q \mu_r, k_1,k_2)}
        {q^{-1}\mu_r^{-1} - q\mu_r} v_r
 + \frac{q\mu_r (k_1+k_1^{-1}) - k_2-k_2^{-1}}
        {q\mu_r - q^{-1}\mu_r^{-1}}  v_{r+1},  \label{eq:t1vr+1}
\\
t_2 v_r &=
   \frac{q^{-1}\mu_r^{-1} (k_2+k_2^{-1}) - k_1-k_1^{-1}}
       {q^{-1}\mu_r^{-1} - q\mu_r} v_r
 + \frac{q^{-1}\mu_r^{-1}}
        {q^{-1}\mu_r^{-1} - q\mu_r} v_{r+1},    \label{eq:t2vr}
\\
t_2 v_{r+1} &=
   \frac{q \mu_r\, G(q \mu_r, k_1,k_2)}
        {q\mu_r - q^{-1}\mu_r^{-1}} v_r 
 + \frac{q\mu_r (k_2+k_2^{-1}) - k_1-k_1^{-1}}
        {q\mu_r - q^{-1}\mu_r^{-1}} v_{r+1}.   \label{eq:t2vr+1}
\end{align}
In \eqref{eq:t1vr}--\eqref{eq:t2vr+1} the denominators are nonzero by Lemma \ref{lem:t1t2action}.
\end{lemma}

\begin{proof}
Similar to the proof of Lemma \ref{lem:R15}.
\end{proof}

In the above, we displayed some actions of $\{t_i\}_{i \in \I}$ on $\{v_i\}_{i \in \I}$.
The remaining actions are given in \eqref{eq:tiv0vn}.

\begin{note}   \label{note:iso}    \samepage
\ifDRAFT {\rm note:iso}. \fi
By Lemmas \ref{lem:tiv0vn}, \ref{lem:R15}, \ref{lem:R20}
the action of $\{t_i\}_{i \in \I}$ on an $X$-standard basis is determined
by the parameter sequence and the $X$-type of the $\Hq$-module.
Therefore an XD $\Hq$-module is uniquely determined up to isomorphism
by its dimension,  its parameter sequence, and its $X$-type.
\end{note}

\section{The eigenspaces of $t_0$}
\label{sec:eigent0}

Throughout this section the following notation is in effect.

\begin{notation}   \label{notation}   \samepage
\ifDRAFT {\rm notation}. \fi
Let $\V$ denote an XD $\Hq$-module with dimension $n+1$.
Let $\{\mu_r\}_{r=0}^n$ denote a standard ordering of the eigenvalues of $X$.
Let $\{k_i\}_{i \in \I}$ denote a parameter sequence of $\V$ that is consistent with
the ordering $\{\mu_r\}_{r=0}^n$.
\end{notation}

In this section we obtain the dimensions of the eigenspaces of $t_0$.
Define subspaces of $\V$:
\begin{align}
 \V(k_0) &= \{v \in \V \,|\, t_0 v = k_0 v\},
&
 \V(k_0^{-1}) &= \{ v \in \V \,|\, t_0 v = k_0^{-1} v\}.         \label{eq:defVk0}
\end{align}
Thus $\V(k_0)$ is the eigenspace of $t_0$ associated with eigenvalue $k_0$,
provided that $\V(k_0) \neq 0$.
Similar for $\V(k_0^{-1})$.

\begin{definition}    \label{def:F+F-}   \samepage
\ifDRAFT {\rm def:F+F-}. \fi
Assume $k_0 \neq k_0^{-1}$.
Define elements of $\Hq$:
\begin{align}                     
  F^+ &= \frac{t_0-k_0^{-1}}{k_0-k_0^{-1}},
&
 F^- &= \frac{t_0-k_0}{k_0^{-1}-k_0}.                       \label{eq:defF+F-}
\end{align}
\end{definition}

Assume for the moment that $k_0 \neq k_0^{-1}$.
Recall that $(t_0 - k_0)(t_0-k_0^{-1})\V = 0$.
So
\begin{align*}
  \V &= \V(k_0) + \V(k_0^{-1})  \qquad\qquad   \text{(direct sum)}.
\end{align*}
Observe that $F^+$ (resp.\ $F^-$) acts on $\V$ as the projection onto $\V(k_0)$
(resp. $\V(k_0^{-1})$).

\begin{lemma}    \label{lem:t0t0inv}    \samepage
\ifDRAFT {\rm lem:t0t0inv}. \fi
Let $\mu$, $\nu$  denote a single bond in the reduced $X$-diagram of $\V$.
Then
\begin{align}
 t_0 \V_X(\mu) &\subseteq \V_X(\mu) + \V_X(\nu),  &
 t_0 \V_X(\mu) &\not\subseteq \V_X(\mu),                          \label{eq:t0VXmu}
\\
 t_0^{-1} \V_X(\mu) &\subseteq \V_X(\mu) + \V_X(\nu),  &
 t_0^{-1} \V_X(\mu) &\not\subseteq \V_X(\mu).                 \label{eq:t0invVXmu}
\end{align}
\end{lemma}

\begin{proof}
By Lemma \ref{lem:G0G2V}(i) $G_0 \V_X(\mu) = \V_X(\nu)$.
The result follows from this and Lemma \ref{lem:t0t3action}.
\end{proof}

\begin{lemma}    \label{lem:F+F-}    \samepage
\ifDRAFT {\rm lem:F+F-}. \fi
Assume $k_0 \neq k_0^{-1}$.
Let $\mu$, $\nu$ denote a single bond in the reduced $X$-diagram of $\V$.
Then for $0 \neq v \in \V_X(\mu)$
each of $F^+ v$ and $F^- v$ is nonzero and contained in $\V_X(\mu) + \V_X(\nu)$.
\end{lemma}

\begin{proof}
By Lemma \ref{lem:t0t0inv} there exist $\alpha \in \F$ and $0 \neq u \in \V_X(\nu)$
such that $t_0 v = \alpha v + u$.
Using this and \eqref{eq:defF+F-} we argue
\[
 (k_0 - k_0^{-1}) F^+ v = t_0 v - k_0^{-1} v
    =  (\alpha - k_0^{-1}) v + u.
\]
Thus $F^+ v$ is nonzero and contained in $\V_X(\mu) + \V_X(\nu)$.
The proof is similar for $F^- v$.
\end{proof}

\begin{lemma} \label{lem:WF+WF-W} \samepage
\ifDRAFT {\rm lem:WF+WF-W}. \fi
Assume $k_0 \neq k_0^{-1}$.
Let $\mu, \nu$ denote a single bond in the reduced $X$-diagram of $\V$,
and set $\W = \V_X(\mu) + \V_X(\nu)$. 
Then the following hold:
\begin{itemize}
\item[\rm (i)]
$\W = F^+ \W + F^- \W \;\;$  (direct sum);
\item[\rm (ii)]
each of $F^+ \W$ and $F^- \W$ has dimension $1$;
\item[\rm (iii)]
$F^+ \W = \W \cap \V(k_0)$ and
$F^- \W = \W \cap \V(k_0^{-1})$;
\item[\rm (iv)]
each of $F^+ \V_X(\mu) $ and $F^+ \V_X(\nu)$ is  equal to $F^+ \W$;
\item[\rm (v)]
each of $F^- \V_X(\mu)$ and $F^- \V_X(\nu)$ is equal to $F^- \W$.
\end{itemize}
\end{lemma}

\begin{proof}
By Lemma \ref{lem:F+F-} each of $F^+ \V_X(\mu)$ and $F^- \V_X(\nu)$ is nonzero
and contained in $\W$.
Similarly, each of $F^- \V_X(\mu)$ and $F^- \V_X(\nu)$ is nonzero
and contained in $\W$.
By these comments, each of $F^+ \W$ and $F^- \W$ is nonzero.
By the construction $\dim \W = 2$ and $F^+ \W \cap F^- \W =0$.
Now (i)--(v) follows from these comments.
\end{proof}

\begin{lemma}    \label{lem:dim}    \samepage
\ifDRAFT {\rm lem:dim}. \fi
Assume $k_0 \neq k_0^{-1}$.
Consider the reduced $X$-diagram of $\V$.
Then the dimension of $\V(k_0)$ (resp.\ $\V(k_0^{-1})$) is equal to the number
of single bonds plus the number of endvertices $\mu$
such that $\mu$ is incident to a double bond and $\V_X(\mu)$ is contained
in $\V(k_0)$ (resp.\ $\V(k_0^{-1})$).
\end{lemma}

\begin{proof}
Follows from Lemmas \ref{lem:endvertex} and \ref{lem:WF+WF-W}.
\end{proof}

\begin{lemma}    \label{lem:F+W0}    \samepage
\ifDRAFT {\rm lem:F+W0}. \fi
The following hold.
\begin{itemize}
\item[\rm (i)]
Assume that  $\V$ has $X$-type among $\text{\sf DS}$, $\text{\sf DDa}$, $\text{\sf DDb}$.
Then $\V_X(\mu_0)$ is contained in $\V(k_0)$.
\item[\rm (ii)]
Assume that $\V$ has $X$-type $\text{\sf DDa}$.
Then $\V_X(\mu_n)$ is contained in $\V(k_0)$.
\item[\rm (iii)]
Assume that $\V$ has $X$-type $\text{\sf DDb}$.
Then $\V_X(\mu_n)$ is contained in $\V(k_0^{-1})$.
\end{itemize}
\end{lemma}

\begin{proof}
Use Lemma \ref{lem:tiv0vn}. 
\end{proof}

\begin{lemma}   \label{lem:onlyone}    \samepage
\ifDRAFT {\rm lem:onlyone}. \fi
Assume $n \geq 1$ and $k_0 \neq k_0^{-1}$.
Then the following are equivalent:
\begin{itemize}
\item[\rm (i)]
$t_0$ has only one eigenvalue on $\V$;
\item[\rm (ii)]
$n=1$ and $\V$ has $X$-type {\sf DDa}.
\end{itemize}
\end{lemma}

\begin{proof}
(i)$\Rightarrow$(ii):
Assume that $t_0$ has only one eigenvalue on $\V$.
Then the reduced $X$-diagram contains no single bond;
otherwise each of $\V(k_0^{\pm 1})$ is nonzero by Lemma \ref{lem:WF+WF-W}.
So the reduced $X$-diagram is {\sf DD} and $n=1$.
By this and Lemma \ref{lem:F+W0} $\V$ has $X$-type {\sf DDa}.

(ii)$\Rightarrow$(i):
Assume that $n=1$ and $\V$ has $X$-type {\sf DDa}.
Note that $\V = \V_X(\mu_0) + \V_X(\mu_1)$.
By Lemma \ref{lem:F+W0}(i), (ii) each of $\V_X(\mu_0)$, $\V_X(\mu_1)$ is contained
in $\V(k_0)$.
Thus $t_0$ has only one eigenvalue $k_0$.
\end{proof}

For the rest of this section the following notation is in effect.

\begin{notation}   \label{notation2}   \samepage
\ifDRAFT {\rm notation2}. \fi
Assume that $t_0$ has two distinct eigenvalues on $\V$.
Define  $d=\dim \V(k_0)-1$ and $d'=\dim \V(k_0^{-1})-1$.
\end{notation}

\begin{lemma}   \label{lem:dimpre}    \samepage
\ifDRAFT {\rm lem:dimpre}. \fi
The $d-d'$ is as follows: 
\[
\begin{array}{c|c}
\textup{\rm $X$-type of $\V$} & d - d' 
\\ \hline    \rule{0mm}{3ex}
 \textup{\sf DS} & 1
\\     \rule{0mm}{2.5ex}
 \textup{\sf DDa} & 2
\\    \rule{0mm}{2.5ex}
 \textup{\sf DDb},  \textup{\sf SSa}, \textup{\sf SSb} & 0
\end{array}
\]
\end{lemma}

\begin{proof}
Follows from Lemmas \ref{lem:dim} and \ref{lem:F+W0}.
\end{proof}

\begin{corollary}            \label{cor:dim}
\ifDRAFT {\rm cor:dim}. \fi
In the table below we express $d$ and $d'$ in terms of $n$:
\[
\begin{array}{c|c|c}
\textup{\rm $X$-type of $\V$} & d  & d' 
\\ \hline    \rule{0mm}{3ex}
 \textup{\sf DS} & n/2 & (n-2)/2
\\     \rule{0mm}{2.5ex}
 \textup{\sf DDa} & (n+1)/2 & (n-3)/2
\\    \rule{0mm}{2.5ex}
 \textup{\sf DDb} & (n-1)/2 & (n-1)/2
\\    \rule{0mm}{2.5ex}
 \textup{\sf SSa}, \textup{\sf SSb} & (n-1)/2 & (n-1)/2
\end{array}
\]
\end{corollary}

\begin{proof}
Follows from Lemma \ref{lem:dimpre}.
\end{proof}

\section{Proof of Theorem \ref{thm:main1}}
\label{sec:proofmain1}

In this section we prove Theorem \ref{thm:main1}.
Until above Lemma \ref{lem:XDYD},  Notation \ref{notation} is in effect.
Referring to Notation \ref{notation2}, for each $\V(k_0^{\pm 1})$ 
we construct an eigenbasis of $\B$ on which
$\A$ acts in an irreducible tridiagonal manner. 

\begin{lemma}   \label{lem:wr}    \samepage
\ifDRAFT {\rm lem:wr}. \fi
With reference to Notation \ref{notation2},
there exist nonzero vectors $\{w_r\}_{r=0}^d$ in $\V(k_0)$ 
that satisfy the following conditions:
\begin{equation}                                        \label{eq:basiswr}
\begin{array}{c|l}
\text{\rm $X$-type of $\V$}  & \qquad \qquad\qquad  \text{\rm Conditions} 
\\ \hline
\text{\sf DS}  &     \rule{0mm}{3ex}
 w_0 \in \V_X(\mu_0)
\\
  & w_r \in \V_X(\mu_{2r-1}) + \V_X(\mu_{2r}) \qquad (1 \leq r \leq d)
\\ \hline
\text{\sf DDa}  &      \rule{0mm}{3ex}
  w_0 \in \V_X(\mu_0)
\\
 & w_r \in \V_X(\mu_{2r-1}) + \V_X(\mu_{2r}) \qquad (1 \leq r \leq d-1)
\\
 & w_d \in \V_X(\mu_{n})
\\ \hline
\text{\sf DDb}  &     \rule{0mm}{3ex}
  w_0 \in \V_X(\mu_0)
\\
 & w_r  \in \V_X(\mu_{2r-1}) + \V_X(\mu_{2r}) \qquad (1 \leq r \leq d)
\\ \hline
\text{\sf SSa}, \text{\sf SSb}   &     \rule{0mm}{3ex}
 w_r \in \V_X(\mu_{2r}) + \V_X(\mu_{2r+1})   \qquad (0 \leq r \leq d)
\end{array}
\end{equation}
\end{lemma}

\begin{proof}
Consider the reduced $X$-diagram of $\V$.
First assume that $\V$ has $X$-type {\sf DS}.
Clearly there exists a nonzero $w_0$ in $\V_X(\mu_0)$.
By Lemma \ref{lem:F+W0}(i) $w_0$ is contained in $\V(k_0)$.
Pick any $r$ $(1 \leq r \leq d)$.
Observe that $\mu_{2r-1}$, $\mu_{2r}$ is a single bond by the shape of the $X$-diagram.
Define $\W = \V_X(\mu_{2r-1}) + \V_X(\mu_{2r})$.
By Lemma \ref{lem:WF+WF-W}(ii), (iii) $\W \cap \V(k_0)$ has dimension one.
Thus there exists a nonzero $w_r \in \V(k_0)$ that is contained in
$\V_X(\mu_{2r-1}) + \V_X(\mu_{2r})$.
We have shown the result for the case {\sf DS}.
The proof is similar for the other cases.
\end{proof}

\begin{lemma}   \label{lem:wdr}    \samepage
\ifDRAFT {\rm lem:wdr}. \fi
With reference to Notation \ref{notation2},
there exist nonzero vectors $\{w'_r\}_{r=0}^d$ in $\V(k_0^{-1})$ 
that satisfy the following conditions:
\begin{equation}                                        \label{eq:basiswdr}
\begin{array}{c|l}
\text{\rm $X$-type of $\V$}  & \qquad \qquad\qquad  \text{\rm Conditions} 
\\ \hline
\text{\sf DS}  &     \rule{0mm}{3ex}
  w'_r \in \V_X(\mu_{2r+1}) + \V_X(\mu_{2r+2}) \qquad (0 \leq r \leq d')
\\ \hline
\text{\sf DDa}  &      \rule{0mm}{3ex}
 w'_r \in \V_X(\mu_{2r+1}) + \V_X(\mu_{2r+2}) \qquad (0 \leq r \leq d')
\\ \hline
\text{\sf DDb}  &     \rule{0mm}{3ex}
  w'_r  \in \V_X(\mu_{2r+1}) + \V_X(\mu_{2r+2}) \qquad (0 \leq r \leq d'-1)
\\
 & w'_{d'} \in \V_X(\mu_n)
\\ \hline
\text{\sf SSa}, \text{\sf SSb}   &     \rule{0mm}{3ex}
 w'_r \in \V_X(\mu_{2r}) + \V_X(\mu_{2r+1})   \qquad (0 \leq r \leq d')
\end{array}
\end{equation}
\end{lemma}

\begin{proof}
Similar to the proof of Lemma \ref{lem:wr}.
\end{proof}

\begin{lemma}   \label{lem:basiswr}    \samepage
\ifDRAFT {\rm lem:basiswr}. \fi
With reference to Lemmas \ref{lem:wr} and \ref{lem:wdr},
the vectors $\{w_r\}_{r=0}^d$ (resp.\ $\{w'_r\}_{r=0}^{d'}$) form a basis for
$\V(k_0)$ (resp.\ $\V(k_0^{-1})$).
\end{lemma}

\begin{proof}
First assume that $\V$ has $X$-type {\sf DS}.
To simplify notation, set $\W_0 = \V_X(\mu_0)$ and
$\W_r = \V_X(\mu_{r-1}) + \V_X(\mu)$ for $1 \leq r \leq d$.
Observe that $\V = \sum_{r=0}^d \W_r$ (direct sum).
By the conditions in Lemma \ref{lem:wr} $w_r \in \W_r \cap \V(k_0)$
for $0 \leq r \leq d$.
By Lemma \ref{lem:F+W0}(i) and 
Lemma \ref{lem:WF+WF-W}(iii) $\W_r \cap \V(k_0) = F^+ \W_r$
for $0 \leq r \leq d$. 
By these comments $w_r \in F^+ \W_r$ for $0 \leq r \leq d$.
By this and $F^+ \V = \sum_{r=0}^d F^+ \W_r$ (direct sum),
the vectors $\{w_r\}_{r=0}^d$ form a basis for $F^+ \V = \V(k_0)$.
Similarly $\{w'_r\}_{r=0}^{d'}$ form a basis for $\V(k_0^{-1})$.
We have shown the result for the case {\sf DS}.
The proof is similar for the other cases.
\end{proof}

Note that each of $\A$ and $\B$ commutes with $t_0$ by
Lemma \ref{lem:titj}(ii), and so
each of $\V(k_0)$ and $\V(k_0^{-1})$ is invariant under $\A$, $\B$.
We first consider the action of $\B$ on $\V(k_0^{\pm 1})$.

\begin{lemma}    \label{lem:ABw}    \samepage
\ifDRAFT {\rm lem:ABw}. \fi
The following hold.
\begin{itemize}
\item[\rm (i)]
Let $\mu$ denote a vertex in the $X$-diagram of $\V$.
Then $\V_X(\mu)$ is contained in the eigenspace of $\B$ with eigenvalue $\mu + \mu^{-1}$.
\item[\rm (ii)]
Let $\mu, \nu$ denote a single bond in the reduced $X$-diagram of $\V$.
Then $\V_X(\mu) + \V_X(\nu)$ is contained in the eigenspace of $\B$
with eigenvalue $\mu + \mu^{-1}$.
\end{itemize}
\end{lemma}

\begin{proof}
(i):
Clear by $\B = X + X^{-1}$.

(ii):
We have $\nu = \mu^{-1}$ since $\mu$, $\nu$ are $1$-adjacent.
So each of $\V_X(\mu)$ and $\V_X(\nu)$ is contained in the eigenspace of $\B$
for the eigenvalue $\mu + \mu^{-1}$.
\end{proof}

\begin{lemma}   \label{lem:Baction}    \samepage
\ifDRAFT {\rm lem:Baction}. \fi	
With reference to Notation \ref{notation2}, let the bases
 $\{w_r\}_{r=0}^d$ and $\{w'_r\}_{r=0}^{d'}$ be from Lemma \ref{lem:basiswr}.
Then the vectors $\{w_r\}_{r=0}^d$ and $\{w'_r\}_{r=0}^{d'}$ 
are eigenvectors of $\B$ with the following eigenvalues:
\[
 \begin{array}{c|c|c}
   \textup{\rm $X$-type of $\V$} 
   & \text{\rm Eigenvalue of $\B$ for $w_r$}
   & \text{\rm Eigenvalue of $\B$ for $w'_r$}
\\ \hline
  \textup{\sf DS}, \textup{\sf DDa}, \textup{\sf DDb}   \rule{0mm}{3ex}
 & k_0 k_3 q^{2r} + \frac{1}{k_0 k_3 q^{2r}}
 & k_0 k_3 q^{2r+2} + \frac{1}{k_0 k_3 q^{2r+2}}
\\
  \rule{0mm}{3ex}
 \textup{\sf SSa}, \textup{\sf SSb}
 & k_1 k_2 q^{2r+1} + \frac{1}{k_1 k_2 q^{2r+1}}
 & k_1 k_2 q^{2r+1} + \frac{1}{k_1 k_2 q^{2r+1}}
\end{array}
\]
\end{lemma}

\begin{proof}
Follows from Lemmas \ref{lem:eigenvalues} and \ref{lem:ABw}.
\end{proof}

Next we consider the action of $\A$ on $\V(k_0^{\pm 1})$.

\begin{lemma} \label{lem:AF+v}  \samepage
\ifDRAFT {\rm lem:AF+v}. \fi
Assume $k_0 \neq k_0^{-1}$. Then the following hold on $\V$:
\begin{align}
 \A F^+  
 &= (k_0-k_0^{-1})F^+ t_1 F^+  + k_0^{-1}(k_1+k_1^{-1}) F^+,    \label{eq:AF+v} \\
 \A F^- 
 &= (k_0^{-1}-k_0)F^- t_1 F^-  + k_0 (k_1+k_1^{-1}) F^- .       \label{eq:AF-v}
\end{align}
\end{lemma}

\begin{proof}
We first show \eqref{eq:AF+v}.
We have $t_0^{-1} F^+  = k_0^{-1} F^+$ on $\V$ since $F^+$ acts on $\V$
as the projection onto $\V(k_0)$.
By \eqref{eq:ki} we have $T_1 = k_1 + k_1^{-1}$ on $\V$.
Using these comments  and $t_1^{-1}=T_1-t_1$ we argue on $\V$
\begin{align*}
 \A F^+ 
 &= (t_0t_1 + t_1^{-1}t_0^{-1}) F^+   \\
 &= t_0t_1 F^+  + k_0^{-1} (T_1 - t_1)F^+  \\
 &= t_0t_1 F^+  + k_0^{-1} (k_1+k_1^{-1}) F^+  - k_0^{-1} t_1 F^+  \\
 &= (t_0 - k_0^{-1}) t_1 F^+  + k_0^{-1} (k_1+k_1^{-1}) F^+.
\end{align*}
Now \eqref{eq:AF+v} follows since $t_0-k_0^{-1} = (k_0-k_0^{-1}) F^+$.
The proof of \eqref{eq:AF-v} is similar.
\end{proof}

\begin{lemma}    \label{lem:F+v}    \samepage
\ifDRAFT {\rm lem:F+v}. \fi
Assume $k_0 \neq k_0^{-1}$.
Let $\mu$, $\nu$ denote a single bond in the reduced $X$-diagram of $\V$.
Then for $0 \neq v \in \V_X(\mu)$ the following hold.
\begin{itemize}
\item[\rm (i)]
There exist nonzero $a$, $b \in \F$ such that
$F^+ v = a v + b G_0 v$.
\item[\rm (ii)]
There exist nonzero $a$, $b \in \F$ such that
$F^- v = a v + b G_0 v$.
\end{itemize}
\end{lemma}

\begin{proof}
(i):
Note that $v$ is a basis for $\V_X(\mu)$.
By Lemma \ref{lem:G0G2V}(i) $G_0 v$ is a basis for $\V_X(\nu)$.
By Lemma \ref{lem:F+F-} $F^+ v$ is nonzero and contained in $\V_X(\mu) + \V_X(\nu)$.
By these comments there exist  $a$, $b \in \F$ such that
$F^+ v = a v + b G_0 v$.
We show $b \neq 0$.
By way of contradiction, assume $b=0$.
Then $F^+ v$ is contained in $\V_X(\mu)$, and so
$F^+ v$ is a basis for $\V_X(\mu)$.
This forces $F^- \V_X(\mu)=0$;
this is a contradiction since $F^- v \neq 0$ by Lemma \ref{lem:F+F-}.
Thus $b \neq 0$.
We can show $a \neq 0$ in a similar way.

(ii): Similar.
\end{proof}

\begin{lemma}    \label{lem:t1v}    \samepage
\ifDRAFT {\rm lem:t1v}. \fi
Let $\mu$ denote an eigenvalue of $X$ and let  $0 \neq v \in \V_X(\mu)$.
Then there exist $e$, $f \in \F$ with $f \neq 0$ such that
\[
  t_1 v = e v + f G_2 v.
\]
\end{lemma}

\begin{proof}
We may assume $n \geq 1$; otherwise the result follows since
$G_2 v = 0$ by Lemma \ref{lem:G20}(i).
First assume that $\mu$ is not incident to a double bond in the reduced $X$-diagram of $\V$.
Then $\mu$ is an endvertex of the reduced $X$-diagram of $\V$ that is incident to a single bond.
By Lemma \ref{lem:endvertex} $\V_X(\mu)$ is invariant under $t_1$.
By Lemma \ref{lem:G20}(i) $G_2 v=0$.
The result follows from these comments.
Next assume that $\mu$ is incident to a double bond in the reduced $X$-diagram of $\V$.
Then the result follows from Lemma \ref{lem:t1t2action}.
\end{proof}

\begin{lemma} \label{lem:F+t1F+v} \samepage
\ifDRAFT {\rm lem:F+t1F+v}. \fi
Assume $k_0 \neq k_0^{-1}$.
Let $\mu$ denote an eigenvalue of $X$ and let $0 \neq v \in \V_X(\mu)$.
\begin{itemize}
\item[\rm (i)]
Assume $F^+ v \neq 0$.
Then there exist  $\alpha$, $\beta$, $\gamma \in \F$ 
with $\beta \gamma \neq 0$ such that
\[
  F^+ t_1 F^+ v = \alpha F^+ v + \beta F^+ G_2 v + \gamma F^+ G_2G_0v.
\]
\item[\rm (ii)]
Assume $F^- v \neq 0$.
Then there exist $\alpha$, $\beta$, $\gamma \in \F$ 
with $\beta \gamma \neq 0$ such that
\[
  F^- t_1 F^- v = \alpha F^- v + \beta F^- G_2 v + \gamma F^- G_2G_0v.
\]
\end{itemize}
\end{lemma}

\begin{proof}
(i):
We assume  $n \geq 1$;
otherwise the assertion holds since $G_0 v=0$ and $G_2 v=0$ by Lemma \ref{lem:G20}.
By Lemma \ref{lem:t1v} there exist $e$, $f \in \F$ 
with $f \neq 0 $ such that
\[
  t_1 v = e v + f G_2 v.
\]
In this equation, apply $F^+$ to each side to get
\begin{equation}   \label{eq:F+t1F+vaux1}
 F^+ t_1 v = e F^+ v + f F^+ G_2 v.
\end{equation}
Similarly, there exist  $g$, $h \in \F$ with $h \neq 0 $ such that
\begin{equation}   \label{eq:F+t1F+vaux2}
 F^+ t_1 G_0 v = g F^+ G_0 v + h F^+ G_2 G_0 v.
\end{equation}
First assume that $\mu$ is incident to a single bond in the reduced $X$-diagram of $\V$.
By Lemma \ref{lem:F+v}(i) there exist nonzero $a$, $b\in \F$
such that
\[
  F^+ v = a v + b G_0 v.
\]
In this equation, apply $F^+ t_1$ to each side to get
\begin{equation}   \label{eq:F+t1F+vaux3}
  F^+ t_1 F^+ v = a F^+ t_1 v + b F^+ t_1 G_0 v.
\end{equation}
Combining \eqref{eq:F+t1F+vaux1}--\eqref{eq:F+t1F+vaux3}
\[
 F^+ t_1 F^+ v
  = a e F^+ v + b g F^+ G_0 v + af F^+ G_2 v + b h F^+ G_2G_0 v.
\]
The result follows since $af \neq 0$, $bh \neq 0$, and 
$F^+ G_0 v \in \text{Span}\{F^+ v\}$ by 
Lemma \ref{lem:WF+WF-W}(iv).
Next assume that $\mu$ is not incident to a single bond in the reduced $X$-diagram of $\V$.
Then $\mu$ is an endvertex that is incident to a double bond.
By Lemma \ref{lem:G20}(ii) $G_0 v = 0$.
By Lemma \ref{lem:endvertex} $v$ is an eigenvector for $t_0$.
So either $t_0 v = k_0 v$ or $t_0 v = k_0^{-1} v$.
By this and $F^+ v \neq 0$ we get $F^+ v = v$.
By this and \eqref{eq:F+t1F+vaux1},
\[
    F^+ t_1 F^+ v = e F^+ v + f F^+ G_2 v.
\]
The result follows since $G_0 v=0$ and $f \neq 0$.

(ii):
Similar.
\end{proof}

\begin{lemma} \label{lem:AF+vAF-v} \samepage
\ifDRAFT {\rm lem:AF+vAF-v}. \fi
Assume $k_0 \neq k_0^{-1}$.
Let $\mu$ denote an eigenvalue of $X$ and let $0 \neq v \in \V_X(\mu)$.
\begin{itemize}
\item[\rm (i)]
Assume $F^+ v \neq 0$.
Then there exist  $\alpha$, $\beta$, $\gamma \in \F$ 
with $\beta \gamma \neq 0$ such that
\[
 \A F^+ v = \alpha F^+ v + \beta F^+ G_2 v + \gamma F^+ G_2G_0v.
\]
\item[\rm (ii)]
Assume $F^- v \neq 0$.
Then there exist  $\alpha$, $\beta$, $\gamma \in \F$ 
with $\beta \gamma \neq 0$ such that
\[
 \A F^- v = \alpha F^- v + \beta F^- G_2 v + \gamma F^- G_2G_0v.
\]
\end{itemize}
\end{lemma}

\begin{proof}
Follows from Lemmas  \ref{lem:AF+v} and \ref{lem:F+t1F+v}.
\end{proof}

\begin{lemma}   \label{lem:Aaction}    \samepage
\ifDRAFT {\rm lem:Aaction}. \fi
With reference to Notation \ref{notation2},
let the bases $\{w_r\}_{r=0}^d$ and $\{w'_r\}_{r=0}^{d'}$ be from
Lemma \ref{lem:basiswr}.
Then with respect to the basis $\{w_r\}_{r=0}^d$ (resp.\ $\{w'_r\}_{r=0}^{d'}$) 
the matrix representing $\A$ is irreducible tridiagonal.
\end{lemma}

\begin{proof}
First assume that $\V$ has $X$-type \textup{\sf DS}.
Note that $d = n/2$ by Corollary \ref{cor:dim}.
We assume $d \geq 1$; otherwise the assertion is obviously true.
We abbreviate $\V_r = \V_X(\mu_r)$ for $0 \leq r \leq n$.
Define subspaces
\begin{align*}
  \W_0 &= \V_0,   &
  \W_r &= \V_{2r-1} + \V_{2r}   \qquad (1 \leq r \leq d).
\end{align*}
By the construction, $w_r \in F^+ \W_r$ for $0 \leq r \leq d$.
First we consider the action of $\A$ on $w_0$.
By the construction $F^+ w_0 = w_0 \neq 0$.
By Lemma \ref{lem:G20}(ii) $G_0 w_0 = 0$.
By these comments and Lemma \ref{lem:AF+vAF-v}(i),
there exist  $\alpha$, $\beta \in \F$ such that
\begin{align}
  \A w_0 &= \alpha w_0 + \beta F^+ G_2 w_0,
  & \beta & \neq 0.                                               \label{eq:Aw0no1}
\end{align}
By Lemma \ref{lem:G0G2V}(ii) $G_2 \V_0 = \V_1$,
and so $0 \neq G_2 w_0 \in \V_1$.
Applying Lemma \ref{lem:WF+WF-W}(iv)
for $\mu = \mu_1$ and $\nu = \mu_2$,
we find that
$0 \neq F^+ G_2 w_0 \in \F w_1$.
So there exists a nonzero $\beta' \in \F$ such that
$F^+ G_2 w_0 = \beta' w_1$.
By this and \eqref{eq:Aw0no1}
\begin{align}
  \A w_0 &= \alpha w_0 + \beta \beta' w_1,    
           & \beta \beta' &\neq 0.                                    \label{eq:Aw0no2}
\end{align}
Next we consider the action of $\A$ on $w_r$ for $1 \leq r \leq d-1$.
Pick any $r$ such that $1 \leq r \leq d-1$.
By Lemma \ref{lem:WF+WF-W}(iv) there exists $v \in \V_{2r-1}$ such that
$w_r = F^+ v$.
By Lemma \ref{lem:AF+vAF-v}(i)
there exist  $\alpha$, $\beta$, $\gamma \in \F$ such that
\begin{align}
 \A F^+ v &= \alpha F^+ v + \beta F^+ G_2 v + \gamma F^+ G_2 G_0 v,
 && \beta \gamma \neq 0.                                        \label{eq:Awrno1}
\end{align}
By Lemma \ref{lem:G0G2V}
\begin{align*}
  G_2 \V_{2r-1} &= \V_{2r-2},  &
  G_2 G_0 \V_{2r-1} &= \V_{2r+1}.
\end{align*}
So
$0 \neq G_2 v \in \V_{2r-2}$ and $0 \neq G_2 G_0 v \in \V_{2r+1}$.
By Lemma \ref{lem:WF+WF-W}(iv) (or by $F^+ \W_0 = \V_0$ if $r=1$)
$F^+ \W_{r-1} = F^+ \V_{2r-2}$.
So $0 \neq F^+ G_2 v \in F^+ \W_{r-1}$.
Thus $F^+ G_2 v = \beta' w_{r-1}$ for some nonzero $\beta' \in \F$.
Similarly $F^+ \W_{r+1} = F^+ \V_{2r+1}$, and so
$0 \neq F^+ G_2 G_0 v \in F^+ \W_{r+1}$.
Thus $F^+ G_2 G_0 v = \gamma' w_{r+1}$ for some nonzero $\gamma' \in \F$.
By these comments and \eqref{eq:Awrno1} 
\begin{align}
 \A w_r &= \beta \beta' w_{r-1} + \alpha w_r + \gamma \gamma' w_{r+1},
 && \beta \beta' \neq 0, \quad \gamma \gamma' \neq 0.                                        \label{eq:Awrno2}
\end{align}
In a similar way as above, we can show that there exist
 $\alpha$, $\beta \in \F$ such that
\begin{align}
 \A w_d &= \beta w_{r-1} + \alpha w_r,   && \beta \neq 0.                        \label{eq:Awd}
\end{align}
By \eqref{eq:Aw0no2}, \eqref{eq:Awrno2}, \eqref{eq:Awd}
we find that the matrix representing $\A$ with respect to $\{w_r\}_{r=0}^d$
is irreducible tridiagonal.
In a similar way, 
we find that  the matrix representing $\A$ with respect to $\{w'_r\}_{r=0}^{d'}$
is irreducible tridiagonal.
We have shown the result when $\V$ has $X$-type \textup{\sf DS}.
The proof is similar for the other types.
\end{proof}

\begin{proposition}    \label{prop:ABaction}    \samepage
\ifDRAFT {\rm prop:ABaction}. \fi
With reference to Notation \ref{notation2},
the following hold.
\begin{itemize}
\item[\rm (i)]
There exists a basis for $\V(k_0)$
with respect to which the matrix representing $\A$ is irreducible tridiagonal
and the matrix representing $\B$ is diagonal with the following $(r,r)$-entry
for $0 \leq r \leq d$:
\[
 \begin{array}{c|c}
  \textup{\rm $X$-type of $\V$} & \text{\rm $(r,r)$-entry for $\B$}
\\ \hline
  \textup{\sf DS}, \textup{\sf DDa}, \textup{\sf DDb} &   \rule{0mm}{3ex}
  k_0 k_3 q^{2r} + \frac{1}{k_0 k_3 q^{2r} }
\\
  \textup{\sf SSa}, \textup{\sf SSb} &   \rule{0mm}{2.5ex}
  k_1 k_2 q^{2r+1} + \frac{1}{k_1 k_2 q^{2r+1} }
 \end{array}
\]
\item[\rm (ii)]
There exists a basis for $\V(k_0^{-1})$ with respect to which the matrix
representing $\A$ is irreducible tridiagonal
and the matrix representing $\B$ is diagonal with the following $(r,r)$-entry
for $0 \leq r \leq d'$:
\[
 \begin{array}{c|c}
  \textup{\rm $X$-type of $\V$} & \text{\rm $(r,r)$-entry for $\B$}
\\ \hline
  \textup{\sf DS}, \textup{\sf DDa}, \textup{\sf DDb} &   \rule{0mm}{3ex}
  k_0 k_3 q^{2r+2} + \frac{1}{k_0 k_3 q^{2r+2} }
\\
  \textup{\sf SSa}, \textup{\sf SSb} &  \rule{0mm}{2.5ex}
  k_1 k_2 q^{2r+1} + \frac{1}{k_1 k_2 q^{2r+1} }
 \end{array}
\]
\end{itemize}
\end{proposition}

\begin{proof}
Follows from Lemmas \ref{lem:Baction} and \ref{lem:Aaction}.
\end{proof}

\begin{lemma}    \label{lem:XDYD}    \samepage
\ifDRAFT {\rm lem:XDYD}. \fi
Let $\V$ denote a finite dimensional irreducible $\Hq$-module with parameter sequence
$\{k_i\}_{i \in \I}$.
Assume that $t_0$ has two distinct eigenvalues on $\V$.
\begin{itemize}
\item[\rm (i)]
Assume that $\V$ is XD.
Then for each of $\V(k_0^{\pm 1})$ there exists a basis
with respect to which the matrix representing $\A$ is irreducible tridiagonal and the matrix
representing $\B$ is diagonal whose diagonal entries form
a $q$-Racah sequence.
\item[\rm (ii)]
Assume that $\V$ is YD.
Then  for each of $\V(k_0^{\pm 1})$ there exists a basis
with respect to which the matrix representing $\B$ is irreducible tridiagonal and the matrix
representing $\A$ is diagonal whose diagonal entries form a $q$-Racah sequence.
\end{itemize}
\end{lemma}

\begin{proof}
(i):
We may assume that $\{k_i\}_{i \in \I}$ is consistent with a standard ordering
of the eigenvalues of $X$ on $\V$, by replacing any of $\{k_i\}_{i \in \I}$ with its inverse if necessary.
Now the result follows from Proposition \ref{prop:ABaction}.

(ii):
Recall the automorphism $\sigma$ of $\Hq$ from Lemma \ref{lem:sigma}.
By the definition of $\sigma$, $t_0$ is fixed by $\sigma$.
By Lemma \ref{lem:sigmaXY} $\sigma$ sends  $Y \mapsto X$ and swaps $\A$, $\B$.
Consider the $\Hq$-module $\V^\sigma$, where $\V^\sigma= \V$ as sets,
and the action of $x \in \Hq$ in $\V^\sigma$ is equal to the action of  $x^\sigma$ in $\V$.
Observe that the $\Hq$-module $\V^\sigma$ is XD.
Applying (i) to $\V^\sigma$ we obtain the result.
\end{proof} 

\begin{proofof}{Theorem \ref{thm:main1}}
Follows from Lemma \ref{lem:XDYD}.
\end{proofof}

\section{The parameter sequence of an XD $\Hq$-module}
\label{sec:restrictions}

Throughout this section Notation \ref{notation} is in effect.
In this section we investigate the parameter sequence of $\V$.

\begin{lemma}    \label{lem:typekr0}    \samepage
\ifDRAFT {\rm lem:typekr0}. \fi
The parameters $\{k_i\}_{i \in \I}$ satisfy the following equation:
\[
\begin{array}{c|c}
\text{\rm $X$-type of $\V$}  & \text{\rm Equation}
\\ \hline 
\textup{\sf DS}  & k_0 k_1 k_2 k_3 = q^{-n-1}   \rule{0mm}{3ex}
\\
 \textup{\sf DDa} &   k_0^2 = q^{-n-1}   \rule{0mm}{2.5ex}
\\
 \textup{\sf DDb} &   k_3^2 = q^{-n-1}  \rule{0mm}{2.5ex}
\\
 \textup{\sf SSa} &   k_1^2 = q^{-n-1}   \rule{0mm}{2.5ex}
\\
 \textup{\sf SSb} &   k_2^2 = q^{-n-1}  \rule{0mm}{2.5ex}
\end{array}
\]
\end{lemma}

\begin{proof}
Pick nonzero vectors $v_0 \in \V_X(\mu_0)$ and $v_n \in \V_X(\mu_n)$.
First assume that $\V$ has $X$-type \textup{\sf DS}.
By Lemma \ref{lem:tiv0vn} $t_1 v_n = k_1 v_n$ and $t_2 v_n = k_2 v_n$.
By this and $X = q^{-1} t_2^{-1} t_1^{-1}$ we obtain
$X v_n = q^{-1} k_1^{-1} k_2^{-1} v_n$.
So $\mu_n =  q^{-1} k_1^{-1} k_2^{-1}$.
We have $\mu_n = k_0 k_3 q^n$ by Lemma \ref{lem:eigenvalues}(i).
Comparing these two values of $\mu_n$, we obtain
$k_0 k_1 k_2 k_3 = q^{-n-1}$.
Next assume that $\V$ has $X$-type \textup{\sf DDa}.
By Lemma \ref{lem:tiv0vn} $t_0 v_n = k_0 v_n$ and $t_3 v_n = k_3^{-1} v_n$.
So $X v_n = k_0 k_3^{-1} v_n$, and hence $\mu_n = k_0 k_3^{-1}$.
By Lemma \ref{lem:eigenvalues}(i) $\mu_n = (k_0 k_3 q^{n+1})^{-1}$.
Comparing these two values of $\mu_n$,
we obtain $k_0^2 = q^{-n-1}$.
Next assume that $\V$ has $X$-type \textup{\sf DDb}.
We obtain $k_3^2 = q^{-n-1}$ in a similar way.
Next assume that $\V$ has $X$-type \textup{\sf SSa}.
By Lemma \ref{lem:tiv0vn} $t_1 v_n = k_1 v_n$ and $t_2 v_n = k_2^{-1} v_n$.
By this and $X = q^{-1} t_2^{-1} t_1^{-1}$ we obtain $X v_n = q^{-1} k_1^{-1} k_2 v_n$,
and hence $\mu_n = q^{-1} k_1^{-1} k_2$.
By Lemma \ref{lem:eigenvalues}(ii) $\mu_n = k_1 k_2 q^n$.
Comparing these two values of $\mu_n$ we obtain $k_1^2 = q^{-n-1}$.
Next assume that $\V$ has $X$-type \textup{\sf SSb}.
We obtain $k_2^2 = q^{-n-1}$ in a similar way.
\end{proof}

\begin{lemma}    \label{lem:typekrpre1}   \samepage
\ifDRAFT {\rm lem:typekrpre1}. \fi
The parameters $\{k_i\}_{i \in \I}$ satisfy the following inequalities:
\[
\begin{array}{c|l}
\text{\rm $X$-type of $\V$} & \hspace{4cm} \text{\rm Inequalities}
\\ \hline  
\textup{\sf DS}, \textup{\sf DDa}, \textup{DDb} &
    \text{\rm $k_0^2 k_3^2$ is not among $q^{-2},q^{-4}, q^{-6},\ldots, q^{-2n}$}    \rule{0mm}{2.8ex}
\\ \hline
\textup{\sf SSa}, \textup{SSb}  &
    \text{\rm $k_1^2 k_2^2$ is not among $q^{-2},q^{-4},q^{-6},\ldots, q^{-2n}$}   \rule{0mm}{2.8ex}
\end{array}
\]
\end{lemma}

\begin{proof}
First assume that $\V$ has $X$-type \textup{\sf DS}.
Since $\{\mu_r\}_{r=0}^n$ are mutually distinct, 
$\mu_{2r} - \mu_{2s-1} \neq 0$ for $0 \leq r \leq n/2$ and $1 \leq s \leq n/2$.
By Lemma \ref{lem:eigenvalues}(i) 
\[
  \mu_{2r} - \mu_{2s-1} = k_0 k_3 q^{2r} - \frac{1}{k_0 k_3 q^{2s}}
   = \frac{k_0^2 k_3^2 q^{2(r+s)} -1}{k_0 k_3 q^{2s}}.
\]
By these comments $k_0^2 k_3^2 q^{2(r+s)} \neq 1$ for $0 \leq r \leq n/2$ and $1 \leq s \leq n/2$.
So $k_0^2 k_3^2$ is not among $q^{-2}$, $q^{-4}$, \ldots, $q^{-2n}$.
Next assume that $\V$ has $X$-type \textup{\sf DDa} or \textup{\sf DDb}.
We have
$\mu_{2r} - \mu_{2s-1} \neq 0$ for $0 \leq r \leq (n-1)/2$ and $1 \leq s \leq (n+1)/2$.
In a similar way as above, we find that $k_0^2 k_3^2$ is not among $q^{-2}$, $q^{-4}$, \ldots, $q^{-2n}$.
Next assume that $\V$ has $X$-type \textup{\sf SSa} or \textup{SSb}.
We have
$\mu_{2r} - \mu_{2s-1} \neq 0$ for $0 \leq r \leq (n-1)/2$ and $1 \leq s \leq (n+1)/2$.
By Lemma \ref{lem:eigenvalues}(ii) 
\[
  \mu_{2r} - \mu_{2s-1} = \frac{1}{k_1 k_2 q^{2r+1}} - k_1 k_2 q^{2s-1}
   = \frac{1 - k_1^2 k_2^2 q^{2(r+s)}}{k_1 k_2 q^{2r+1}}.
\]
By these comments $k_1^2 k_2^2 q^{2(r+s)} \neq 1$ for $0 \leq r \leq (n-1)/2$ and $1 \leq s \leq (n+1)/2$.
So $k_0^2 k_3^2$ is not among $q^{-2}$, $q^{-4}$, \ldots, $q^{-2n}$.
\end{proof}

\begin{lemma}  \label{lem:typekrpre2}   \samepage
\ifDRAFT {\rm lem:typekrpre2}. \fi
The parameters $\{k_i\}_{i \in \I}$ satisfy the following inequalities:
\[
\begin{array}{c|l}
\text{\rm $X$-type of $\V$} & \hspace{4cm} \text{\rm Inequalities}
\\ \hline
\textup{\sf DS} &
  \begin{array}{l}
    \rule{0mm}{4.3mm}
    \text{\rm Neither of  $k_0^2$, $k_3^2$ is among $q^{-2},q^{-4},q^{-6},\ldots,q^{-n}$ } \\
    \text{\rm None of  $k_0 k_3 k_1^{\pm 1} k_2 ^{\pm 1}$ is among $q^{-1},q^{-3},q^{-5},\ldots,q^{1-n}$}
  \end{array}
\\ \hline
\textup{\sf DDa}, \textup{\sf DDb} &
 \begin{array}{l}
   \rule{0mm}{4.3mm}
   \text{\rm None of $k_0k_3k_1^{\pm 1}k_2^{\pm 1}$ is among $q^{-1},q^{-3},q^{-5},\ldots,q^{-n}$}
 \end{array}
\\ \hline
\textup{\sf SSa}, \textup{\sf SSb} &
 \begin{array}{l}
   \rule{0mm}{4.3mm}
   \text{\rm None of $k_1k_2k_0^{\pm 1}k_3^{\pm 1}$ is among $q^{-1},q^{-3},q^{-5},\ldots,q^{-n}$}
 \end{array}
\end{array}
\]
\end{lemma}

\begin{proof}
First assume that $\V$ has $X$-type \textup{\sf DS}.
By Corollary \ref{cor:nonzero}(i) $\mu_{2r-1}$ is not among $k_0^{-1}k_3$, $k_0 k_3^{-1}$ for $1 \leq r \leq n/2$.
By Lemma \ref{lem:eigenvalues}(i) $\mu_{2r-1} = (k_0 k_3 q^{2r})^{-1}$.
By these comments neither of  $k_0^2$, $k_3^2$ is among $q^{-2}$, $q^{-4}$, \ldots, $q^{-n}$.
By Corollary \ref{cor:nonzero}(ii) $q \mu_{2r}$ is not among $k_1^{\pm 1} k_2^{\pm 1}$ for $0 \leq r \leq (n-2)/2$.
By Lemma \ref{lem:eigenvalues}(i) $\mu_{2r} = k_0 k_3 q^{2r}$.
By these comments none of  $k_0 k_3 k_1^{\pm 1} k_2^{\pm 1}$ is among 
$q^{-1}$, $q^{-3}$, \ldots, $q^{1-n}$.
Next assume that $\V$ has $X$-type \textup{\sf DDa} or \textup{\sf DDb}.
By Corollary \ref{cor:nonzero}(ii) $q \mu_{2r}$ is not among $k_1^{\pm 1} k_2^{\pm 1}$ for $0 \leq r \leq (n-1)/2$.
By Lemma \ref{lem:eigenvalues}(i) $\mu_{2r} = k_0 k_3 q^{2r}$.
By these comments none of $k_0 k_3 k_1^{\pm 1} k_2^{\pm 1}$ is among 
$q^{-1}$, $q^{-3}$, \ldots, $q^{-n}$.
Next assume that $\V$ has $X$-type \textup{\sf SSa} or \textup{\sf SSb}.
By Corollary \ref{cor:nonzero}(i) $\mu_{2r}$ is not among $k_0^{\pm 1} k_3^{\pm 1}$ for $0 \leq r \leq (n-1)/2$.
By Lemma \ref{lem:eigenvalues}(ii) $\mu_{2r} = (k_1 k_2 q^{2r+1})^{-1}$.
By these comments none of $k_1 k_2 k_0^{\pm 1} k_3^{\pm 1}$ is among
$q^{-1}$, $q^{-3}$, \ldots, $q^{-n}$.
\end{proof}

\begin{lemma}  \label{lem:typekr}   \samepage
\ifDRAFT {\rm lem:typekr}. \fi
The parameters $\{k_i\}_{i \in \I}$ satisfy the following inequalities:
\[
\begin{array}{c|l}
\text{\rm $X$-type of $\V$} & \hspace{4cm} \text{\rm Inequalities}
\\ \hline
\textup{\sf DS} &
  \begin{array}{l}
    \rule{0mm}{4.3mm}
    \text{\rm Neither of $\pm k_0 k_3$ is among $q^{-1},q^{-2},q^{-3},\ldots, q^{-n}$} \\
    \text{\rm None of $\pm k_0$, $\pm k_1$, $\pm k_2$, $\pm k_3$ is among $q^{-1},q^{-2},q^{-3},\ldots,q^{-n/2}$}
  \end{array}
\\ \hline
\textup{\sf DDa} &
 \begin{array}{l}
   \rule{0mm}{4.3mm}
   \text{\rm None of $\pm k_3^{\pm 1}$ is among $1,q,q^2,\ldots, q^{(n-1)/2}$} \\
   \text{\rm None of $k_0k_3k_1^{\pm 1}k_2^{\pm 1}$ is among $q^{-1},q^{-3},q^{-5},\ldots,q^{-n}$}
 \end{array}
\\ \hline
\textup{\sf DDb} &
 \begin{array}{l}
   \rule{0mm}{4.3mm}
   \text{\rm None of $\pm k_0^{\pm 1}$ is among $1,q,q^2,\ldots, q^{(n-1)/2}$} \\
   \text{\rm None of $k_0k_3k_1^{\pm 1}k_2^{\pm 1}$ is among $q^{-1},q^{-3},q^{-5},\ldots,q^{-n}$}
 \end{array}
\\ \hline 
\textup{\sf SSa} &
 \begin{array}{l}
   \rule{0mm}{4.3mm}
   \text{\rm None of $\pm k_2^{\pm 1}$ is among $1,q,q^2,\ldots, q^{(n-1)/2}$} \\
   \text{\rm None of $k_1k_2k_0^{\pm 1}k_3^{\pm 1}$ is among $q^{-1},q^{-3},q^{-5},\ldots,q^{-n}$}
 \end{array}
\\ \hline
\textup{\sf SSb} &
 \begin{array}{l}
   \rule{0mm}{4.3mm}
   \text{\rm None of $\pm k_1^{\pm 1}$ is among $1,q,q^2,\ldots, q^{(n-1)/2}$} \\
   \text{\rm None of $k_1k_2k_0^{\pm 1}k_3^{\pm 1}$ is among $q^{-1},q^{-3},q^{-5},\ldots,q^{-n}$}
 \end{array}
\end{array}
\]
\end{lemma}

\begin{proof}
First assume that $\V$ has $X$-type \textup{\sf DS}.
By Lemma \ref{lem:typekrpre1} $k_0^2 k_3^2$ is not among 
$q^{-2}$, $q^{-4}$, \ldots, $q^{-2n}$.
So neither of $\pm k_0 k_3$ is among $q^{-1},q^{-2},\ldots,q^{-n}$.
By Lemma \ref{lem:typekrpre2} neither of $k_0^2$, $k_3^2$ is among
$q^{-2}$, $q^{-4}$, \ldots, $q^{-n}$.
So none of $\pm k_0$, $\pm k_3$ is among $q^{-1},q^{-2},\ldots,q^{-n/2}$.
By Lemma \ref{lem:typekrpre2} $k_0 k_3 k_1^{-1} k_2$ is not among 
$q^{-1}$, $q^{-3}$, \ldots, $q^{1-n}$.
By Lemma \ref{lem:typekr0} $k_0 k_1 k_2 k_3 = q^{-n-1}$.
By these comments $k_1^{-2} q^{-n-1}$ is not among 
$q^{-1}$, $q^{-3}$, \ldots, $q^{1-n}$.
So neither of $\pm k_1$ is among $q^{-1}$, $q^{-2}$, \ldots, $q^{-n/2}$.
Similarly, neither of $\pm k_2$ is  among $q^{-1}$, $q^{-2}$, \ldots, $q^{-n/2}$.
Next assume that $\V$ has $X$-type \textup{\sf DDa}.
By Lemma \ref{lem:typekr0} $k_0^2 = q^{-n-1}$.
By Lemma \ref{lem:typekrpre1} $k_0^2 k_3^2$ is not among 
$q^{-2}$, $q^{-4}$, \ldots, $q^{-2n}$.
By these comments $k_3^2 q^{-n-1}$ is not among $q^{-2}$, $q^{-4}$, \ldots, $q^{-2n}$.
So neither of $\pm k_3$ is among 
$q^{(n-1)/2}$, $q^{(n-3)/2}$, \ldots, $q$, $1$, $q^{-1}$, \ldots, $q^{(1-n)/2}$.
So none of $\pm k_3^{\pm 1}$ is among $1$, $q$, $q^2$, \ldots, $q^{(n-1)/2}$.
By Lemma \ref{lem:typekrpre2} none of $k_0k_3 k_1^{\pm 1} k_2^{\pm 1}$ is
among $q^{-1}$, $q^{-3}$, \ldots, $q^{-n}$.
For the remaining types, we can show the result in a similar way as the case \textup{\sf DDa}.
\end{proof}

\section{A basis $\{u_r\}_{r=0}^n$ for $\V$ on which $X$ is upper tridiagonal and
$Y$ is lower tridiagonal}
\label{sec:ur}

Throughout this section Notation \ref{notation} is in effect.
In this section we introduce a certain basis $\{u_r\}_{r=0}^n$ for $\V$.
A square matrix is said to be {\em upper tridiagonal} whenever each nonzero entry
lies on the diagonal, the superdiagonal, or immediately above the superdiagonal.
A square matrix is said to be {\em lower tridiagonal} whenever its transpose is upper tridiagonal.
In later sections we show that with respect to the basis $\{u_r\}_{r=0}^n$
the matrix representing $X$ is upper tridiagonal and the matrix
representing $Y$ is lower tridiagonal.
To define the basis $\{u_r\}_{r=0}^n$, we consider the following scalars.

\begin{definition}     \label{def:be}    \samepage
\ifDRAFT {\rm def:be}. \fi
For $r=0,1,2,\ldots$ we define $\beta_r \in \F$ as follows:
\[
\begin{array}{c|l}
\textup{\rm $X$-type of $\V$} & \qquad\qquad\qquad \text{Definition of $\beta_r$} 
\\ \hline
\textup{\sf DS},\; \textup{\sf DDa}  \rule{0mm}{9mm} & \quad
\beta_r =
  \begin{cases}
  k_0 k_1 q^r      & \text{ if $r$ is even}  \\
  k_0 k_1 q^{r+1} & \text{ if $r$ is odd}
 \end{cases}
\\
\textup{\sf DDb}   \rule{0mm}{9mm} & \quad
\beta_r =
 \begin{cases}
   (k_2 k_3 q^{r+1})^{-1}   & \text{ if $r$ is even}   \\
   (k_2 k_3 q^r)^{-1}  & \text{ if $r$ is odd}
 \end{cases}
\\
\textup{\sf SSa}   \rule{0mm}{9mm} & \quad
\beta_r =
 \begin{cases}
  (k_0 k_1 q^r)^{-1}     & \text{ if $r$ is even}   \\
  (k_0 k_1 q^{r+1})^{-1}     & \text{ if $r$ is odd}
 \end{cases}
\\
\textup{\sf SSb}   \rule{0mm}{9mm}  & \quad
\beta_r =
 \begin{cases}
  k_2 k_3 q^{r+1}           & \text{ if $r$ is even}   \\
  k_2 k_3 q^r              & \text{ if $r$ is odd}
 \end{cases}
\end{array}
\]
\end{definition}

\begin{definition}     \label{def:ur}    \samepage
\ifDRAFT {\rm def:ur}. \fi
Pick $0 \neq u_0 \in \V_X(\mu_0)$.
We define vectors $u_1,u_2,\ldots$ in $\V$ inductively as follows:
\[
u_r =
  \begin{cases}
  u_{r-1} - \beta_{r-1} Y u_{r-1}      & \text{ if $\mu_{r-1}$ and $\mu_r$ are $1$-adjacent}, \\
  u_{r-1} - \beta_{r-1} Y^{-1} u_{r-1} & \text{ if $\mu_{r-1}$ and $\mu_r$ are $q$-adjacent}.
 \end{cases}
\]
\end{definition}

\begin{lemma}    \label{lem:urbasis}   \samepage
\ifDRAFT {\rm lem:urbasis}. \fi
With reference to Definition \ref{def:ur}, the vectors $\{u_r\}_{r=0}^n$ form a basis for $\V$.
\end{lemma}

\begin{proof}
Let $\{v_r\}_{r=0}^n$ denote the $X$-standard basis for $\V$ corresponding to the
ordering $\{\mu_r\}_{r=0}^n$.
For notational convenience set $\V_r = \V_X(\mu_r)$ for $0 \leq r \leq n$ and $\V_{-1} = 0$.
First assume that $\V$ has $X$-type \textup{\sf DS}.
By Lemma \ref{lem:tiv0vn} $t_0 v_0 \in \V_0$ and $t_0^{-1} v_0 \in \V_0$.
By Lemma \ref{lem:R15} 
\begin{align*}    
 t_0 v_r &\in 
   \begin{cases}
     \V_{r-1} + \V_r  &  \text{ if $r$ is even},
  \\
    (\V_r + \V_{r+1}) \setminus \V_r   &  \text{ if $r$ is odd}
   \end{cases}
  &&  (1 \leq r \leq n).
\end{align*}
By this and \eqref{eq:ki}
\begin{align*}
 t_0^{-1} v_r &\in 
   \begin{cases}
     \V_{r-1} + \V_r  &  \text{ if $r$ is even},
  \\
    (\V_r + \V_{r+1}) \setminus \V_r   &  \text{ if $r$ is odd}
   \end{cases}
  &&  (1 \leq r \leq n).
\end{align*}
Similarly, $t_1 v_n \in \V_n$, $t_1^{-1} v_n \in \V_n$, and
\begin{align*}
 t_1 v_r &\in
   \begin{cases}
     (\V_r + \V_{r+1}) \setminus \V_r  &  \text{ if $r$ is even},
  \\
    \V_{r-1} + \V_r   &  \text{ if $r$ is odd}
   \end{cases}
  &&  (0 \leq r \leq n-1),
\\
 t_1^{-1} v_r &\in
   \begin{cases}
     (\V_r + \V_{r+1}) \setminus \V_r  &  \text{ if $r$ is even},
  \\
    \V_{r-1} + \V_r    &  \text{ if $r$ is odd}
   \end{cases}
  &&  (0 \leq r \leq n-1).
\end{align*}
Using these comments one routinely finds that for $0 \leq r \leq n-2$
\begin{align}
 Y v_r &\in (\V_{r-1} + \V_r + \V_{r+1} + \V_{r+2}) \setminus (\V_{r-1} + \V_r + \V_{r+1})
   && \text{ if $r$ is even},                          \label{eq:Yvr}
\\
 Y^{-1} v_r &\in (\V_{r-1} + \V_r + \V_{r+1} + \V_{r+2}) \setminus (\V_{r-1} + \V_r + \V_{r+1})
   && \text{ if $r$ is odd}.                 \notag  
\end{align}
We now show that $\{u_r\}_{r=0}^n$ are linearly independent.
To this end we show
\begin{equation}
 u_r \in \big( \V_0 + \V_1 + \cdots + \V_r \big)
       \setminus   \big( \V_0 + \V_1 + \cdots + \V_{r-1} \big)   \label{eq:uraux1}
\end{equation}
using induction on $r=0,1,\ldots,n$.
For $r=0$ the line \eqref{eq:uraux1} obviously holds.
For $r=1$ we argue as follows. 
By Definition \ref{def:ur} and the shape of the $X$-diagram, 
 $u_1 = u_0 - \beta_0 Y^{-1} u_0$.
We have $u_0 \in \V_0$ and $t_0^{-1} u_0 \in \V_0$. 
We have $t_1^{-1} v_0 \in (\V_0 + \V_1) \setminus \V_0$.
So $Y^{-1} v_0 \in  (\V_0 + \V_1) \setminus \V_0$.
By these comments $u_1 \in  (\V_0 + \V_1) \setminus \V_0$,
and so \eqref{eq:uraux1} holds for $r=1$.
Now assume $2 \leq r \leq n$.
First assume that $r$ is even.
By  Definition \ref{def:ur}  and the shape of the $X$-diagram, 
\begin{align*}
  u_r &= u_{r-1} - \beta_{r-1} Y u_{r-1},
&
 u_{r-1} &= u_{r-2} - \beta_{r-2}  Y^{-1} u_{r-2}.
\end{align*}
Combining these two equations,
\begin{equation*}
 u_r = u_{r-1} + \beta_{r-1} \beta_{r-2} u_{r-2} - \beta_{r-1} Y u_{r-2}.  
\end{equation*}
By induction
\begin{align*}
  u_{r-1} &\in \V_0 + \cdots + \V_{r-1},
\\
  u_{r-2} &\in (\V_0 +  \cdots + \V_{r-2})
               \setminus            (\V_0 + \cdots + \V_{r-3}). 
\end{align*}
By \eqref{eq:Yvr}
\[
  Y v_{r-2} \in (\V_{r-3} + \V_{r-2} + \V_{r-1} + \V_r) 
                       \setminus   (\V_{r-3} + \V_{r-2} + \V_{r-1} ).
\]
By these comments \eqref{eq:uraux1} holds for even $r$.
Next assume that $r$ is odd.
We can show \eqref{eq:uraux1} in a similar way.
Thus the vectors $\{u_r\}_{r=0}^n$ are linearly independent, and so form a basis for $\V$.
We have shown the result when $\V$ has $X$-type \textup{\sf DS}.
The proof is similar for the other types.
\end{proof}

\begin{lemma}     \label{lem:zero}    \samepage
\ifDRAFT {\rm lem:zero}. \fi
With reference to Definition \ref{def:ur},
$u_r=0$ for $r > n$.
\end{lemma}

\begin{proof}
Let $\{v_r\}_{r=0}^n$ denote the $X$-standard basis for $\V$ 
such that $v_0 = u_0$.
By Lemmas \ref{lem:tiv0vn}, \ref{lem:R15}, \ref{lem:R20} we have the actions 
of $\{t_i\}_{i \in \I}$ on $\{v_r\}_{r=0}^n$.
Using these actions and Definition \ref{def:ur},
we represent $u_r$ inductively for $r=0,1,\ldots,n+1$
as a linear combination of $\{v_r\}_{r=0}^n$.
Here we omit the precise computations.
Now using Lemma \ref{lem:typekr0}
we find that $u_{n+1}=0$.
The result follows from this and Definition \ref{def:ur}.
\end{proof}

\section{The action of $Y^{\pm 1}$ on the basis $\{u_r\}_{r=0}^n$}
\label{sec:eigenY}

Throughout this section Notation \ref{notation} is in effect.
Let the scalars $\beta_0, \beta_1, \ldots$ be from Definition \ref{def:be} and
the vectors $u_0,u_1,\ldots$ be from Definition \ref{def:ur}.
By Lemma \ref{lem:urbasis} $\{u_r\}_{r=0}^n$ form a basis for $\V$.
In this section we obtain the action of $Y^{\pm 1}$ on the basis $\{u_r\}_{r=0}^n$,
and show that with respect to the basis $\{u_r\}_{r=0}^n$ the matrix representing
$Y^{\pm 1}$ is lower tridiagonal.
As a corollary, we obtain the condition for that $Y$ is diagonalizable on $\V$.
Note that $u_r = 0$ for $r > n$ by Lemma \ref{lem:zero}.

\begin{lemma}   \label{lem:DSYaction}   \samepage
\ifDRAFT {\rm lem:DSYaction}. \fi
Assume that $\V$ has $X$-type among \textup{\sf DS}, \textup{\sf DDa}.
Then for $0 \leq r \leq n$
\begin{align}
Y u_r &=
 \begin{cases}
  \beta_r u_r +  \beta_{r+2}^{-1} (u_{r+1} - u_{r+2})   & \text{if $r$ is even},
  \\
  \beta_r^{-1} (u_r - u_{r+1})    & \text{if $r$ is odd},
 \end{cases}                                                              \label{eq:DSYur}
\\
Y^{-1} u_r &=
 \begin{cases}
  \beta_r^{-1} (u_{r} - u_{r+1})   & \text{if $r$ is even},
  \\
  \beta_r u_r + \beta_r^{-1} (u_{r+1} - u_{r+2})    & \text{if $r$ is odd}.
 \end{cases}                                                              \label{eq:DSYinvur}
\end{align}
\end{lemma}

\begin{proof}
Note by Definition \ref{def:be} that  
\begin{align*}
\beta_r &=
  \begin{cases}
  k_0 k_1 q^r      & \text{ if $r$ is even},  \\
  k_0 k_1 q^{r+1} & \text{ if $r$ is odd}
 \end{cases}
&& (r=0,1,2,\ldots).
\end{align*}
Pick any integer $r$ such that $0 \leq r \leq n$.
By Definition \ref{def:ur}
\begin{align}
 u_{r+1} &= u_r - k_0 k_1 q^{r+1} Y u_r   && \text{if $r$ is odd}.    \label{eq:Yur1}
\end{align}
So \eqref{eq:DSYur} holds for odd $r$.
By Definition \ref{def:ur}
\begin{align}
 u_{r+1} &= u_r - k_0 k_1 q^r Y^{-1} u_r   && \text{if $r$ is even}.   \label{eq:Yur2}
\end{align}
So \eqref{eq:DSYinvur} holds for even $r$.
Applying $Y^{-1}$ to \eqref{eq:Yur1}
\begin{align*}
  Y^{-1} u_{r+1} &= Y^{-1} u_r - k_0 k_1 q^{r+1} u_r      && \text{if $r$ is odd}.
\end{align*}
By \eqref{eq:DSYinvur}
\begin{align*}
 Y^{-1} u_{r+1} &= (k_0 k_1 q^{r+1})^{-1} (u_{r+1} - u_{r+2})   && \text{if $r$ is odd}.
\end{align*}
Combining the above two equations, we obtain \eqref{eq:DSYinvur} for odd $r$.
Applying $Y$ to \eqref{eq:Yur2}
\begin{align*}
 Y u_{r+1} &= Y u_r - k_0 k_1 q^r u_r  &&  \text{if $r$ is even}.
\end{align*}
By \eqref{eq:DSYur}
\begin{align*}
 Y u_{r+1} &= (k_0 k_1 q^{r+2})^{-1}  (u_{r+1} - u_{r+2})  && \text{if $r$ is even}.
\end{align*}
Combining the above two equations we obtain \eqref{eq:DSYur} for even $r$.
\end{proof}

The following three lemmas can be shown in a similar way.

\begin{lemma}   \label{lem:DDbYaction}   \samepage
\ifDRAFT {\rm lem:DDbYaction}. \fi
Assume that $\V$ has $X$-type \textup{\sf DDb}.
Then for $0 \leq r \leq n$
\begin{align*}
 Y u_r &= 
 \begin{cases}
 \beta_r u_r + \beta_r^{-1} (u_{r+1} - u_{r+2})     & \text{ if $r$ is even},
 \\
 \beta_r^{-1}  (u_r - u_{r+1})        & \text{ if $r$ is odd},
 \end{cases}
\\
Y^{-1} u_r &=
 \begin{cases}
  \beta_r^{-1} (u_r - u_{r+1})      & \text{if $r$ is even},
  \\
 \beta_r u_r + \beta_{r+2}^{-1} (u_{r+1} - u_{r+2})     & \text{if $r$ is odd}.
 \end{cases}
\end{align*}
\end{lemma}

\begin{lemma}   \label{lem:SSaYaction}   \samepage
\ifDRAFT {\rm lem:SSaYaction}. \fi
Assume that $\V$ has $X$-type \textup{\sf SSa}.
Then for $0 \leq r \leq n$
\begin{align*}
 Y u_r &=
 \begin{cases}
  \beta_r^{-1} (u_r - u_{r+1})     & \text{ if $r$ is even},
 \\
  \beta_r u_r + \beta_r^{-1} (u_{r+1} - u_{r+2})    &  \text{ if $r$ is odd},
 \end{cases}
\\
Y^{-1} u_r &=
 \begin{cases}
 \beta_r u_r + \beta_{r+2}^{-1} (u_{r+1} - u_{r+2})      & \text{if $r$ is even},
  \\
  \beta_r^{-1} (u_r - u_{r+1})         & \text{if $r$ is odd}.
 \end{cases}
\end{align*}
\end{lemma}

\begin{lemma}   \label{lem:SSbYaction}   \samepage
\ifDRAFT {\rm lem:SSbYaction}. \fi
Assume that $\V$ has $X$-type \textup{\sf SSb}.
Then for $0 \leq r \leq n$
\begin{align*}
 Y u_r &=
 \begin{cases}
    \beta_r^{-1}  (u_r - u_{r+1})     & \text{ if $r$ is even},
 \\
   \beta_r u_r + \beta_{r+2}^{-1} (u_{r+1} - u_{r+2})    &  \text{ if $r$ is odd},
 \end{cases}
\\
Y^{-1} u_r &=
 \begin{cases}
 \beta_r u_r + \beta_r^{-1} (u_{r+1} - u_{r+2})    & \text{if $r$ is even},
  \\
 \beta_r^{-1} (u_r - u_{r+1})        & \text{if $r$ is odd}.
 \end{cases}
\end{align*}
\end{lemma}

By Lemmas \ref{lem:DSYaction}--\ref{lem:SSbYaction} we obtain the following result.

\begin{lemma}    \label{lem:Yeigen}   \samepage
\ifDRAFT {\rm lem:Yeigen}. \fi
With respect to the basis $\{u_r\}_{r=0}^n$ the matrix representing $Y$
is lower tridiagonal with the following diagonal entries:
\begin{equation}                                        \label{eq:Yeigen}
\begin{array}{c|c}
\textup{\rm $X$-type of $\V$} & \text{\rm Diagonal entries}
\\ \hline
\textup{\sf DS} &                     \rule{0mm}{3ex}
\beta_0, \beta_1^{-1}, \beta_2, \beta_3^{-1}, \ldots, \beta_{n-1}^{-1}, \beta_n
\\
\textup{\sf DDa}, \textup{\sf DDb} &                   \rule{0mm}{3ex}
\beta_0, \beta_1^{-1}, \beta_2, \beta_3^{-1}, \ldots, \beta_{n-1}, \beta_n^{-1}
\\
\textup{\sf SSa}, \textup{\sf SSb} &                   \rule{0mm}{3ex}
\beta_0^{-1}, \beta_1, \beta_2^{-1}, \beta_3, \ldots, \beta_{n-1}^{-1}, \beta_n
\end{array}
\end{equation}
\end{lemma}

\begin{proposition}    \label{prop:Yeigen}   \samepage
\ifDRAFT {\rm prop:Yeigen}. \fi
There exists a basis for $\V$ with respect to which the matrix
representing $Y$ is lower tridiagonal with the following diagonal entries:
\[
\begin{array}{c|c}
\textup{\rm $X$-type of $\V$} & \text{\rm Diagonal entries}
\\ \hline
\textup{\sf DS} &                     \rule{0mm}{3ex}
k_0 k_1, \;
\frac{1}{k_0 k_1 q^2}, \;
k_0 k_1 q^2, \;
\frac{1}{k_0 k_1 q^4}, \;
\ldots, \;
\frac{1}{k_0 k_1 q^n}, \;
k_0 k_1 q^n
\\
\textup{\sf DDa}, \textup{\sf SSa} &                   \rule{0mm}{3ex}
k_0 k_1, \;
\frac{1}{k_0 k_1 q^2}, \;
k_0 k_1 q^2, \;
\frac{1}{k_0 k_1 q^4}, \;
\ldots, \;
k_0 k_1 q^{n-1}, \;
\frac{1}{k_0 k_1 q^{n+1}}
\\
\textup{\sf DDb}, \textup{\sf SSb} &                   \rule{0mm}{3ex}
\frac{1}{k_2 k_3 q}, \;
k_2 k_3 q, \;
\frac{1}{k_2 k_3 q^3}, \;
k_2 k_3 q^3, \;
\ldots, \;
\frac{1}{k_2 k_3 q^n}, \;
k_2 k_3 q^n
\end{array}
\]
\end{proposition}

\begin{proof}
Follows from Definition \ref{def:be} and Lemma \ref{lem:Yeigen}.
\end{proof}

\begin{lemma}   \label{lem:Ymfree}  \samepage
\ifDRAFT {\rm prop:Ymfree}. \fi
$Y$ is multiplicity-free on any YD $\Hq$-module.
\end{lemma}

\begin{proof}
By Proposition \ref{prop:mfree} $X$ is multiplicity-free on any XD $\Hq$-module.
Consider the automorphism $\sigma$ from Lemma \ref{lem:sigma}.
By Lemma \ref{lem:sigmaXY} $\sigma$ sends $Y \mapsto X$.
The result follows.
\end{proof}

\begin{corollary} \label{cor:Ydiagonalizable}  \samepage
\ifDRAFT {\rm cor:Ydiagonalizable}. \fi
The following {\rm (i)--(iii)} are equivalent:
\begin{itemize}
\item[\rm (i)]
$Y$ is diagonalizable on $\V$;
\item[\rm (ii)]
$Y$ is multiplicity-free on $\V$;
\item[\rm (iii)]
the following inequalities hold:
\begin{equation}          \label{eq:condYmfree}
\begin{array}{c|c}
\textup{\rm $X$-type of $\V$}  &  \text{\rm Inequalities}  
\\ \hline
\textup{\sf DS}, \textup{\sf DDa}, \textup{\sf SSa}  \rule{0mm}{5mm}
 & \text{\rm Neither of $\pm k_0k_1$ is among $q^{-1},q^{-2},q^{-3},\ldots,q^{-n}$}
\\
\textup{\sf DDb}, \textup{\sf SSb}  \rule{0mm}{5mm}  
& \text{\rm Neither of $\pm k_2k_3$ is among $q^{-1},q^{-2},q^{-3},\ldots,q^{-n}$}
\end{array}
\end{equation}
\end{itemize}
\end{corollary}

\begin{proof}
(i)$\Rightarrow$(ii):
Follows from Lemma \ref{lem:Ymfree}.

(ii)$\Rightarrow$(i): By the definition.

(ii)$\Leftrightarrow$(iii):
One verifies that the eigenvalues of $Y$ given in Proposition \ref{prop:Yeigen}
are mutually distinct if and only if \eqref{eq:condYmfree} holds.
\end{proof}

\begin{note}   \label{note:Aactonur}   \samepage
\ifDRAFT {\rm note:Aactonur}. \fi
The action of $\A$ on $\{u_r\}_{r=0}^n$ is immediately obtained from
the action of $Y^{\pm 1}$ given in Lemmas \ref{lem:DSYaction}--\ref{lem:SSbYaction}.
Observe that with respect to the basis $\{u_r\}_{r=0}^n$
the matrix representing $\A$ is lower tridiagonal.
Moreover it is not lower bidiagonal when $n \geq 2$.
\end{note}

\section{The action of $\{t_i\}_{i \in \I}$ on the basis $\{u_r\}_{r=0}^n$}
\label{sec:tiaction}

Throughout this section Notation \ref{notation} is in effect.
Let the vectors $u_0,u_1,\ldots$ be from Definition \ref{def:ur}.
By Lemma \ref{lem:urbasis} $\{u_r\}_{r=0}^n$ form a basis for $\V$.
Note that $u_r = 0$ for $r>n$ by Lemma \ref{lem:zero}.
For notational convenience, define $u_r=0$ for $r<0$.
In this section we obtain the action of $\{t_i\}_{i \in \I}$ on $\{u_r\}_{r=0}^n$.

\begin{lemma}   \label{lem:DSt0t1action}    \samepage
\ifDRAFT {\rm lem:DSt0t1action}. \fi
Assume that $\V$ has $X$-type \textup{\sf DS}.
Then for $0 \leq r \leq n$
\begin{align}
t_0 u_{r} &=
 \begin{cases}
  k_0 q^r u_{r-1}
  + \big(\frac{1}{k_0q^r} - k_0 - k_0^{-1} \big) \, (u_{r-1} - u_r)
    & \text{ if $r$ is even},
  \\
  \frac{1}{k_0 q^{r+1}} \, ( u_r - u_{r+1} )
    & \text{ if $r$ is odd},
 \end{cases}                                 \label{eq:DSt0ur}
\\
t_1 u_{r} &=
 \begin{cases}
  - k_1 (1-q^r)(1-k_0^2 q^r) \, u_{r-1}
  + k_1 u_{r} + k_1^{-1} u_{r+1}
   &  \text{ if $r$ is even},
 \\
  k_1^{-1} u_r
   & \text{ if $r$ is odd}.
 \end{cases}                                 \label{eq:DSt1ur}
\end{align}
\end{lemma}

\begin{proof}
We show \eqref{eq:DSt0ur} and \eqref{eq:DSt1ur} using induction on $r=0,1,\ldots,n$.
By Definition \ref{def:consistent} $t_0 u_0 = k_0 u_0$. So \eqref{eq:DSt0ur} holds for $r=0$.
By Definition \ref{def:ur} 
\[
  u_1 = u_0 - k_0 k_1 t_1^{-1} t_0^{-1} u_0.
\]
In this line, use $t_0^{-1} u_0 = k_0^{-1} u_0$, $t_1^{-1} = T_1 - t_1$ and \eqref{eq:ki} to find
\[
  t_1 u_0 = k_1 u_0 + k_1^{-1} u_1.
\]
So \eqref{eq:DSt1ur} holds for $r=0$.
Now pick any integer $r$ such that $1 \leq r \leq n$.
First assume that $r$ is odd.
By Definition \ref{def:ur}
\[
  u_r = u_{r-1} - k_0 k_1 q^{r-1} t_1^{-1} t_0^{-1} u_{r-1}.
\]
In this line, apply $t_1$ to each side, use $t_0^{-1} = T_0 - t_0$ and \eqref{eq:ki},
and evaluate $t_1 u_{r-1}$ and $t_0 u_{r-1}$ by induction.
This yields \eqref{eq:DSt1ur} for odd $r$.
By $t_0 = Y t_1^{-1}$, $t_1^{-1} = T_1 - t_1$ and \eqref{eq:ki}
\[
  t_0 u_r = (k_1 + k_1^{-1}) Y u_r - Y t_1 u_r.
\]
In this line, use $t_1 u_r = k_1^{-1} u_r$, and evaluate $Y u_r$
by \eqref{eq:DSYur}.
This yields \eqref{eq:DSt0ur} for odd $r$.
Next assume that $r$ is even.
By Definition \ref{def:ur}
\[
  u_r = u_{r-1} - k_0 k_1 q^r t_0 t_1 u_{r-1}.
\]
In this line, evaluate $t_1 u_{r-1}$ by induction to find
\[
 u_r = u_{r-1} - k_0 q^r t_0 u_{r-1}.
\]
In this line, apply $t_0$ to each side, and use $t_0^2 = t_0 T_0 -1$ and \eqref{eq:ki}
to find
\[
 t_0 u_r = t_0 u_{r-1} - k_0 q^r (k_0 + k_0^{-1}) t_0 u_{r-1} + k_0 q^r u_{r-1}. 
\]
In this line, evaluate $t_0 u_{r-1}$ by induction.
This yields \eqref{eq:DSt0ur} for even  $r$.
By $t_1 = T_1 - t_1^{-1}$, $t_1^{-1} = Y^{-1} t_0$ and \eqref{eq:ki}
\[
  t_1 u_r = (k_1 + k_1^{-1}) u_r - Y^{-1} t_0 u_r.
\]
In this line, use \eqref{eq:DSt0ur} for even $r$ and \eqref{eq:DSYinvur}.
This yields \eqref{eq:DSt1ur} for even $r$.
\end{proof}

\begin{lemma}   \label{lem:DSt2t3action}    \samepage
\ifDRAFT {\rm lem:DSt2t3action}. \fi
Assume that $\V$ has $X$-type \textup{\sf DS}.
Then for $0 \leq r \leq n$
\begin{align}
t_2 u_{r} &=
 \begin{cases}
  k_2q^{n-r}  (u_r - u_{r+1})
  & \text{ if $r$ is even},
 \\
  \frac{(1-q^{n-r+1})(1-k_2^2 q^{n-r+1}) }
       {k_2 q^{n-r+1} }  \, u_{r-1}       
  + ( k_2+k_2^{-1} -  k_2 q^{n-r+1} ) \, u_r
  & \text{ if $r$ is odd},
 \end{cases}                               \label{eq:DSt2ur}
\\
t_3 u_{r} &= 
 \begin{cases}
  k_3 u_r   & \text{ if $r$ is even},
 \\
  - \frac{(1-q^{n-r+1})(1-k_2^2 q^{n-r+1}) }
            {k_2^2 k_3 q^{2(n-r+1)} } \, u_{r-1}
  + k_3^{-1} u_r + k_3 u_{r+1}
  & \text{ if $r$ is odd}.
 \end{cases}                                \label{eq:DSt3ur}
\end{align}
\end{lemma}

\begin{proof}
Note that $k_0 k_1 k_2 k_3 = q^{-n-1}$ by Lemma \ref{lem:typekr0}.
Also note by \eqref{eq:defHq3} that $Y=q^{-1} t_3^{-1} t_2^{-1}$ and $Y^{-1} = q t_2 t_3$.
We show \eqref{eq:DSt2ur} and \eqref{eq:DSt3ur} using induction on $r=0,1,\ldots,n$.
By Definition \ref{def:consistent} $t_3 u_0 = k_3 u_0$.
So \eqref{eq:DSt3ur} holds for $r=0$.
By Definition \ref{def:ur}
\[
  u_1 = u_0 - k_0 k_1 Y^{-1} u_0.
\]
In this line, use $Y^{-1} = q t_2 t_3$, $t_3 u_0 = k_3 u_0$ and 
$k_0 k_1 k_3 q = k_2^{-1} q^{-n}$ to find
\[
  t_2 u_0 = k_2 q^n (u_0 - u_1).
\]
So \eqref{eq:DSt2ur} holds for $r=0$.
Now pick any integer $r$  such that $1 \leq r \leq n$.
First assume that $r$ is odd.
By Definition \ref{def:ur}
\[
  u_r = u_{r-1} - k_0 k_1 q^{r-1} Y^{-1} u_{r-1}.
\]
In this line, use $Y^{-1} = q t_2 t_3$, and evaluate $t_3 u_{r-1}$ by induction to find
\[
 u_r = u_{r-1} - k_0 k_1 k_3 q^r t_2 u_{r-1}.
\]
In this line, apply $t_2^{-1}$ to each side to find
\[
 t_2^{-1} u_r = t_2^{-1} u_{r-1} - k_0 k_1 k_3 q^r u_{r-1}.
\]
In this line, use $t_2^{-1} = T_2 - t_2$ and \eqref{eq:ki} to find
\[
 t_2 u_r = t_2 u_{r-1} + (k_2 + k_2^{-1})(u_r - u_{r+1}) + k_0 k_1 k_3 q^r u_{r-1}.
\]
In this line, evaluate $t_2 u_{r-1}$ by induction, and use $k_0 k_1 k_3 = q^{-n-1} k_2^{-1}$.
This yields \eqref{eq:DSt2ur} for odd $r$.
By $t_3 = T_3 - t_3^{-1}$, $t_3^{-1} = q Y t_2$ and \eqref{eq:ki}
\[
  t_3 u_r = (k_3 + k_3^{-1}) u_r - q Y t_2 u_r.
\]
In this line, use \eqref{eq:DSt2ur} for odd $r$, \eqref{eq:DSYur}, and $k_3 = q^{-n-1} (k_0 k_1 k_2)^{-1}$.
This yields \eqref{eq:DSt3ur} for odd $r$.
Next assume that $r$ is even.
By Definition \ref{def:ur}
\[
  u_r = u_{r-1} - k_0 k_1 q^r Y u_{r-1}.
\]
In this line, apply $t_3$ to each side, and use $Y=q^{-1} t_3^{-1} t_2^{-1}$,
$t_2^{-1} = T_2 - t_2$ and \eqref{eq:ki} to find
\[
 t_3 u_r = t_3 u_{r-1} - k_0 k_1 q^{r-1} (k_2 + k_2^{-1}) u_{r-1} + k_0 k_1 q^{r-1} t_2 u_{r-1}.
\]
In this line, evaluate $t_3 u_{r-1}$ and $t_2 u_{r-1}$ by induction.
This yields \eqref{eq:DSt3ur} for even $r$.
By $t_2 = q^{-1} Y^{-1} t_3^{-1}$, $t_3^{-1} = T_3 - t_3$ and \eqref{eq:ki}
\[
  t_2 u_r = q^{-1} Y^{-1} (k_3 + k_3^{-1}) u_r - q^{-1} Y^{-1} t_3 u_r.
\]
In this line, use \eqref{eq:DSt3ur} for even $r$, \eqref{eq:DSYinvur}, 
and $k_0 k_1 k_3 = q^{-n-1} k_2^{-1}$.
This yields \eqref{eq:DSt2ur} holds for even $r$.
\end{proof}

We have obtained the action of $\{t_i\}_{i \in \I}$ when $\V$ has $X$-type \textup{\sf DS}.
In a similar way, we obtain the following four lemmas.

\begin{lemma}   \label{lem:DDat0action}    \samepage
\ifDRAFT {\rm lem:DDat0action}. \fi
Assume that $\V$ has $X$-type \textup{\sf DDa}.
Then for $0 \leq r \leq n$
\begin{align*}
t_0 u_{r} &=
 \begin{cases}
 - \frac{(1-q^r)(1-q^{n-r+1}) }
           {k_0q^{n+1} } \, u_{r-1}
  + \big( k_0+k_0^{-1}- \frac{1}{k_0q^r} \big) \, u_r
    & \text{ if $r$ is even},
  \\
   \frac{1}{k_0 q^{r+1}} \, (u_r - u_{r+1})
    & \text{ if $r$ is odd},
 \end{cases} 
\\
t_1 u_{r} &=
 \begin{cases}
  \frac{(1-q^r)(1-q^{n-r+1}) }
         {k_1^{-1}q^{n-r+1} } \, u_{r-1}
  + k_1 u_{r} + k_1^{-1} u_{r+1}
   &  \text{ if $r$ is even},
 \\
  k_1^{-1} u_r
   & \text{ if $r$ is odd},
 \end{cases}
\\
t_2 u_{r} &=
 \begin{cases}
  \frac{1}{k_0k_1k_3 q^{r+1}} ( u_r - u_{r+1})
  & \text{ if $r$ is even},
 \\
  \frac{(1-k_0k_1k_2k_3 q^r)(1-k_0k_1k_2^{-1}k_3 q^r)}
       {k_0 k_1 k_3 q^r}  \, u_{r-1}       
  + \big( k_2+k_2^{-1} -  \frac{1}{k_0k_1k_3 q^r} \big) \, u_r
  & \text{ if $r$ is odd},
 \end{cases} 
\\
t_3 u_{r} &= 
 \begin{cases}
  k_3 u_r   & \text{ if $r$ is even},
 \\
  - k_3^{-1} (1-k_0k_1k_2k_3q^r)(1-k_0k_1k_2^{-1}k_3 q^r) u_{r-1}
  + k_3^{-1} u_r + k_3 u_{r+1}
  & \text{ if $r$ is odd}.
 \end{cases}
\end{align*}
\end{lemma}

\begin{lemma} \label{lem:DDbt0action} \samepage
\ifDRAFT {\rm lem:DDbt0action}. \fi
Assume that $\V$ has $X$-type \textup{\sf DDb}. 
Then for $0 \leq r \leq n$
\begin{align*}
t_0 u_{r} &=
 \begin{cases}
  k_0 u_r
   &  \text{ if $r$ is even},
  \\
  \frac{(1-k_0k_1k_2k_3q^r)(1-k_0k_1^{-1}k_2k_3q^r)}
       {k_0k_2^2 q^{2r-n-1}}    \, u_{r-1}
    + k_0^{-1} (u_r - u_{r+1})
    & \text{ if $r$ is odd},
  \end{cases}
\\
t_1 u_{r} &=
 \begin{cases}
   (k_1+k_1^{-1}-k_0k_2k_3q^{r+1}) u_r 
    + k_0k_2k_3 q^{r+1} u_{r+1}
  & \text{ if $r$ is even},
 \\
  \big( k_1+k_1^{-1}- k_0k_2k_3 q^r -\frac{1}{k_0k_2k_3q^r} \big) \, u_{r-1}
   + k_0k_2k_3 q^r u_r
  & \text{ if $r$ is odd},
 \end{cases} 
\\
t_2 u_{r} &=
 \begin{cases}
   - \frac{(1-q^r)(1-q^{n-r+1})
           }{k_2 q^{r} } u_{r-1}
    + k_2 (u_r - u_{r+1} )
   & \text{ if $r$ is even},
  \\
  k_2^{-1} u_r
   & \text{ if $r$ is odd},
 \end{cases} 
\\
t_3 u_{r} &= 
 \begin{cases}
  \big( k_3+k_3^{-1} - k_3 q^r - \frac{1}{k_3 q^r} \big) \, u_{r-1}
   + k_3 q^r u_r
  & \text{ if $r$ is even},
 \\
  (k_3+k_3^{-1}-k_3 q^{r+1}) u_r + k_3 q^{r+1} u_{r+1}
  & \text{ if $r$ is odd}.
 \end{cases}
\end{align*}
\end{lemma}

\begin{lemma} \label{lem:SSat0action} \samepage
\ifDRAFT {\rm lem:SSat0action}. \fi
Assume that $\V$ has $X$-type  \textup{\sf SSa}. 
Then for $0 \leq r \leq n$
\begin{align*}
t_0 u_{r} &=
 \begin{cases}
  - \frac{(1-q^r)(1-q^{n-r+1})}
           {k_0 q^{r} } \, u_{r-1}
   + k_0 (u_r - u_{r+1})
   &  \text{ if $r$ is even},
  \\
   k_0^{-1} u_r
    &  \text{ if $r$ is odd},
  \end{cases}
\\
t_1 u_{r} &=
 \begin{cases}
  \big( k_1+k_1^{-1}- k_1 q^r- \frac{1}{k_1 q^r} \big) \, u_{r-1}
  + k_1 q^r u_r
  & \text{ if $r$ is even},
 \\
  (k_1+k_1^{-1}-k_1 q^{r+1}) u_r + k_1 q^{r+1} u_{r+1}
  & \text{ if $r$ is odd},
 \end{cases} 
\\
t_2 u_{r} &=
 \begin{cases}
  k_2 u_r
  & \text{ if $r$ is even},
 \\
  \frac{(1-k_0k_1k_2k_3 q^r)(1-k_0k_1k_2k_3^{-1}q^r)}
         {k_0^2 k_2 q^{2r-n-1}}\, u_{r-1}
  + k_2^{-1} (u_r -  u_{r+1} )
  & \text{ if $r$ is odd},
 \end{cases} 
\\
t_3 u_{r} &= 
 \begin{cases}
  ( k_3+k_3^{-1} - k_0k_1k_2 q^{r+1}) \, u_{r}
  + k_0 k_1k_2 q^{r+1} u_{r+1}
  & \text{ if $r$ is even},
 \\
 \big( k_3+k_3^{-1}-k_0k_1k_2q^r - \frac{1}{k_0k_1k_2q^r} \big) \, u_{r-1}
  + k_0k_1k_2 q^r u_r
  & \text{ if $r$ is odd}.
 \end{cases}
\end{align*}
\end{lemma}

\begin{lemma} \label{lem:SSbtiaction} \samepage
\ifDRAFT {\rm lem:SSbtiaction}. \fi
Assume that $\V$ has $X$-type \textup{\sf SSb}. 
Then for $0 \leq r \leq n$
\begin{align*}
t_0 u_{r} &=
 \begin{cases}
    \frac{1}{k_1k_2k_3q^{r+1}} \, (u_r - u_{r+1}) 
   & \text{ if $r$ is even},
  \\
   \big( k_1k_2k_3q^r + \frac{1}{k_1k_2k_3q^r} - k_0-k_0^{-1} \big) \, u_{r-1}
   + \big( k_0 + k_0^{-1} - \frac{1}{k_1k_2k_3q^r} \big) \, u_r
    & \text{ if $r$ is odd},
  \end{cases}
\\
t_1 u_{r} &=
 \begin{cases}
  k_1 u_r
  & \text{ if $r$ is even},
 \\
  - k_1^{-1}(1-k_0k_1k_2k_3q^r)(1-k_0^{-1}k_1k_2k_3q^r) u_{r-1}
  + k_1^{-1} u_r + k_1 u_{r+1}
  & \text{ if $r$ is odd},
 \end{cases} 
\\
t_2 u_{r} &=
 \begin{cases}
 - \frac{(1-q^r)(1-q^{n-r+1})}
           {k_2 q^{n+1} } \, u_{r-1}
  + \big( k_2+ k_2^{-1} - \frac{1}{k_2 q^r} \big) \, u_r
  & \text{ if $r$ is even},
 \\
  \frac{1}{k_2 q^{r+1}} (u_r - u_{r+1})
  & \text{ if $r$ is odd},
 \end{cases} 
\\
t_3 u_{r} &= 
 \begin{cases}
    \frac{(1-q^r)(1-q^{n-r+1}) }
           {k_3^{-1}q^{n-r+1} } \, u_{r-1}
   + k_3 u_r + k_3^{-1} u_{r+1}
  & \text{ if $r$ is even},
 \\
  k_3^{-1} u_r
  & \text{ if $r$ is odd}.
 \end{cases}
\end{align*}
\end{lemma}

\section{The action of $X^{\pm 1}$ on the basis $\{u_r\}_{r=0}^n$}
\label{sec:Xaction}

Throughout this section Notation \ref{notation} is in effect.
Consider the basis $\{u_r\}_{r=0}^n$ from Definition \ref{def:ur}.
In this section we obtain the action of $X^{\pm 1}$ on $\{u_r\}_{r=0}^n$,
and show that with respect to the basis $\{u_r\}_{r=0}^n$
the matrix representing $X^{\pm 1}$ is upper tridiagonal.

To simplify the action of $X^{\pm 1}$, we normalize the vectors $\{u_r\}_{r=0}^n$
using the following scalars.

\begin{definition}   \label{def:er}    \samepage
\ifDRAFT {\rm def:er}. \fi
Set $e_0 = 1$. 
For $1 \leq r \leq n$ define $e_r \in \F$ as follows:
\[
\begin{array}{c|l}
 \text{$X$-type of $\V$}  &  \qquad\qquad\qquad e_r
\\ \hline
 \textup{\sf DS}, \textup{\sf DDa} &    \rule{0mm}{2.5em}
  \begin{cases}
     \frac{1}{(1- q^r)(1-k_0^2 q^r)}  &  \text{ if $r$ is even}
   \\
   \frac{1}{(1-k_0 k_1 k_2 k_3 q^r)(1-k_0 k_1 k_2^{-1} k_3 q^r)}  & \text{ if $r$ is odd}
 \end{cases}
\\
 \textup{\sf DDb}  &       \rule{0mm}{2.5em}
  \begin{cases}
    - \frac{1}{ k_0^2 (1-q^r)(1-k_3^2 q^r)}  & \text{ if $r$ is even}
 \\
  - \frac{k_2^2} {(1-k_0 k_1 k_2 k_3 q^r)(1-k_0 k_1^{-1} k_2 k_3 q^r)}
                                                        & \text{ if $r$ is odd}
 \end{cases}
\\
 \textup{\sf SSa}  &             \rule{0mm}{2.5em}
  \begin{cases}
    - \frac{1}{ k_2^2 (1-q^r)(1-k_1^2 q^r)}  & \text{ if $r$ is even}
 \\
  - \frac{k_0^2}{(1- k_0 k_1 k_2 k_3 q^r)(1-k_0 k_1 k_2 k_3^{-1}q^r)}  & \text{ if $r$ is odd}
 \end{cases}
\\
 \textup{\sf SSb}  &             \rule{0mm}{2.5em}
  \begin{cases}
     \frac{1}{ (1-q^r)(1-k_2^2 q^r)}  & \text{ if $r$ is even}
 \\
   \frac{1}{(1- k_0 k_1 k_2 k_3 q^r)(1- k_0^{-1} k_1 k_2 k_3 q^r)}  & \text{ if $r$ is odd}
 \end{cases}
\end{array}
\]
In the above, all the denominators are nonzero by Lemmas \ref{lem:typekr0} and \ref{lem:typekr}.
\end{definition}

\begin{definition}   \label{def:udr}    \samepage
\ifDRAFT {\rm def:udr}. \fi
For $0 \leq r \leq n$ define 
\[
  u'_r = e_0 e_1 \cdots e_r \, u_r.
\]
\end{definition}

Below we give formulas for the action of $X^{\pm 1}$ on the basis $\{u'_r\}_{r=0}^n$.
For notational convenience, set $u'_r = 0$ for $r < 0$.
These formulas can be routinely obtained using Lemmas 
\ref{lem:DSt0t1action}--\ref{lem:SSbtiaction}.

\begin{lemma}   \label{lem:DSDDa}    \samepage
\ifDRAFT {\rm lem:DSDDa}. \fi
Assume that $\V$ has $X$-type among \textup{\sf DS}, \textup{\sf DDa}.
Then for $0 \leq r \leq n$
\begin{align*}
 X u'_r &=
  \begin{cases}
    k_0 k_3 q^r u'_r + (k_0 k_3 q^r)^{-1} (u'_{r-1} - u'_{r-2}) & \text{ if $r$ is even},
  \\
   (k_0 k_3 q^{r+1})^{-1}  (u'_r - u'_{r-1})     & \text{ if $r$ is odd},
  \end{cases}
\\
 X^{-1} u'_r &=
 \begin{cases}
  (k_0 k_3 q^r)^{-1}  (u'_r - u'_{r-1})  & \text{ if $r$ is even},
 \\
  k_0 k_3 q^{r+1}u'_r  + (k_0 k_3 q^{r-1})^{-1} (u'_{r-1} - u'_{r-2})  & \text{ if $r$ is odd}.
 \end{cases}
\end{align*}
\end{lemma}

\begin{lemma}   \label{lem:DDb}    \samepage
\ifDRAFT {\rm lem:DDb}. \fi
Assume that $\V$ has $X$-type  \textup{\sf DDb}.
Then for $0 \leq r \leq n$
\begin{align*}
 X u'_r &=
  \begin{cases}
   k_0 k_3 q^r u'_r +(k_0 k_3 q^r)^{-1} u'_{r-1}  & \text{ if $r$ is even},
  \\
   (k_0 k_3 q^{r+1})^{-1}  (u'_r - u'_{r-1}) -(k_0^3 k_3^3 q^{3r-1})^{-1} u'_{r-2}    & \text{ if $r$ is odd},
  \end{cases}
\\
 X^{-1} u'_r &=
 \begin{cases}
 (k_0 k_3 q^r)^{-1} (u'_r - u'_{r-1}) - (k_0^3 k_3^3 q^{3r-2})^{-1}  u'_{r-2} & \text{ if $r$ is even},
 \\
 k_0 k_3 q^{r+1} u'_r + (k_0 k_3 q^{r-1})^{-1} u'_{r-1}  & \text{ if $r$ is odd}.
 \end{cases}
\end{align*}
\end{lemma}

\begin{lemma}   \label{lem:SSa}    \samepage
\ifDRAFT {\rm lem:SSa}. \fi
Assume that $\V$ has $X$-type  \textup{\sf SSa}.
Then for $0 \leq r \leq n$
\begin{align*}
 X u'_r &=
  \begin{cases}
    (k_1 k_2 q^{r+1})^{-1} (u'_r - u'_{r-1}) - (k_1^3 k_2^3 q^{3r-1})^{-1} u'_{r-2}  & \text{ if $r$ is even},
  \\
   k_1 k_2 q^r u'_r +  (k_1 k_2 q^r)^{-1} u'_{r-1}   & \text{ if $r$ is odd},
  \end{cases}
\\
 X^{-1} u'_r &=
 \begin{cases}
  k_1 k_2 q^{r+1}  u'_r + (k_1 k_2 q^{r-1})^{-1} u'_{r-1}  & \text{ if $r$ is even},
 \\
 (k_1 k_2 q^r)^{-1}  (u'_r - u'_{r-1}) - (k_1^3 k_2^3 q^{3r-2})^{-1}  u'_{r-2}  & \text{ if $r$ is odd}.
 \end{cases}
\end{align*}
\end{lemma}

\begin{lemma}   \label{lem:SSb}    \samepage
\ifDRAFT {\rm lem:SSb}. \fi
Assume that $\V$ has $X$-type  \textup{\sf SSb}.
Then for $0 \leq r \leq n$
\begin{align*}
 X u'_r &=
  \begin{cases}
    (k_1 k_2 q^{r+1})^{-1}  (u'_r - u'_{r-1})      & \text{ if $r$ is even},
  \\
   k_1 k_2 q^r u'_r + (k_1 k_2 q^r)^{-1}  (u'_{r-1} - u'_{r-2})   & \text{ if $r$ is odd},
  \end{cases}
\\
 X^{-1} u'_r &=
 \begin{cases}
  k_1 k_2 q^{r+1} u'_r + (k_1 k_2 q^{r-1})^{-1}  (u'_{r-1} - u'_{r-2})   & \text{ if $r$ is even},
 \\
 (k_1 k_2 q^r)^{-1} (u'_r - u'_{r-1})   & \text{ if $r$ is odd}.
 \end{cases}
\end{align*}
\end{lemma}

\begin{corollary}    \label{cor:Xact}   \samepage
\ifDRAFT {\rm cor:Xact}. \fi
With respect to the basis $\{u_r\}_{r=0}^n$ the matrix representing $X^{\pm 1}$
is upper tridiagonal.
\end{corollary}

\begin{proof}
By lemmas \ref{lem:DSDDa}--\ref{lem:SSb}
the matrix representing $X^{\pm 1}$ with respect to $\{u'_r\}_{r=0}^n$ is
upper tridiagonal.
The result follows since $\{u'_r\}_{r=0}^n$ is a normalization of $\{u_r\}_{r=0}^n$.
\end{proof}

\begin{note}   \label{note:Bactonur}   \samepage
\ifDRAFT {\rm note:Bactonur}. \fi
The action of $\B$ on $\{u'_r\}_{r=0}^n$ is immediately obtained from
the action of $X^{\pm 1}$ given in Lemmas \ref{lem:DSDDa}--\ref{lem:SSb}.
Observe that 
with respect to the basis $\{u_r\}_{r=0}^n$ the matrix representing $\B$
is upper tridiagonal.
Moreover it is not upper bidiagonal when $n \geq 2$.
\end{note}

\section{A basis for $\V(k_0^{\pm 1})$}
\label{sec:basisVpm}

Throughout this section Notation \ref{notation} is in effect.
Assume $k_0 \neq k_0^{-1}$.
Let the subspaces $\V(k_0^{\pm 1})$ be from \eqref{eq:defVk0},
and the elements $F^{\pm}$ be from \eqref{eq:defF+F-}.
In this section, we construct a basis for $\V(k_0^{\pm 1})$
with respect to which the matrix representing $\A$ is lower bidiagonal
and the matrix representing $\B$ is upper bidiagonal.
Let the basis $\{u_r\}_{r=0}^n$ be from Definition \ref{def:ur}.
The following five Lemmas can be routinely obtained
using the action of $t_0$ given in Lemmas \ref{lem:DSt0t1action}--\ref{lem:SSbtiaction}.

\begin{lemma}   \label{lem:DSFur}  \samepage
\ifDRAFT {\rm lem:DSFur}. \fi
Assume that $\V$ has $X$-type \textup{\sf DS}.
Then for $0 \leq r \leq n$
\begin{align}
F^+ u_r &=
 \begin{cases}
  - \frac{1-k_0^2 q^r}
            {q^r (1- k_0^2)} \big( (1-q^r)  u_{r-1} - u_r \big)
  & \text{ if $r$ is even},
 \\
  - \frac{1}
           {q^{r+1}(1-k_0^2)} \big( (1-q^{r+1}) u_r - u_{r+1} \big)
  & \text{ if $r$ is odd},
 \end{cases}                                                  \label{eq:DSF+ur}
\\
F^- u_r &=
 \begin{cases}
  \frac{(1-q^r)}
         {q^r(1-k_0^2)} \big( (1- k_0^2 q^r)  u_{r-1} - u_r \big)
  & \text{ if $r$ is even},
 \\
  \frac{1}
         {q^{r+1}(1-k_0^2)} \big( (1-k_0^2 q^{r+1})  u_r - u_{r+1} \big)
    & \text{ if $r$ is odd}.
 \end{cases}                                              \label{eq:DSF-ur}
\end{align}
\end{lemma}

\begin{lemma}   \label{lem:DDaFur}  \samepage
\ifDRAFT {\rm lem:DDaFur}. \fi
Assume that $\V$ has $X$-type \textup{\sf DDa}.
Then for $0 \leq r \leq n$
\begin{align}
F^+ u_r &=
 \begin{cases}
  - \frac{1-q^{n-r+1}}
            { 1-q^{n+1} } 
    \big(  (1-q^r) u_{r-1} - u_r \big)
  & \text{ if $r$ is even},
 \\
  \frac{q^{n-r}}
         {1-q^{n+1}} 
  \big( (1- q^{r+1}) u_r - u_{r+1} \big)
  & \text{ if $r$ is odd},
 \end{cases}                                              \label{eq:DDaF+ur}
\\
F^- u_r &=
 \begin{cases}
 \frac{1-q^r}
        {1-q^{n+1}} 
 \big( (1-q^{n-r+1}) u_{r-1} + q^{n-r+1} u_r \big)
  & \text{ if $r$ is even},
 \\
  \frac{1 }
         {1-q^{n+1} } 
 \big( (1-q^{n-r}) u_r + q^{n-r} u_{r+1} \big)
    & \text{ if $r$ is odd}.
 \end{cases}                                              \label{eq:DDaF-ur}
\end{align}
\end{lemma}

\begin{lemma}   \label{lem:DDbFur}  \samepage
\ifDRAFT {\rm lem:DDbFur}. \fi
Assume that $\V$ has $X$-type \textup{\sf DDb}.
 Then for $0 \leq r \leq n$
\begin{align*}
F^+ u_r &=
 \begin{cases}
  u_r
  & \text{ if $r$ is even},
 \\
  - \frac{(1-k_0k_1k_2k_3q^r)(1-k_0k_1^{-1}k_2k_3q^r)}
           {(1-k_0^2)k_2^2 q^{2r-n-1}} \, u_{r-1}
  + \frac{1}
            {1-k_0^2} \, u_{r+1}
  & \text{ if $r$ is odd},
 \end{cases}
\\
F^- u_r &=
 \begin{cases}
   0
  & \text{ if $r$ is even},
 \\
   \frac{(1-k_0k_1k_2k_3q^r)(1-k_0k_1^{-1}k_2k_3q^r)}
         {(1-k_0^2)k_2^2 q^{2r-n-1}} \, u_{r-1}
  + u_r 
  - \frac{1}
           {1-k_0^2} \, u_{r+1}
    & \text{ if $r$ is odd}.
 \end{cases}
\end{align*}
\end{lemma}

\begin{lemma}   \label{lem:SSaFur}  \samepage
\ifDRAFT {\rm lem:SSaFur}. \fi
Assume that $\V$ has $X$-type \textup{\sf SSa}.
 Then for $0 \leq r \leq n$
\begin{align*}
F^+ u_r &=
 \begin{cases}
    \frac{(1-q^r)(1-q^{n-r+1})}
         {(1-k_0^2) q^{r}} \, u_{r-1}
  + u_r 
  + \frac{k_0^2}
           {1-k_0^2} \, u_{r+1}    
  & \text{ if $r$ is even},
 \\
   0
  & \text{ if $r$ is odd},
 \end{cases}
\\
F^- u_r &=
 \begin{cases}
   - \frac{(1-q^r)(1-q^{n-r+1})}
          {(1-k_0^2) q^{r}} \, u_{r-1}
   - \frac{k_0^2}
            {1-k_0^2} \, u_{r+1}
  & \text{ if $r$ is even},
 \\
    u_r
    & \text{ if $r$ is odd}.
 \end{cases}
\end{align*}
\end{lemma}

\begin{lemma}   \label{lem:SSbFur}  \samepage
\ifDRAFT {\rm lem:SSbFur}. \fi
Assume that $\V$ has $X$-type \textup{\sf SSb}.
Then for $0 \leq r \leq n$
\begin{align*}
F^+ u_r &=
 \begin{cases}
  - \frac{1} {(1-k_0^2)k_0^{-1}k_1k_2k_3 q^{r+1}} 
        \big( (1-k_0^{-1} k_1 k_2 k_3 q^{r+1})u_r - u_{r+1} \big)
  & \text{ if $r$ is even},
 \\
  - \frac{1-k_0k_1k_2k_3q^r} {(1-k_0^2)k_0^{-1}k_1k_2k_3q^r} 
       \big( (1-k_0^{-1}k_1 k_2 k_3 q^r) u_{r-1} - u_r \big)
  & \text{ if $r$ is odd},
 \end{cases}
\\
F^- u_r &=
 \begin{cases}
    \frac{1} {(1-k_0^2)k_0^{-1}k_1k_2k_3q^{r+1}} 
       \big( (1-k_0 k_1 k_2 k_3 q^{r+1} ) u_r - u_{r+1} \big)
  & \text{ if $r$ is even},
 \\
   \frac{1-k_0^{-1}k_1k_2k_3 q^r} {(1-k_0^2)k_0^{-1}k_1k_2k_3q^r} 
     \big( (1-k_0 k_1 k_2 k_3 q^r) u_{r-1} - u_r \big)
    & \text{ if $r$ is odd}.
 \end{cases}
\end{align*}
\end{lemma}

We obtain a basis for $\V(k_0^{\pm 1})$ by applying $F^{\pm}$ to an appropriate 
subset of $\{u_r\}_{r=0}^d$ as follows.

\begin{lemma}    \label{lem:F+VF-Vbasis}    \samepage
\ifDRAFT {\rm lem:F+VF-Vbasis}. \fi
The subspaces $\V(k_0^{\pm 1})$ have the following bases:
\begin{equation}       \label{eq:basisF+VF-V}
\begin{array}{c|l|l}
\textup{\rm $X$-type of $\V$} & \qquad \text{\rm Basis for $\V(k_0)$} & \text{\rm Basis for $\V(k_0^{-1})$}
\\ \hline
\textup{\sf DS} &        \rule{0mm}{3ex}
 \{ F^+ u_{2r}\}_{r=0}^{n/2} & \{ F^- u_{2r} \}_{r=1}^{n/2}
\\
\textup{\sf DDa} & 
\{F^+ u_{2r}\}_{r=0}^{(n-1)/2} \cup \{F^+ u_n\}   &   \{ F^- u_{2r} \}_{r=1}^{(n-1)/2}
\\
\textup{\sf DDb} &
\{F^+ u_{2r}\}_{r=0}^{(n-1)/2}   &   \{ F^- u_{2r+1} \}_{r=0}^{(n-1)/2}
\\
\textup{\sf SSa} &
\{F^+ u_{2r}\}_{r=0}^{(n-1)/2}  &  \{ F^- u_{2r+1} \}_{r=0}^{(n-1)/2}
\\
\textup{\sf SSb} &
\{F^+ u_{2r}\}_{r=0}^{(n-1)/2}  &  \{ F^- u_{2r} \}_{r=0}^{(n-1)/2}
\end{array}
\end{equation}
\end{lemma}

\begin{proof}
First assume that $\V$ has $X$-type \textup{\sf DS}.
We first show that $\{F^+ u_{2r}\}_{r=0}^{n/2}$ is a basis for $\V(k_0)$.
By Corollary \ref{cor:dim} the dimension of $\V(k_0)$ is $n/2 + 1$.
So it suffices to show that $\{F^+ u_{2r}\}_{r=0}^{n/2}$ are linearly independent.
By Lemma \ref{lem:typekr} neither of $\pm k_0$ is among $q^{-1}, q^{-2}, \ldots, q^{-n/2}$.
So $k_0^2 q^{2r} \neq 1$ for $1 \leq r \leq n/2$.
By the assumption, $k_0^2 \neq 1$.
By these comments $1 - k_0^2 q^{2r}$ is nonzero for $0 \leq r \leq n/2$.
By this and \eqref{eq:DSF+ur}
$F^+ u_{2r}$ is nonzero for $0 \leq r \leq n/2$.
By this and \eqref{eq:DSF+ur} the vectors $\{F^+ u_{2r}\}_{r=0}^{n/2}$ are linearly independent.
Next we show that $\{F^- u_{2r}\}_{r=1}^{n/2}$ is a basis for $\V(k_0^{-1})$.
By Corollary \ref{cor:dim} the dimension of $\V(k_0^{-1})$ is $n/2$.
So it suffices to show that $\{F^- u_{2r}\}_{r=1}^{n/2}$ are linearly independent.
Observe by \eqref{eq:DSF-ur} that $F^- u_{2r}$ is nonzero for $1 \leq r \leq n/2$.
By this and \eqref{eq:DSF-ur} the vectors $\{F^- u_{2r}\}_{r=1}^{n/2}$
are linearly independent.
We have shown the result when $\V$ has $X$-type \textup{\sf DS}.
The proof is similar for the other types.
\end{proof}

\section{The action of $\A$ on $\V(k_0^{\pm 1})$}
\label{sec:Aaction}

Throughout this section Notation \ref{notation} is in effect.
Assume $k_0 \neq k_0^{-1}$.
Let the subspaces $\V(k_0^{\pm 1})$ be from \eqref{eq:defVk0},
and the elements $F^{\pm}$ be from \eqref{eq:defF+F-}.
Consider the basis \eqref{eq:basisF+VF-V} for $\V(k_0^{\pm 1})$.
In this section, we obtain the action of $\A$ on this basis.
The following five Lemmas can be routinely obtained
sing Lemmas \ref{lem:DSYaction}--\ref{lem:SSbYaction} and
Lemmas \ref{lem:DSFur}--\ref{lem:SSbFur}.

\begin{lemma}    \label{lem:DSAFur}    \samepage
\ifDRAFT {\rm lem:DSAFur}. \fi
Assume that $\V$ has $X$-type \textup{\sf DS}.
Then for $0 \leq r \leq n/2$
\begin{align*}
 \A F^+ u_{2r} &=
  \left(  k_0 k_1 q^{2r} + \frac{1}{k_0 k_1 q^{2r}} \right) F^+ u_{2r}
 - \frac{q^2 (1-k_0^2 q^{2r}) }{ k_0 k_1 q^{2r+2} (1-k_0^2 q^{2r+2}) } F^+ u_{2r+2},
\end{align*}
and for $1 \leq r \leq n/2$
\begin{align*}
 \A F^- u_{2r} &=
  \left(  k_0 k_1 q^{2r} + \frac{1}{k_0 k_1 q^{2r}} \right) F^- u_{2r}
 - \frac{q^2 (1- q^{2r}) }{ k_0 k_1 q^{2r+2} (1- q^{2r+2}) } F^- u_{2r+2}.
\end{align*}
\end{lemma}

\begin{lemma}   \label{lem:DDaAFur}    \samepage
\ifDRAFT {\rm lem:DDaAFur}. \fi
Assume that $\V$ has $X$-type \textup{\sf DDa}.
Then for $0 \leq r \leq (n-3)/2$
{\small
\begin{align*}
\A F^+ u_{2r} &=
 \left( k_0 k_1 q^{2r} + \frac{1}{k_0 k_1 q^{2r}} \right) F^+ u_{2r}
 - \frac{1- q^{n-2r+1} } {k_0 k_1 q^{2r+2} (1-q^{n-2r-1}) } F^+ u_{2r+2},
\\
\A F^+ u_{n-1} &=
 \left( k_0 k_1 q^{n-1} + \frac{1}{k_0 k_1 q^{n-1}} \right) F^+ u_{n-1}
 + \frac{1-q^2}{k_0 k_1 q^{n+1} } F^+ u_n,
\\
\A F^+ u_n &=
  \left(  k_0 k_1 q^{n+1} + \frac{1}{ k_0 k_1 q^{n+1} } \right) F^+ u_n,
\\
\intertext{and for $1 \leq r \leq (n-1)/2$}
 \A F^- u_{2r} &=
 \left( k_0 k_1 q^{2r} + \frac{1}{ k_0 k_1 q^{2r} } \right) F^- u_{2r}
 - \frac{ 1-q^{2r} }{ k_0 k_1 q^{2r} (1-q^{2r+2}) } F^- u_{2r+2}
\end{align*}
}
\end{lemma}

\begin{lemma}   \label{lem:DDbAFur}    \samepage
\ifDRAFT {\rm lem:DDbAFur}. \fi
Assume that $\V$ has $X$-type \textup{\sf DDb}.
Then for $0 \leq r \leq (n-1)/2$
{\small
\begin{align*}
\A F^+ u_{2r} &=
 \left( k_2 k_3 q^{2r+1} + \frac{1}{k_2 k_3 q^{2r+1}} \right) F^+ u_{2r}
 - k_2 k_3 q^{2r+1} F^+ u_{2r+2},
\\
 \A F^- u_{2r+1} &=
 \left( k_2 k_3 q^{2r+1} + \frac{1}{ k_2 k_3 q^{2r+1} } \right) F^- u_{2r+1}
 - k_2 k_3 q^{2r+3} F^- u_{2r+3}.
\end{align*}
}
\end{lemma}

\begin{lemma}   \label{lem:SSaAFur}    \samepage
\ifDRAFT {\rm lem:SSaAFur}. \fi
Assume that $\V$ has $X$-type \textup{\sf SSa}.
Then for $0 \leq r \leq (n-1)/2$
{\small
\begin{align*}
\A F^+ u_{2r} &=
 \left( k_0 k_1 q^{2r} + \frac{1}{k_0 k_1 q^{2r}} \right) F^+ u_{2r}
 - k_0 k_1 q^{2r+2} F^+ u_{2r+2},
\\
 \A F^- u_{2r+1} &=
 \left( k_0 k_1 q^{2r+2} + \frac{1}{ k_0 k_1 q^{2r+2} } \right) F^- u_{2r+1}
 - k_0 k_1 q^{2r+2} F^- u_{2r+3}.
\end{align*}
}
\end{lemma}

\begin{lemma}   \label{lem:SSbAFur}    \samepage
\ifDRAFT {\rm lem:SSbAFur}. \fi
Assume that $\V$ has $X$-type \textup{\sf SSb}.
Then for $0 \leq r \leq (n-1)/2$
{\small
\begin{align*}
\A F^+ u_{2r} &=
 \left( k_2 k_3 q^{2r+1} + \frac{1}{ k_2 k_3 q^{2r+1} } \right) F^+ u_{2r}
 - \frac{1}{ k_2 k_3 q^{2r+1} } F^+ u_{2r+2},
\\
\A F^- u_{2r} &=
 \left( k_2 k_3 q^{2r+1} + \frac{1}{ k_2 k_3 q^{2r+1} } \right) F^- u_{2r}
 - \frac{1}{ k_2 k_3 q^{2r+1} } F^- u_{2r+2}.
\end{align*}
}
\end{lemma}

By Lemmas \ref{lem:DSAFur}--\ref{lem:SSbAFur} we obtain the
following corollary.

\begin{corollary}   \label{cor:Aeigen}    \samepage
\ifDRAFT {\rm cor:Aeigen}. \fi
With reference to Notation \ref{notation2}, the following hold.
\begin{itemize}
\item[\rm (i)]
Consider the basis for $\V(k_0)$ from \eqref{eq:basisF+VF-V}.
With respect to this basis the matrix representing $\A$ is lower
bidiagonal with the following $(r,r)$-entry for $0 \leq r \leq d$:
\begin{equation}                                           \label{eq:Aeigen+}
\begin{array}{c|c}
\textup{\rm $X$-type of $\V$} & \text{\rm $(r,r)$-entry for $\A$} 
\\ \hline
\textup{\sf DS}, \textup{\sf DDa}, \textup{\sf SSa} &         \rule{0mm}{3ex}
k_0 k_1 q^{2r} + \frac{1}{k_0 k_1 q^{2r}}  
\\
\textup{\sf DDb}, \textup{\sf SSb} &         \rule{0mm}{3ex}
k_2 k_3 q^{2r+1} + \frac{1}{k_2 k_3 q^{2r+1} } 
\end{array}
\end{equation}
\item[\rm (ii)]
Consider the basis for $\V(k_0^{-1})$ from \eqref{eq:basisF+VF-V}.
With respect to this basis the matrix representing $\A$ is lower
bidiagonal with the following $(r,r)$-entry for $0 \leq r \leq d'$:
\begin{equation}                                           \label{eq:Aeigen-}
\begin{array}{c|c}
\textup{\rm $X$-type of $\V$} & \text{\rm $(r,r)$-entry for $\A$} 
\\ \hline
\textup{\sf DS}, \textup{\sf DDa} &         \rule{0mm}{3ex}
 k_0 k_1 q^{2r+2} + \frac{1}{k_0 k_1 q^{2r+2} }   
\\
\textup{\sf DDb}, \textup{\sf SSb}   &         \rule{0mm}{3ex}
 k_2 k_3 q^{2r+1} + \frac{1}{k_2 k_3 q^{2r+1} }  
\\
\textup{\sf SSa} &         \rule{0mm}{3ex}
 k_0 k_1 q^{2r} + \frac{1}{k_0 k_1 q^{2r} }  
\end{array}
\end{equation}
\end{itemize}
\end{corollary}

\section{The action of $\B$ on $\V(k_0^{\pm 1})$}
\label{sec:Baction}

Throughout this section Notation \ref{notation} is in effect.
Assume $k_0 \neq k_0^{-1}$.
Let the subspaces $\V(k_0^{\pm 1})$ be from \eqref{eq:defVk0},
and the elements $F^{\pm}$ be from \eqref{eq:defF+F-}.
Consider the basis \eqref{eq:basisF+VF-V} for $\V(k_0^{\pm 1})$.
In this section
we obtain the action of $\B$ on this basis.
We use the normalized basis $\{u'_r\}_{r=0}^n$ from Definition \ref{def:udr}.
The following five Lemmas can be routinely obtained
using Lemmas \ref{lem:DSDDa}--\ref{lem:SSb} and \ref{lem:DSFur}--\ref{lem:SSbFur}.
For notational convenience, set $u'_r =0$ for $r < 0$.

\begin{lemma}    \label{lem:DSBFur}    \samepage
\ifDRAFT {\rm lem:DSBFur}. \fi
Assume that $\V$ has $X$-type \textup{\sf DS}.
Then for $0 \leq r \leq n/2$
\begin{align*}
 \B F^+ u'_{2r} &=
  \left(  k_0 k_3 q^{2r} + \frac{1}{k_0 k_3 q^{2r}} \right) F^+ u'_{2r}
 - \frac{1}{ k_0 k_3 q^{2r} } F^+ u'_{2r-2},
\end{align*}
and for $1 \leq r \leq n/2$
\begin{align*}
 \B F^- u'_{2r} &=
  \left(  k_0 k_3 q^{2r} + \frac{1}{k_0 k_3 q^{2r}} \right) F^- u'_{2r}
 - \frac{1}{ k_0 k_3 q^{2r} } F^- u'_{2r-2}.
\end{align*}
\end{lemma}

\begin{lemma}   \label{lem:DDaBFur}    \samepage
\ifDRAFT {\rm lem:DDaBFur}. \fi
Assume that $\V$ has $X$-type \textup{\sf DDa}.
Then for $0 \leq r \leq (n-1)/2$
{\small
\begin{align*}
\B F^+ u'_{2r} &=
 \left( k_0 k_3 q^{2r} + \frac{1}{k_0 k_3 q^{2r}} \right) F^+ u'_{2r}
 - \frac{1} {k_0 k_3 q^{2r} } F^+ u'_{2r-2},
\\
\B F^+ u'_n &=
  \left(  k_0 k_3 q^{n+1} + \frac{1}{ k_0 k_3 q^{n+1} } \right) F^+ u'_n
 - \frac{1-q^{n+1}}{k_0 k_3 q^{n+1}} F^+ u'_{n-1}
\\
\intertext{and for $1 \leq r \leq (n-1)/2$}
 \B F^- u'_{2r} &=
 \left( k_0 k_3 q^{2r} + \frac{1}{ k_0 k_3 q^{2r} } \right) F^- u'_{2r}
 - \frac{ 1 }{ k_0 k_3 q^{2r} } F^- u'_{2r-2}
\end{align*}
}
\end{lemma}

\begin{lemma}   \label{lem:DDbBFur}    \samepage
\ifDRAFT {\rm lem:DDbBFur}. \fi
Assume that $\V$ has $X$-type \textup{\sf DDb}.
Then for $0 \leq r \leq (n-1)/2$
{\small
\begin{align*}
\B F^+ u'_{2r} &=
 \left( k_0 k_3 q^{2r} + \frac{1}{k_0 k_3 q^{2r}} \right) F^+ u'_{2r}
 - \frac{1}{ k_0^3 k_3^3 q^{6r-2}} F^+ u'_{2r-2},
\\
 \B F^- u'_{2r+1} &=
 \left( k_0 k_3 q^{2r+2} + \frac{1}{ k_0 k_3 q^{2r+2} } \right) F^- u'_{2r+1}
 -\frac{1}{k_0^3 k_3^3 q^{6r+2}} F^- u'_{2r-1}.
\end{align*}
}
\end{lemma}

\begin{lemma}   \label{lem:SSaBFur}    \samepage
\ifDRAFT {\rm lem:SSaBFur}. \fi
Assume that $\V$ has $X$-type \textup{\sf SSa}.
Then for $0 \leq r \leq (n-1)/2$
{\small
\begin{align*}
\B F^+ u'_{2r} &=
 \left( k_1 k_2 q^{2r+1} + \frac{1}{k_1 k_2 q^{2r+1}} \right) F^+ u'_{2r}
 - \frac{1}{k_1^3 k_2^3 q^{6r-1}} F^+ u'_{2r-2},
\\
 \B F^- u'_{2r+1} &=
 \left( k_1 k_2 q^{2r+1} + \frac{1}{ k_1 k_2 q^{2r+1} } \right) F^- u'_{2r+1}
 - \frac{1}{k_1^3 k_2^3 q^{6r+1}} F^- u'_{2r-1}.
\end{align*}
}
\end{lemma}

\begin{lemma}   \label{lem:SSbBFur}    \samepage
\ifDRAFT {\rm lem:SSbBFur}. \fi
Assume that $\V$ has $X$-type \textup{\sf SSb}.
Then for $0 \leq r \leq (n-1)/2$
{\small
\begin{align*}
\B F^+ u'_{2r} &=
 \left( k_1 k_2 q^{2r+1} + \frac{1}{ k_1 k_2 q^{2r+1} } \right) F^+ u'_{2r}
 - \frac{1 - k_0^{-1} k_1 k_2 k_3 q^{2r+1}}{ k_1 k_2 q^{2r+1}(1-k_0^{-1}k_1k_2k_3 q^{2r-1}) } F^+ u'_{2r-2},
\\
\B F^- u'_{2r} &=
 \left( k_1 k_2 q^{2r+1} + \frac{1}{ k_1 k_2 q^{2r+1} } \right) F^- u'_{2r}
 - \frac{1-k_0k_1k_2k_3q^{2r+1}}{ k_1 k_2 q^{2r+1} (1 - k_0k_1k_2k_3 q^{2r-1}) } F^- u'_{2r-2}.
\end{align*}
}
\end{lemma}

By Lemmas \ref{lem:DSBFur}--\ref{lem:SSbBFur} we obtain the following
corollary.

\begin{corollary}   \label{cor:Beigen}    \samepage
\ifDRAFT {\rm cor:Beigen}. \fi
With reference to Notation \ref{notation2},
the following hold.
\begin{itemize}
\item[\rm (i)]
Consider the basis for $\V(k_0)$ from \eqref{eq:basisF+VF-V}.
With respect to this basis the matrix representing $\B$ is upper
bidiagonal with the following $(r,r)$-entry for $0 \leq r \leq d$:
\begin{equation}                  \label{eq:Beigen+}
 \begin{array}{c|c}
  \textup{\rm $X$-type of $\V$} & \text{\rm $(r,r)$-entry for $\B$}
\\ \hline
  \textup{\sf DS}, \textup{\sf DDa}, \textup{\sf DDb} &   \rule{0mm}{3ex}
  k_0 k_3 q^{2r} + \frac{1}{k_0 k_3 q^{2r} }
\\
  \textup{\sf SSa}, \textup{\sf SSb} &   \rule{0mm}{2.5ex}
  k_1 k_2 q^{2r+1} + \frac{1}{k_1 k_2 q^{2r+1} }
 \end{array}
\end{equation}
\item[\rm (ii)]
Consider the basis for $\V(k_0^{-1})$ from \eqref{eq:basisF+VF-V}.
With respect to this basis the matrix representing $\B$ is upper
bidiagonal with the following $(r,r)$-entry for $0 \leq r \leq d'$:
\begin{equation}                      \label{eq:Beigen-}
 \begin{array}{c|c}
  \textup{\rm $X$-type of $\V$} & \text{\rm $(r,r)$-entry for $\B$}
\\ \hline
  \textup{\sf DS}, \textup{\sf DDa}, \textup{\sf DDb} &   \rule{0mm}{3ex}
  k_0 k_3 q^{2r+2} + \frac{1}{k_0 k_3 q^{2r+2} }
\\
  \textup{\sf SSa}, \textup{\sf SSb} &  \rule{0mm}{2.5ex}
  k_1 k_2 q^{2r+1} + \frac{1}{k_1 k_2 q^{2r+1} }
 \end{array}
\end{equation}
\end{itemize}
\end{corollary}

\section{The equitable Askey-Wilson relations}
\label{sec:AW}

In this section we recall some relations concerning a Leonard
pair of $q$-Racah type.

\begin{lemma}  {\rm (See \cite[Theorem 10.1]{Hu}.)} \label{lem:Z4AW} \samepage
\ifDRAFT {\rm lem:Z4AW}. \fi
Let $V$ denote a vector space over $\F$ with dimension $d+1$, $d \geq 0$.
Let $A,A^*$ denote a Leonard pair on $V$ that has $q$-Racah type.
Let $(a,b,c,d)$ denote a Huang data of $A,A^*$.
Then there exists a unique $\F$-linear transformation $A^\ve : V \to V$ such that
\begin{align}
 A + \frac{q A^* A^\ve - q^{-1}A^\ve A^*}{q^2-q^{-2}} &= 
      \frac{(q^{d+1}+q^{-d-1})(a+a^{-1}) + (b+b^{-1}) (c+c^{-1})}{q+q^{-1}} I,     \label{eq:AW1} \\
 A^* + \frac{q A^\ve A - q^{-1}A A^\ve}{q^2-q^{-2}} &=
      \frac{(q^{d+1}+q^{-d-1}) (b+b^{-1}) + (c+c^{-1})(a+a^{-1})} {q+q^{-1}} I,       \label{eq:AW2} \\
 A^\ve + \frac{q A A^* - q^{-1}A^* A}{q^2-q^{-2}} &=
       \frac{(q^{d+1}+q^{-d-1})(c+c^{-1}) + (a+a^{-1})(b+b^{-1})}{q+q^{-1}} I.        \label{eq:AW3}
\end{align}
\end{lemma}

The relations \eqref{eq:AW1}--\eqref{eq:AW3} are called the 
{\em equitable Askey-Wilson relations}.

\begin{note}
In \eqref{eq:AW1}--\eqref{eq:AW3} the right-hand side is invariant
when we replace any of $a$, $b$, $c$ by its inverse.
Therefore $A^\ve$ does not depend on the choice of Huang data.
\end{note}

\begin{lemma}    \label{lem:Z4AW2}     \samepage
\ifDRAFT {\rm lem:Z4AW2}. \fi
With reference to Lemma \ref{lem:Z4AW},
assume $d \geq 1$. 
Then $A^\ve$ is the unique $\F$-linear transformation such that
each of the following  is a scalar multiple of the identity:
\begin{align}
 A + \frac{q A^* A^\ve - q^{-1}A^\ve A^*}{q^2-q^{-2}}, & &
 A^* + \frac{q A^\ve A - q^{-1}A A^\ve}{q^2-q^{-2}}, & &
 A^\ve + \frac{q A A^* - q^{-1}A^* A}{q^2-q^{-2}}.
\end{align}
\end{lemma}

\begin{proof}
Let $\alpha$, $\beta$, $\gamma \in \F$ denote the scalars on the right in
\eqref{eq:AW1}, \eqref{eq:AW2}, \eqref{eq:AW3}, respectively.
Assume that there exists an $\F$-linear transformation $\tilde{A}^\ve : V \to V$
and scalars $\tilde{\alpha}$, $\tilde{\beta}$, $\tilde{\gamma} \in \F$ such that
\begin{align}
 A + \frac{q A^* \tilde{A}^\ve - q^{-1} \tilde{A}^\ve A^*}{q^2-q^{-2}} 
                                                         &= \tilde{\alpha} I,              \label{eq:AW1b}    \\
 A^* + \frac{q \tilde{A}^\ve A - q^{-1}A \tilde{A}^\ve}{q^2-q^{-2}} 
                                                        &= \tilde{\beta} I,               \label{eq:AW2b}   \\
 \tilde{A}^\ve + \frac{q A A^* - q^{-1}A^* A}{q^2-q^{-2}} 
                                                    &= \tilde{\gamma} I.                  \label{eq:AW3b}
\end{align}
We show that $A^\ve = \tilde{A}^\ve$.
By \eqref{eq:AW3} and \eqref{eq:AW3b}
\begin{equation}
  \tilde{A}^\ve = A^\ve + (\tilde{\gamma} - \gamma) I.                           \label{eq:AWaux}
\end{equation}
In \eqref{eq:AW1b} eliminate $\tilde{A}^\ve$ using this, and simplify the result to find
\[
 A + \frac{q A^* A^\ve - q^{-1} A^\ve A^*}{q^2-q^{-2}} 
  + \frac{(\tilde{\gamma}-\gamma) A^*}{q+q^{-1}} = \tilde{\alpha} I.
\]
By this and \eqref{eq:AW1} 
\[
   \alpha I +  \frac{(\tilde{\gamma}-\gamma) A^*}{q+q^{-1}} = \tilde{\alpha} I,
\]
and this becomes
\[
   (\tilde{\gamma} - \gamma) A^* = (q+q^{-1})(\tilde{\alpha} - \alpha) I.         \label{eq:AWaux1}
\]
This forces $\tilde{\gamma} = \gamma$ by our assumption $d \geq 1$.
Now $A^\ve = \tilde{A}^\ve$ follows from \eqref{eq:AWaux}.
The result follows.
\end{proof}

\section{Huang data for the Leonard pairs obtained from a feasible $\Hq$-module}
\label{sec:Huang}

Throughout this section  Notation \ref{notation} is in effect.
Assume that  $\V$ is feasible.
By Theorem \ref{thm:main1} we have a Leonard pair $\A,\B$ on $\V(k_0)$
(resp.\ $\V(k_0^{-1})$) that has  $q$-Racah type.
Our target in this section is to prove the following result.

\begin{proposition}    \label{prop:Huangdata}    \samepage
\ifDRAFT {\rm prop:Huangdata}. \fi
The following hold.
\item[\rm (i)]
The Leonard pair $\A,\B$ on $\V(k_0)$ has 
a Huang data $(a,b,c,d)$ such that
\begin{equation}     \label{eq:Huangdata+}
\begin{array}{c|cccc}
 \textup{\rm $X$-type of $\V$} & a & b & c & d
\\ \hline
\textup{\sf DS} &
 k_0k_1q^{n/2} &  k_0k_3q^{n/2} &  k_0k_2 q^{n/2}  & n/2 \\
\textup{\sf DDa} &
  k_1 & k_3 & k_2 & (n+1)/2 \\
\textup{\sf DDb} &
  k_2 &  k_0 q^{-1} & k_1 &  (n-1)/2 \\
\textup{\sf SSa} &
 k_0 q^{-1} & k_2 & k_3 & (n-1)/2  \\
\textup{\sf SSb} &
 k_3 & k_1 & k_0 q^{-1} & (n-1)/2 
\end{array}
\end{equation}
\item[\rm (ii)]
The Leonard pair $\A,\B$ on $\V(k_0^{-1})$ 
has a Huang data $(a',b',c',d')$ such that
\begin{equation}     \label{eq:Huangdata-}
\begin{array}{c|cccc}
 \textup{\rm $X$-type of $\V$} & a' & b' & c' & d'
\\ \hline
\textup{\sf DS} &
 k_0k_1q^{(n+2)/2} & k_0k_3q^{(n+2)/2} & k_0k_2q^{(n+2)/2} & (n-2)/2  \\
\textup{\sf DDa} &
 k_1 & k_3 & k_2 & (n-3)/2  \\
\textup{\sf DDb} &
  k_2 &  k_0 q & k_1  & (n-1)/2  \\
\textup{\sf SSa} &
  k_0 q & k_2 & k_3 & (n-1)/2  \\
\textup{\sf SSb} &
  k_3 & k_1 & k_0 q & (n-1)/2
\end{array}
\end{equation}
\end{proposition}

\begin{lemma}    \label{lem:thr}   \samepage
\ifDRAFT {\rm lem:thr}. \fi
Define $d = \dim \V(k_0)-1$ and $d' = \dim \V(k_0^{-1})-1$.
\begin{itemize}
\item[\rm (i)]
Consider the Leonard pair $\A,\B$ on $\V(k_0)$.
Define scalars $\{\th_r\}_{r=0}^d$ as follows:
\begin{equation}   \label{eq:defthr}
\begin{array}{c|c}
\textup{\rm $X$-type of $\V$} & \text{\rm Definition of $\th_r$}
\\ \hline
\textup{\sf DS}, \textup{\sf DDa}, \textup{\sf SSa}   \rule{0mm}{3ex}
  &   k_0 k_1 q^{2r} + \frac{1}{k_0 k_1 q^{2r} }  
\\
\textup{\sf DDb}, \textup{\sf SSb}   \rule{0mm}{2.5ex}
& k_2 k_3 q^{2r+1} + \frac{1}{k_2 k_3 q^{2r+1} }
\end{array}
\end{equation}
Then $\{\th_r\}_{r=0}^d$ is a standard ordering of the eigenvalues of $\A$.
\item[\rm (i)]
Consider the Leonard pair $\A,\B$ on $\V(k_0^{-1})$.
Define scalars $\{\th'_r\}_{r=0}^d$ as follows:
\begin{equation}   \label{eq:defthdr}
\begin{array}{c|c}
\textup{\rm $X$-type of $\V$} & \text{\rm Definition of $\th'_r$}
\\ \hline
\textup{\sf DS}, \textup{\sf DDa}     \rule{0mm}{3ex}
  &   k_0 k_1 q^{2r+2} + \frac{1}{k_0 k_1 q^{2r+2} }  
\\
\textup{\sf DDb}, \textup{\sf SSb}      \rule{0mm}{2.5ex}
& k_2 k_3 q^{2r+1} + \frac{1}{k_2 k_3 q^{2r+1} }
\\
\textup{\sf SSa}                       \rule{0mm}{2.5ex}
 & k_0 k_1 q^{2r} + \frac{1}{k_0 k_1 q^{2r} }
\end{array}
\end{equation}
Then $\{\th'_r\}_{r=0}^d$ is a standard ordering of the eigenvalues of $\A$.
\end{itemize}
\end{lemma}

\begin{proof}
(i):
By Corollaries \ref{cor:Aeigen} and \ref{cor:Beigen} there exists a basis for $\V(k_0)$
with respect to which the matrix representing $\A$ is lower bidiagonal
and the matrix representing $\B$ is upper bidiagonal.
Moreover, the diagonal entries of the matrix representing $\A$ are given
in \eqref{eq:Aeigen+}.
Now the result follows by Lemma \ref{lem:LBUB}.

(ii): 
Similar.
\end{proof}

\begin{lemma}    \label{lem:thsr}   \samepage
\ifDRAFT {\rm lem:thsr}. \fi
Define $d = \dim \V(k_0)-1$ and $d' = \dim \V(k_0^{-1})-1$.
\begin{itemize}
\item[\rm (i)]
Consider the Leonard pair $\A,\B$ on $\V(k_0)$.
Define scalars $\{\th^*_r\}_{r=0}^d$ as follows:
\begin{equation}   \label{eq:defthsr}
\begin{array}{c|c}
  \textup{\rm $X$-type of $\V$} & \text{\rm Definition of $\th^*_r$}
\\ \hline
  \textup{\sf DS}, \textup{\sf DDa}, \textup{\sf DDb} &   \rule{0mm}{3ex}
  k_0 k_3 q^{2r} + \frac{1}{k_0 k_3 q^{2r} }
\\
  \textup{\sf SSa}, \textup{\sf SSb} &   \rule{0mm}{2.5ex}
  k_1 k_2 q^{2r+1} + \frac{1}{k_1 k_2 q^{2r+1} }
 \end{array}
\end{equation}
Then $\{\th^*_r\}_{r=0}^d$ is a standard ordering of the eigenvalues of $\B$.
\item[\rm (i)]
Consider the Leonard pair $\A,\B$ on $\V(k_0^{-1})$.
Define scalars $\{\th^{* \prime}_r\}_{r=0}^d$ as follows:
\begin{equation}   \label{eq:defthsdr}
 \begin{array}{c|c}
  \textup{\rm $X$-type of $\V$} & \text{\rm Definition of $\th^{* \prime}_r$}
\\ \hline
  \textup{\sf DS}, \textup{\sf DDa}, \textup{\sf DDb} &   \rule{0mm}{3ex}
  k_0 k_3 q^{2r+2} + \frac{1}{k_0 k_3 q^{2r+2} }
\\
  \textup{\sf SSa}, \textup{\sf SSb} &  \rule{0mm}{2.5ex}
  k_1 k_2 q^{2r+1} + \frac{1}{k_1 k_2 q^{2r+1} }
 \end{array}
\end{equation}
Then $\{\th^{* \prime}_r\}_{r=0}^d$ is a standard ordering of the eigenvalues of $\B$.
\end{itemize}
\end{lemma}

\begin{proof}
Similar to the proof of Lemma \ref{lem:thr}.
\end{proof}

\begin{lemma} {\rm (See \cite[Proposition 6.6]{T:AWDAHA}.)}
\label{lem:ABC}   \samepage
\ifDRAFT {\rm lem:ABC}. \fi
Define $\C = t_0 t_2 + (t_0 t_2)^{-1}$.
Then the elements $\A,\B,\C$ are related as follows:
\begin{align}
  \A + \frac{q \B \C - q^{-1} \C \B}{q^2 - q^{-2}}
 &= \frac{(q^{-1}t_0 + q t_0^{-1}) T_1 + T_2 T_3}
            {q+q^{-1}},                                              \label{eq:A}
\\
   \B + \frac{q \C \A - q^{-1} \A \C}{q^2 - q^{-2}}
 &= \frac{(q^{-1}t_0 + q t_0^{-1}) T_3 + T_1 T_2}
            {q+q^{-1}},                                               \label{eq:B}
\\
   \C + \frac{q \A \B - q^{-1} \B \A}{q^2 - q^{-2}}
 &= \frac{(q^{-1}t_0 + q t_0^{-1}) T_2 + T_3 T_1}
            {q+q^{-1}}.                                              \label{eq:C}
\end{align}
\end{lemma}

\begin{lemma} \label{lem:L3Ld3} \samepage
\ifDRAFT {\rm lem:L3Ld3}. \fi
The following hold.
\begin{itemize}
\item[\rm (i)]
Let $(a,b,c,d)$ denote a Huang data for the Leonard pair $\A,\B$ on $\V(k_0)$.
Assume $d \geq 1$. 
Then
{\small
\[
 c+c^{-1} = 
  \frac{(q^{-1} k_0 + q k_0^{-1}) (k_2+k_2^{-1}) + (k_1+k_1^{-1})(k_3+k_3^{-1}) - (a+a^{-1})(b+b^{-1}) }
       {q^{d+1}+q^{-d-1}}.  
\]
}
\item[\rm (i)]
Let $(a',b',c',d')$ denote a Huang data for the Leonard pair $\A,\B$ on $\V(k_0^{-1})$.
Assume $d' \geq 1$. 
Then
{\small 
\[
 c' + {c'}^{-1} = 
  \frac{(q k_0 + q^{-1} k_0^{-1}) (k_2+k_2^{-1}) + (k_1+k_1^{-1})(k_3+k_3^{-1}) - (a'+{a'}^{-1})(b'+{b'}^{-1}) }
       {q^{d'+1}+q^{-d'-1}}. 
\]
}
\end{itemize}
\end{lemma}

\begin{proof}
(i):
Note that $t_0$ acts on $\V(k_0)$ as $k_0$ times the identity.
By this and \eqref{eq:ki}, in each \eqref{eq:A}--\eqref{eq:C} the right-hand side acts on $\V(k_0)$
as a scalar multiple of the identity.
By this and Lemmas \ref{lem:Z4AW}, \ref{lem:Z4AW2} the action of the right-hand side
of \eqref{eq:C} is equal to the action of the right-hand side of \eqref{eq:AW3}.
From this we obtain the result.

(ii):
Similar.
\end{proof}

\begin{proofof}{Proposition \ref{prop:Huangdata}}
(i):
First assume that $\V$ has $X$-type \textup{\sf DS}.
For notational convenience, 
Set $V = \V(k_0)$,
and let $A : V \to V$ (resp.\ $A^* : V \to V$) denote the 
$\F$-linear transformation that is induced by the action of $\A$ (resp.\ $\B$).
By Corollary \ref{cor:dim} the Leonard pair $A,A^*$ has diameter $d=n/2$.
Define scalars $\{\th_r\}_{r=0}^d$ and $\{\th^*_r\}_{r=0}^d$ as
\begin{align}
  \th_r &= k_0 k_1 q^{2r} + \frac{1}{k_0 k_1 q^{2r}},
&
  \th^*_r &= k_0 k_3 q^{2r} + \frac{1}{k_0 k_3 q^{2r}}  &&  (0 \leq r \leq d).     \label{eq:DSthrthsr}
\end{align}
By Lemmas \ref{lem:thr} (resp.\ Lemma \ref{lem:thsr})  the sequence $\{\th_r\}_{r=0}^d$
(resp.\ $\{\th^*_r\}_{r=0}^d$)
is a standard ordering of the eigenvalues of $A$ (resp.\ $A^*$).
Now define scalars
\begin{align*}
  a &= k_0 k_1 q^{d}, &
  b &= k_0 k_3 q^{d}, &
  c &= k_0 k_2 q^{d}. 
\end{align*}
By \eqref{eq:DSthrthsr}
\begin{align*}
  \th_r &= a q^{2r-d} + a^{-1} q^{d-2r},
&
  \th^*_r &= b q^{2r-d} + b^{-1} q^{d-2r}  && (0 \leq r \leq d). 
\end{align*}
By Lemma \ref{lem:typekr0}
$k_0k_1k_2k_3 = q^{-2d-1}$.
Using this, one checks that $a,b,c$ satisfy
the displayed equation in Lemma \ref{lem:L3Ld3}(i).
By this and the last sentence in Lemma \ref{lem:qRacah},
$(a,b,c,d)$ is a Huang data of $A,A^*$.
We have shown the result when $\V$ has $X$-type \textup{\sf DS}.
The proof is similar for the other types.

(ii):
Similar.
\end{proofof}

\begin{corollary}    \label{cor:ki}   \samepage
\ifDRAFT {\rm cor:ki}. \fi
The following hold.
\begin{itemize}
\item[\rm (i)]
For the Huang data $(a,b,c,d)$ from Proposition \ref{prop:Huangdata}(i) 
the parameters $\{k_i\}_{i \in \I}$ satisfy
\begin{equation*}
\begin{array}{c|ccccc}
\text{\rm Case} & k_0 & k_1 & k_2 & k_3 
\\ \hline
\rule{0mm}{4.5mm}
\textup{\sf DS}  
 & (abcq^{1-d})^{1/2} & aq^{-d}k_0^{-1} & cq^{-d}k_0^{-1} & bq^{-d}k_0^{-1}
\\
\rule{0mm}{4mm}
\textup{\sf DDa}  & q^{-d} &  a & c & b
\\
\rule{0mm}{4mm}
\textup{\sf DDb}  &  bq & c & a & q^{-d-1}
\\
\rule{0mm}{4mm}
\textup{\sf SSa}  & aq & q^{-d-1} & b & c
\\
\textup{\sf SSb}  & cq & b & q^{-d-1} & a
\end{array}
\end{equation*}
\item[\rm (ii)]
For the Huang data $(a',b',c',d')$ from Proposition \ref{prop:Huangdata}(ii)
the parameters $\{k_i\}_{i \in \I}$ satisfy
\begin{equation*}
\begin{array}{c|ccccc}
\text{\rm Case} & k_0 & k_1 & k_2 & k_3 
\\ \hline
\rule{0mm}{4.5mm}
\textup{\sf DS}  
 & (a' b' c' q^{-d'-3})^{1/2} & a' q^{-d'-2 }k_0^{-1} & c' q^{-d'-2} k_0^{-1} & b' q^{-d'-2}k_0^{-1}
\\
\rule{0mm}{4mm}
\textup{\sf DDa}  & q^{-d'-2} &  a' & c' & b'
\\
\rule{0mm}{4mm}
\textup{\sf DDb}  &  b' q^{-1} & c' & a' & q^{-d'-1}
\\
\rule{0mm}{4mm}
\textup{\sf SSa}  & a' q^{-1} & q^{-d'-1} & b' & c'
\\
\textup{\sf SSb}  & c' q^{-1} & b' & q^{-d'-1} & a'
\end{array}
\end{equation*}
\end{itemize}
\end{corollary}

\begin{proof}
Use Lemma \ref{lem:typekr0}.
\end{proof}

\begin{note}
In Corollary \ref{cor:ki}(i) the assertion $k_0 = (a b c q^{1-d})^{1/2}$ means
that $k_0$ is equal to one of the square roots of $a b c q^{1-d}$.
Similar for (ii).
\end{note}

\section{Proof of Theorem \ref{thm:main}; ``only if'' direction}
\label{sec:onlyif}

In this section we prove the ``only if'' direction of Theorem \ref{thm:main}.

\begin{lemma}  \label{lem:a2b2}  \samepage
\ifDRAFT {\rm lem:a2b2}. \fi
With reference to Notation \ref{notation}, assume that 
$\V$ is feasible and has $X$-type \textup{\sf DS}.
Define scalars $a$, $b$, $c$, $d$ as in \eqref{eq:Huangdata+}.
Then $a^2 \neq q^{-2d}$ and $b^2 \neq q^{-2d}$.
\end{lemma}

\begin{proof}
By Corollary \ref{cor:Ydiagonalizable} $k_0^2 k_1^2 \neq q^{-2n}$.
By Lemma \ref{lem:typekr} $k_0^2 k_3^2 \neq q^{-2n}$.
The result follows from these comments and \eqref{eq:Huangdata+}.
\end{proof}

\begin{lemma}   \label{lem:a2b2c2}  \samepage
\ifDRAFT {\rm lem:a2b2c2}. \fi
With reference to Notation \ref{notation}, assume that $\V$ is feasible
and has $X$-type among $\textup{\sf DDb}$, $\textup{\sf SSa}$, $\textup{\sf SSb}$.
Define scalars $a$, $b$, $c$, $d$ as in \eqref{eq:Huangdata+}.
Then the following inequalities hold:
\[
\begin{array}{c|ccc}
\textup{\rm $X$-type of $\V$} & \multicolumn{3}{c}{\text{\rm Inequalities}}
\\ \hline 
\rule{0mm}{4.3mm}
\textup{\sf DDb} & b^2 \neq q^{-2}  &  a^2 \neq q^{\pm 2d}
\\
\rule{0mm}{4mm}
\textup{\sf SSa} & a^2 \neq q^{-2} &  b^2 \neq q^{\pm 2d}
\\
\rule{0mm}{4mm}
\textup{\sf SSb} & \;\; c^2 \neq q^{-2} \;\; & \;\; a^2 \neq q^{\pm 2d} \;\; & \;\; b^2 \neq q^{\pm 2d}
\end{array}
\]
\end{lemma}

\begin{proof}
First assume that $\V$ has $X$-type \textup{\sf DDb}.
We have $k_0^2 \neq 1$ since $\V$ is feasible.
By Lemma \ref{lem:typekr0} $k_3^2 = q^{-n-1}$.
By Corollary \ref{cor:Ydiagonalizable} $k_2^2 k_3^2$ is not among
$q^{-2}$, $q^{-2n}$.
The result follows from these comments.
Next assume that $\V$ has $X$-type \textup{\sf SSa}.
By Lemma \ref{lem:typekr0} $k_1^2 = q^{-n-1}$.
By Corollary \ref{cor:Ydiagonalizable} $k_0^2 k_1^2 \neq q^{-n-1}$.
By Lemma \ref{lem:typekr} $k_2^2$ is not among  $q^{n-1}$, $q^{1-n}$.
The result follows from these comments.
Next assume that $\V$ has $X$-type \textup{\sf SSb}.
We have $k_0^2 \neq 1$ since $\V$ is feasible.
By Lemma \ref{lem:typekr0} $k_2^2 = q^{-n-1}$.
By Corollary \ref{cor:Ydiagonalizable} $k_2^2 k_3^2$ is not among $q^{-2}$, $q^{-2n}$.
By Lemma \ref{lem:typekr} $k_1^2$ is not among $q^{n-1}$, $q^{1-n}$.
The result follows from these comments.
\end{proof}

\begin{corollary}  \label{cor:Huangdata}  \samepage
\ifDRAFT {\rm cor:Huangdata}. \fi
With reference to Notation \ref{notation}, assume that $\V$ is feasible.
Then there exist a Huang data $(a,b,c,d)$ of $\A,\B$ on $\V(k_0)$ and a Huang data
$(a',b',c',d')$ of $\A,\B$ on $\V(k_0^{-1})$ such that the following hold:
\begin{equation}                 \label{eq:Huangdata5}
\begin{array}{c|c|c|c|c|ccc}
\textup{\rm $X$-type of $\V$} & d'-d & a'/a & b'/b & c'/c &  & \text{\rm Inequalities} &  
\\ \hline 
\rule{0mm}{4.3mm}
\textup{\sf DS} & -1 & q & q & q & \quad a^2 \neq q^{-2d} & \quad  b^2 \neq q^{-2d}
\\
\rule{0mm}{4mm}
\textup{\sf DDa} & -2 & 1 & 1 & 1
\\
\rule{0mm}{4mm}
\textup{\sf DDb} & 0 & 1 & q^2 & 1  & \quad a^2 \neq q^{\pm 2d} & & \quad b^2 \neq q^{-2}
\\
\rule{0mm}{4mm}
\textup{\sf SSa} & 0 & q^2 & 1 & 1 & & \quad b^2 \neq q^{\pm 2d} & \quad a^2 \neq q^{-2}
\\
\rule{0mm}{4mm}
\textup{\sf SSb} & 0 & 1 & 1 & q^2 & \quad a^2 \neq q^{\pm 2d} & \quad b^2 \neq q^{\pm 2d } & \quad c^2 \neq q^{-2} 
\end{array}
\end{equation}
\end{corollary}

\begin{proof}
Immediate from Proposition \ref{prop:Huangdata} and
Lemmas \ref{lem:a2b2}, \ref{lem:a2b2c2}.
\end{proof}

\begin{proofof}{Theorem \ref{thm:main}; ``only if'' direction}
Let $A,A^*$ (resp.\ $A',A^{*\prime}$) denote a Leonard pair on $V$ (resp.\ $V'$)
that has $q$-Racah type. 
Assume that these Leonard pairs are linked, so there exists a feasible $\Hq$-module
structure on $\V := V \oplus V'$ such that $V$, $V'$ are the eigenspaces of $t_0$
and \eqref{eq:linked} holds.
Let $\{k_i\}_{i \in \I}$ denote a parameter sequence of $\V$
that is consistent with a standard ordering of the eigenvalues of $X$.
First assume $V = \V(k_0)$ and $V' = \V(k_0^{-1})$.
By Corollary \ref{cor:Huangdata}
there exist a Huang data $(a,b,c,d)$ of $\A,\B$ on $\V(k_0)$ and a Huang data
$(a',b',c',d')$ of $\A,\B$ on $\V(k_0^{-1})$ that satisfy \eqref{eq:Huangdata5}.
Note by \eqref{eq:linked} that $(a,b,c,d)$ is a Huang data of $A,A^*$ and
$(a',b',c',d')$ is a Huang data of $A',A^{*\prime}$.
By \eqref{eq:Huangdata5}, for each $X$-type of $\V$,
these Huang data satisfy the following case in Theorem \ref{thm:main}:
\[
\begin{array}{c|ccccc}
 \textup{\rm $X$-type of $\V$} & 
 \textup{\sf DS} & \textup{\sf DDa} & \textup{\sf DDb} & \textup{\sf SSa} &  \textup{\sf SSb}
\\ \hline
 \text{Case} & \text{(ii)} & \text{(i)} & \text{(iv)} & \text{(iii)} & \text{(v)}
\end{array}
\]
Next assume $V=\V(k_0^{-1})$ and $V'=\V(k_0)$.
By Corollary \ref{cor:Huangdata} in which $(a,b,c,d)$ and $(a',b',c',d')$ are exchanged,
for each $X$-type of $\V$
there exist a Huang data $(a',b',c',d')$ of $\A,\B$ on $\V(k_0)$ and a Huang data
$(a,b,c,d)$ of $\A,\B$ on $\V(k_0^{-1})$ that satisfy
\[
\begin{array}{c|c|c|c|c|ccc}
\textup{\rm $X$-type of $\V$} & d'-d & a'/a & b'/b & c'/c &  & \text{\rm Inequalities} &  
\\ \hline 
\rule{0mm}{4.3mm}
\textup{\sf DS} & 1 & q^{-1} & q^{-1} & q^{-1} & \quad a^{\prime 2} \neq q^{-2d'} 
                                & \quad b^{\prime 2} \neq q^{-2d'}
\\
\rule{0mm}{4mm}
\textup{\sf DDa} & 2 & 1 & 1 & 1
\\
\rule{0mm}{4mm}
\textup{\sf DDb} & 0 & 1 & q^{-2} & 1  & \quad a^{\prime 2} \neq q^{\pm 2d'} & & \quad b^{\prime 2} \neq q^{-2}
\\
\rule{0mm}{4mm}
\textup{\sf SSa} & 0 & q^{-2} & 1 & 1 & & \quad b^{\prime 2} \neq q^{\pm 2d'} & \quad a^{\prime 2} \neq q^{-2}
\\
\rule{0mm}{4mm}
\textup{\sf SSb} & 0 & 1 & 1 & q^{-2} & \quad a^{\prime 2} \neq q^{\pm 2d'} 
                               & \quad b^{\prime 2} \neq q^{\pm 2d' } & \quad c^{\prime 2} \neq q^{-2} 
\end{array}
\]
In this table we rewrite the inequalities in terms of $a$, $b$, $c$, $d$:
\[
\begin{array}{c|c|c|c|c|ccc}
\textup{\rm $X$-type of $\V$} & d'-d & a'/a & b'/b & c'/c &  & \text{\rm Inequalities} &  
\\ \hline 
\rule{0mm}{4.3mm}
\textup{\sf DS} & 1 & q^{-1} & q^{-1} & q^{-1} 
                  & \quad a^2 \neq q^{-2d} & \quad b^2 \neq q^{-2d}
\\
\rule{0mm}{4mm}
\textup{\sf DDa} & 2 & 1 & 1 & 1
\\
\rule{0mm}{4mm}
\textup{\sf DDb} & 0 & 1 & q^{-2} & 1  & \quad a^2 \neq q^{\pm 2d} & & \quad b^2 \neq q^{2}
\\
\rule{0mm}{4mm}
\textup{\sf SSa} & 0 & q^{-2} & 1 & 1 & & \quad b^2 \neq q^{\pm 2d} & \quad a^2 \neq q^{2}
\\
\rule{0mm}{4mm}
\textup{\sf SSb} & 0 & 1 & 1 & q^{-2} & \quad a^2 \neq q^{\pm 2d} 
                               & \quad b^2 \neq q^{\pm 2d} & \quad c^2 \neq q^{2} 
\end{array}
\]
Thus case (vi) holds if $\V$ has $X$-type {\sf DS}, and case (vii) holds if $\V$ has $X$-type {\sf DDa}.
If $\V$ has $X$-type  among {\sf DDb}, {\sf SSa}, {\sf SSb},
we replace each of $a$, $b$, $c$, $a'$, $b'$, $c'$ with its inverse.
This gives
\[
\begin{array}{c|c|c|c|c|ccc}
\textup{\rm $X$-type of $\V$} & d'-d & a'/a & b'/b & c'/c &  & \text{\rm Inequalities} &  
\\ \hline 
\rule{0mm}{4mm}
\textup{\sf DDb} & 0 & 1 & q^2 & 1  & \quad a^{-2} \neq q^{\pm 2d} & & \quad b^{-2} \neq q^{2}
\\
\rule{0mm}{4mm}
\textup{\sf SSa} & 0 & q^2 & 1 & 1 & & \quad b^{-2} \neq q^{\pm 2d} & \quad a^{-2} \neq q^{2}
\\
\rule{0mm}{4mm}
\textup{\sf SSb} & 0 & 1 & 1 & q^2 & \quad a^{-2} \neq q^{\pm 2d} & \quad b^{-2} \neq q^{\pm 2d } & \quad c^{-2} \neq q^{2} 
\end{array}
\]
For each $X$-type among {\sf DDb}, {\sf SSa}, {\sf SSb} 
these Huang data satisfy the following case in Theorem \ref{thm:main}:
\[
\begin{array}{c|ccc}
 \textup{\rm $X$-type of $\V$} &   \textup{\sf DDb} & \textup{\sf SSa} &  \textup{\sf SSb}
\\ \hline
 \text{Case} & \text{(iv)} & \text{(iii)} & \text{(v)}
\end{array}
\]
The result follows.
\end{proofof}

\section{Construction of an XD $\Hq$-module}
\label{sec:const}

With reference to Notation \ref{notation},
let \textup{\sf T} denote the $X$-type of $\V$.
By Lemmas \ref{lem:typekr0} and \ref{lem:typekr}
the parameters $\{k_i\}_{i \in \I}$ satisfy the following conditions:
\begin{equation}    \label{eq:constraints}
\begin{array}{c|l}
\textup{\sf T} & \hspace{4cm} \text{\rm Conditions}
\\ \hline
\textup{\sf DS} &
  \begin{array}{l}
    \rule{0mm}{4.3mm}
     k_0 k_1 k_2 k_3 = q^{-n-1} \\
    \text{\rm Neither of $\pm k_0 k_3$ is among $q^{-1},q^{-2},q^{-3},\ldots, q^{-n}$} \\
    \text{\rm None of $\pm k_0$, $\pm k_1$, $\pm k_2$, $\pm k_3$ is among $q^{-1},q^{-2},q^{-3},\ldots,q^{-n/2}$}
  \end{array}
\\ \hline
\textup{\sf DDa} &
 \begin{array}{l}
   \rule{0mm}{4.3mm}
    k_0^2 = q^{-n-1} \\
   \text{\rm None of $\pm k_3^{\pm 1}$ is among $1,q,q^2,\ldots, q^{(n-1)/2}$} \\
   \text{\rm None of $k_0k_3k_1^{\pm 1}k_2^{\pm 1}$ is among $q^{-1},q^{-3},q^{-5},\ldots,q^{-n}$}
 \end{array}
\\ \hline
\textup{\sf DDb} &
 \begin{array}{l}
   \rule{0mm}{4.3mm}
   k_3^2 = q^{-n-1}  \\
   \text{\rm None of $\pm k_0^{\pm 1}$ is among $1,q,q^2,\ldots, q^{(n-1)/2}$} \\
   \text{\rm None of $k_0k_3k_1^{\pm 1}k_2^{\pm 1}$ is among $q^{-1},q^{-3},q^{-5},\ldots,q^{-n}$}
 \end{array}
\\ \hline 
\textup{\sf SSa} &
 \begin{array}{l}
   \rule{0mm}{4.3mm}
   k_1^2 = q^{-n-1}  \\
   \text{\rm None of $\pm k_2^{\pm 1}$ is among $1,q,q^2,\ldots, q^{(n-1)/2}$} \\
   \text{\rm None of $k_1k_2k_0^{\pm 1}k_3^{\pm 1}$ is among $q^{-1},q^{-3},q^{-5},\ldots,q^{-n}$}
 \end{array}
\\ \hline
\textup{\sf SSb} &
 \begin{array}{l}
   \rule{0mm}{4.3mm}
   k_2^2 = q^{-n-1}  \\
   \text{\rm None of $\pm k_1^{\pm 1}$ is among $1,q,q^2,\ldots, q^{(n-1)/2}$} \\
   \text{\rm None of $k_1k_2k_0^{\pm 1}k_3^{\pm 1}$ is among $q^{-1},q^{-3},q^{-5},\ldots,q^{-n}$}
 \end{array}
\end{array}
\end{equation}

In this section we prove the following result.

\begin{proposition}     \label{prop:const}   \samepage
\ifDRAFT {\rm prop:const}. \fi
Let $n \geq 0$ denote an integer and let $\{k_i\}_{i \in \I}$ denote 
nonzero scalars in $\F$.
Let  \textup{\sf T} be among \textup{\sf DS}, \textup{\sf DDa}, 
\textup{\sf DDb}, \textup{\sf SSa}, \textup{\sf SSb}.
Assume that $n$ is even if \textup{\sf T} is \textup{\sf DS} and odd otherwise.
Assume that $\{k_i\}_{i \in \I}$ satisfy the conditions \eqref{eq:constraints}.
Then there exists an XD $\Hq$-module $\V$ with dimension $n+1$ that has
$X$-type \textup{\sf T} and parameter sequence $\{k_i\}_{i \in \I}$
that is consistent with a standard ordering of the eigenvalues of $X$.
\end{proposition}

Let $n \geq 0$ denote an integer and let $\{k_i\}_{i \in \I}$ denote nonzero scalars in $\F$.
Fix {\sf T} among {\sf DS}, {\sf DDa}, {\sf DDb}, {\sf SSa}, {\sf SSb}.
Assume that $n$ is even if \textup{\sf T} is \textup{\sf DS} and odd otherwise.
Assume that $\{k_i\}_{i \in \I}$ satisfy the conditions \eqref{eq:constraints}.
Define scalars $\{\mu_r\}_{r=0}^n$ as follows.
If {\sf T} is among {\sf DS}, {\sf DDa}, {\sf DDb},
\begin{align}       \label{eq:Dmur2}
 \mu_r &= 
  \begin{cases}
    k_0k_3 q^{r}            &  \text{ if $r$ is even},  \\
    \frac{1}{k_0k_3q^{r+1} } & \text{ if $r$ is odd}
  \end{cases}
  && (0 \leq r \leq n).
\end{align}
If {\sf T} is among {\sf SSa}, {\sf SSb},
\begin{align}       \label{eq:Smur2}
 \mu_r &= 
  \begin{cases}
    \frac{1}{k_1k_2 q^{r+1} }  &  \text{ if $r$ is even},  \\
    k_1k_2 q^r                & \text{ if $r$ is odd}
  \end{cases}
  && (0 \leq r \leq n).
\end{align}

\begin{lemma}   \label{lem:murdistinct}  \samepage
\ifDRAFT {\rm lem:murdistinct}. \fi
The scalars $\{\mu_r\}_{r=0}^n$ are mutually distinct.
\end{lemma}

\begin{proof}
Routine using \eqref{eq:constraints}.
\end{proof}

\begin{lemma}  \label{lem:murmur+1}
\ifDRAFT {\rm lem:murmur+1}. \fi
The reduced diagram of $\{\mu_r\}_{r=0}^n$ is as follows:
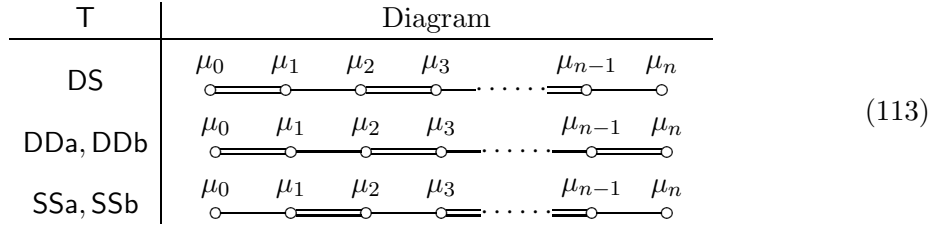
\begin{equation}      \label{eq:diagrams}
\begin{array}{c|c}
\textup{\sf T} & \text{\rm Diagram} 
\\ \hline
\textup{\sf DS} &
$\begin{xy}
  \ar @{=} (-10,0) *++!D{\mu_0} *\cir<2pt>{}; 
            (0,0)   *++!D{\mu_1} *\cir<2pt>{}="B"
  \ar @{-} "B" ; (10,0) *++!D{\mu_2} *\cir<2pt>{}="C"
  \ar @{=} "C" ; (20,0) *++!D{\mu_3} *\cir<2pt>{}="D"
  \ar @{-} "D" ; (30,0) *{\cdots\cdots}="E"
  \ar @{=} "E" ; (40,0) *++!D{\mu_{n-1}} *\cir<2pt>{}="F"
  \ar @{-} "F" ; (50,0) *++!D{\mu_n}  *\cir<2pt>{};
 \end{xy}$
\\
\textup{\sf DDa}, \textup{\sf DDb} & 
$ \begin{xy}
  \ar @{=} (-10,0) *++!D{\mu_0} *\cir<2pt>{}; 
            (0,0)   *++!D{\mu_1} *\cir<2pt>{}="B"
  \ar @{-} "B" ; (10,0) *++!D{\mu_2} *\cir<2pt>{}="C"
  \ar @{=} "C" ; (20,0) *++!D{\mu_3} *\cir<2pt>{}="D"
  \ar @{-} "D" ; (30,0) *{\cdots\cdots}="E"
  \ar @{-} "E" ; (40,0) *++!D{\mu_{n-1}} *\cir<2pt>{}="F"
  \ar @{=} "F" ; (50,0) *++!D{\mu_n}  *\cir<2pt>{};
 \end{xy}$
\\
\textup{\sf SSa}, \textup{\sf SSb} &
$ \begin{xy}
  \ar @{-} (-10,0) *++!D{\mu_0} *\cir<2pt>{}; 
            (0,0)   *++!D{\mu_1} *\cir<2pt>{}="B"
  \ar @{=} "B" ; (10,0) *++!D{\mu_2} *\cir<2pt>{}="C"
  \ar @{-} "C" ; (20,0) *++!D{\mu_3} *\cir<2pt>{}="D"
  \ar @{=} "D" ; (30,0) *{\cdots\cdots}="E"
  \ar @{=} "E" ; (40,0) *++!D{\mu_{n-1}} *\cir<2pt>{}="F"
  \ar @{-} "F" ; (50,0) *++!D{\mu_n}  *\cir<2pt>{};
 \end{xy}$
\end{array}
\end{equation}
\end{lemma}

\begin{proof}
Follows from \eqref{eq:Dmur2} and \eqref{eq:Smur2}.
\end{proof}

We now construct an $\Hq$-module.
Let $\V$ denote a vector space over $\F$ with dimension $n+1$,
and let  $\{v_r\}_{r=0}^n$ denote a basis for $\V$.
We define the action of $\{t_i\}_{i \in \I}$ on $\{v_r\}_{r=0}^n$ as follows.
Recall the function $G$ from \eqref{eq:defG}.
For $0 \leq r \leq n-1$ such that $\mu_r$, $\mu_{r+1}$ are $1$-adjacent,
we define the action of $t_0$, $t_3$ by
\begin{align}
 t_0 v_r &=
  \frac{\mu_r (k_0+k_0^{-1}) - k_3-k_3^{-1}}
       {\mu_r - \mu_r^{-1}} v_r
 + \frac{\mu_r}
        {\mu_r - \mu_r^{-1}}  v_{r+1},              \label{eq:t0vr2}
\\
 t_0 v_{r+1} &=
  \frac{G(\mu_r, k_0,k_3)}
       {\mu_r (\mu_r^{-1} - \mu_r)} v_{r} 
  + \frac{\mu_r^{-1} (k_0+k_0^{-1}) - k_3-k_3^{-1}}
         {\mu_r^{-1} - \mu_r} v_{r+1},              \label{eq:t0vr+12}
\\
 t_3 v_r &=
   \frac{\mu_r (k_3+k_3^{-1}) - k_0-k_0^{-1}}
        {\mu_r - \mu_r^{-1}} v_r
 +  \frac{1}
         {\mu_r^{-1} - \mu_r} v_{r+1},              \label{eq:t3vr2}
\\
 t_3 v_{r+1} &=
     \frac{G(\mu_r, k_0,k_3)}
          {\mu_r - \mu_r^{-1}}  v_r
  +  \frac{\mu_r^{-1} (k_3+k_3^{-1}) - k_0-k_0^{-1}}
          {\mu_r^{-1} - \mu_r} v_{r+1}.              \label{eq:t3vr+12}
\end{align}
For $0 \leq r \leq n-1$ such that $\mu_r$, $\mu_{r+1}$ are $q$-adjacent,
we define the action of $t_1$, $t_2$ by
\begin{align}
t_1 v_r &=
   \frac{q^{-1}\mu_r^{-1} (k_1+k_1^{-1}) - k_2-k_2^{-1}}
        {q^{-1}\mu_r^{-1} - q\mu_r} v_r
 + \frac{1}
        {q\mu_r - q^{-1}\mu_r^{-1}} v_{r+1},   \label{eq:t1vr2}
\\
t_1 v_{r+1} &=
   \frac{G(q \mu_r, k_1,k_2)}
        {q^{-1}\mu_r^{-1} - q\mu_r} v_r
 + \frac{q\mu_r (k_1+k_1^{-1}) - k_2-k_2^{-1}}
        {q\mu_r - q^{-1}\mu_r^{-1}}  v_{r+1},  \label{eq:t1vr+12}
\\
t_2 v_r &=
   \frac{q^{-1}\mu_r^{-1} (k_2+k_2^{-1}) - k_1-k_1^{-1}}
       {q^{-1}\mu_r^{-1} - q\mu_r} v_r
 + \frac{q^{-1}\mu_r^{-1}}
        {q^{-1}\mu_r^{-1} - q\mu_r} v_{r+1},    \label{eq:t2vr2}
\\
t_2 v_{r+1} &=
   \frac{q \mu_r\, G(q \mu_r, k_1,k_2)}
        {q\mu_r - q^{-1}\mu_r^{-1}} v_r 
 + \frac{q\mu_r (k_2+k_2^{-1}) - k_1-k_1^{-1}}
        {q\mu_r - q^{-1}\mu_r^{-1}} v_{r+1}.   \label{eq:t2vr+12}
\end{align}
In \eqref{eq:t0vr2}--\eqref{eq:t2vr+12} the denominators are nonzero 
by Lemmas  \ref{lem:murdistinct} and \ref{lem:murmur+1}.
We have defined some actions of $\{t_i\}_{i \in \I}$.
The remaining actions are defined as follows:
\begin{equation}          \label{eq:defonv0vn}
\begin{array}{c|llll}
\textup{\sf T} & \multicolumn{4}{c}{\text{\rm Action}}
\\ \hline \rule{0mm}{4.3mm}
 \textup{\sf DS} &
 t_0 v_0 = k_0 v_0 \;\; & t_3 v_0 = k_3 v_0\;\;  & t_1 v_n = k_1 v_n \;\;& t_2 v_n = k_2 v_n  
\\  \rule{0mm}{4mm}
 \textup{\sf DDa} & 
 t_0 v_0 = k_0 v_0 & t_3 v_0 = k_3 v_0 & t_0 v_n = k_0 v_n & t_3 v_n = k_3^{-1} v_n
\\ \rule{0mm}{4mm}
 \textup{\sf DDb} & 
 t_0 v_0 = k_0 v_0 & t_3 v_0 = k_3 v_0 & t_0 v_n = k_0^{-1} v_n & t_3 v_n = k_3 v_n
\\ \rule{0mm}{4mm}
 \textup{\sf SSa} & 
 t_1 v_0 = k_1 v_0 & t_2 v_0 = k_2 v_0 & t_1 v_n = k_1 v_n & t_2 v_n = k_2^{-1} v_n
\\ \rule{0mm}{4mm}
 \textup{\sf SSb} & t_1 v_0 = k_1 v_0 & t_2 v_0 = k_2 v_0 & t_1 v_n = k_1^{-1} v_n & t_2 v_n = k_2 v_n
\end{array}
\end{equation}
For $i \in \I$ we define the action of $t_i^{-1}$ on $\{v_r\}_{r=0}^n$ by 
\begin{align}   \label{eq:deftiinv}
  t_i^{-1}v_r &= (k_i+k_i^{-1})v_r - t_i v_r
   && (0 \leq r \leq n).
\end{align}
We have defined the action of \{$t_i^{\pm 1}\}_{i \in \I}$ on $\{v_r\}_{r=0}^n$.

\begin{lemma}  \label{lem:Hqmodule} \samepage
\ifDRAFT {\rm lem:Hqmodule}. \fi
The above actions of $\{t_i^{\pm 1}\}_{i \in \I}$ on $\{v_r\}_{r=0}^n$ give an $\Hq$-module structure
on $\V$.
\end{lemma}

\begin{proof}
One routinely checks that the defining relations \eqref{eq:defHq1}--\eqref{eq:defHq3} of $\Hq$
hold on $\{v_r\}_{r=0}^n$.
\end{proof}

\begin{lemma}  \label{lem:Xeigen}  \samepage
\ifDRAFT {\rm lem:Xeigen}. \fi
For $0 \leq r \leq n$ the vector $v_r$ is an eigenvector of $X$
with eigenvalue $\mu_r$.
Moreover $X$ is diagonalizable on $\V$.
\end{lemma}

\begin{proof}
Pick any integer $r$ such that $0 \leq r \leq n$.
One checks $t_3t_0 v_r = \mu_r v_r$.
Therefore $v_r$ is an eigenvector of $X$ with eigenvalue $\mu_r$.
Now $X$ has $n+1$ mutually distinct eigenvalues on $\V$ by
Lemma \ref{lem:murdistinct}. So $X$ is diagonalizable on $\V$.
\end{proof}

\begin{lemma}  \label{lem:G0vrG2vr}  \samepage
\ifDRAFT {\rm lem:G0vrG2vr}. \fi
For $0 \leq r \leq n-1$ the following hold.
\begin{itemize}
\item[\rm (i)]
Assume that $\mu_r$, $\mu_{r+1}$ are $1$-adjacent.
Then $G_0 v_r = v_{r+1}$ and $G_0 v_{r+1} = G(\mu_r,k_0,k_3) v_r$.
\item[\rm (ii)]
Assume that $\mu_r$, $\mu_{r+1}$ are $q$-adjacent.
Then $G_2 v_r = v_{r+1}$ and  $G_2 v_{r+1} = G(q\mu_r,k_1,k_2) v_r$.
\end{itemize}
\end{lemma}

\begin{proof}
Routine verification.
\end{proof}

\begin{lemma}  \label{lem:G02G22}  \samepage
\ifDRAFT {\rm lem:G02G22}. \fi
For $0 \leq r \leq n-1$ the following hold.
\begin{itemize}
\item[\rm (i)]
Assume that $\mu_r$, $\mu_{r+1}$ are $1$-adjacent.
Then $G(\mu_r,k_0,k_3) \neq 0$.
\item[\rm (ii)]
Assume that $\mu_r$, $\mu_{r+1}$ are $q$-adjacent.
Then $G(q\mu_r,k_1,k_2) \neq 0$.
\end{itemize}
\end{lemma}

\begin{proof}
(i):
We claim that
$\mu_r$ is not among  $k_0k_3$, $k_0k_3^{-1}$, $k_0^{-1}k_3$, $k_0^{-1}k_3^{-1}$.
First assume that {\sf T} is \textup{\sf DS}.
Note that $r$ is odd by \eqref{eq:diagrams} and since $\mu_r$, $\mu_{r+1}$ are
$1$-adjacent.
By \eqref{eq:Dmur2}  $\mu_r = (k_0k_3q^{r+1})^{-1}$.
By \eqref{eq:constraints} $k_0^2 k_3^2$ is not among $q^{-2}$, $q^{-4}$, \ldots, $q^{-2n}$,
and neither of $k_0^2$, $k_3^2$ is among $q^{-2}$, $q^{-4}$, \ldots, $q^{-n}$.
By these comments $\mu_r$ is not among $k_0 k_3$, $k_0 k_3^{-1}$, $k_0^{-1} k_3$.
Moreover $\mu_r$ is not equal to $k_0^{-1} k_3^{-1}$ since $q$ is not a root of unity.
Thus the claim holds when \textup{\sf T} is \textup{\sf DS}.
Next assume that \textup{\sf T} is \textup{\sf DDa}.
Note that $n$ is odd by the construction,
and $r$ is odd by \eqref{eq:diagrams} and since $\mu_r$, $\mu_{r+1}$ are
$1$-adjacent.
By \eqref{eq:Dmur2}  $\mu_r = (k_0k_3q^{r+1})^{-1}$.
By \eqref{eq:constraints} $k_0^2 = q^{-n-1}$,
and $k_3^2$ is not among $1$, $q^2$, $q^4$, \ldots, $q^{n-1}$.
By these comments $\mu_r$ is not among $k_0 k_3$, $k_0^{-1} k_3$.
Moreover $\mu_r$ is not among $k_0 k_3^{-1}$, $k_0^{-1} k_3^{-1}$ since
$q$ is not a root of unity.
Thus the claim holds when \textup{\sf T} is \textup{\sf DDa}.
In a similar way we can show the claim when \textup{\sf T} is among 
\textup{\sf DDb}, \textup{\sf SSa}, \textup{\sf SSb}.
Now $G(\mu_r,k_0,k_3) \neq 0$ by \eqref{eq:defG} and the claim.

(ii):
Similar.
\end{proof}

\begin{lemma}   \label{lem:irred} \samepage
\ifDRAFT {\rm lem:irred}. \fi
The $\Hq$-module $\V$ is irreducible and has parameter sequence 
$\{k_i\}_{i \in \I}$.
\end{lemma}

\begin{proof}
Let $\W$ denote a nonzero $\Hq$-submodule of $\V$.
We claim that $v_s \in \W$ for some $s$ $(0 \leq s \leq n)$.
Let $s$ $(0 \leq s \leq n)$ denote the maximal integer such that
$\W \cap \sum_{r=s}^n \F v_r$ is nonzero.
Pick a nonzero vector $w$ in $\W \cap \sum_{r=s}^n \F v_r$, and
write
$w = \sum_{r=s}^n \alpha_r v_r$.
Note that $\alpha_s \neq 0$ by the maximality of $s$.
By Lemma \ref{lem:Xeigen}
$X w = \sum_{r=s}^n \alpha_r \mu_r v_r$.
By these comments 
\[
  \mu_s w - X w = \sum_{r=s+1}^n (\mu_s - \mu_r) \alpha_r v_r.
\]
By the maximality of $s$ we must have $\mu_s w - X w = 0$.
By Lemma \ref{lem:murdistinct} $\mu_s - \mu_r \neq 0$ for $s+1 \leq r \leq n$.
By these comments $\alpha_r =0$ for $s+1 \leq r \leq n$.
Therefore $w = \alpha_s v_s$ and the claim follows.
By the claim and Lemmas  \ref{lem:G0vrG2vr}, \ref{lem:G02G22} we find that
$v_r \in \W$ for all $r$ $(0 \leq r \leq n)$.
So $\W = \V$.
We have shown that $\V$ is an irreducible $\Hq$-module.
By \eqref{eq:deftiinv} $\{k_i\}_{i \in \I}$ is a parameter sequence of $\V$.
\end{proof}

\begin{proofof}{Proposition \ref{prop:const}} 
In the above we have constructed an $\Hq$-module $\V$.
By Lemmas  \ref{lem:Xeigen} and \ref{lem:irred} the $\Hq$-module $\V$ is XD,
and $\{k_i\}_{i \in \I}$ is a parameter sequence of $\V$.
By \eqref{eq:defonv0vn} and Lemmas \ref{lem:murmur+1}, \ref{lem:Xeigen},
$\V$ has $X$-type {\sf T} and $\{\mu_r\}_{r=0}^n$ is a standard
ordering of the eigenvalues of $X$.
Moreover 
the parameter sequence $\{k_i\}_{i \in \I}$ is consistent with the ordering $\{\mu_r\}_{r=0}^n$.
The result follows.
\end{proofof}

\section{Proof of Theorem \ref{thm:main}; ``if'' direction}
\label{sec:if}

In this section we prove the ``if'' direction of Theorem \ref{thm:main}.
Let $A,A^*$ denote a Leonard pair on $V$
and let $A',A^{* \prime}$ denote a Leonard pair on $V'$.
Assume that these Leonard pairs have $q$-Racah type,
and let $(a,b,c,d)$ (resp.\ $(a',b',c',d')$) denote a Huang data of
$A,A^*$ (resp.\ $A', A^{* \prime}$).
Assume that these Huang data satisfy one of 
the conditions (i)--(vii) in Theorem \ref{thm:main}.
We show that there exists a feasible $\Hq$-module structure on $V \oplus V'$
such that $V$, $V'$ are the eigenspaces of $t_0$ and \eqref{eq:linked} holds.
We may assume that the Huang data satisfy one of (i)--(v) by exchanging our two 
Leonard pairs if necessary (see Remark \ref{rem:exchange}).
For notational convenience,
we rename the cases (i)--(v) to increase the compatibility with Corollary \ref{cor:Huangdata};
thus we assume that one of the following cases occurs:
\begin{equation}               \label{eq:cases}
\begin{array}{c|c|c|c|c|ccc}
\text{Case} & d'-d & a'/a & b'/b & c'/c &  & \text{\rm Inequalities} &  
\\ \hline 
\rule{0mm}{4.3mm}
\textup{\sf DS} & -1 & q & q & q & \quad a^2 \neq q^{-2d} & \quad  b^2 \neq q^{-2d}
\\
\rule{0mm}{4mm}
\textup{\sf DDa} & -2 & 1 & 1 & 1
\\
\rule{0mm}{4mm}
\textup{\sf DDb} & 0 & 1 & q^2 & 1  & \quad a^2 \neq q^{\pm 2d} & & \quad b^2 \neq q^{-2}
\\
\rule{0mm}{4mm}
\textup{\sf SSa} & 0 & q^2 & 1 & 1 & & \quad b^2 \neq q^{\pm 2d} & \quad a^2 \neq q^{-2}
\\
\rule{0mm}{4mm}
\textup{\sf SSb} & 0 & 1 & 1 & q^2 & \quad a^2 \neq q^{\pm 2d} & \quad b^2 \neq q^{\pm 2d } & \quad c^2 \neq q^{-2} 
\end{array}
\end{equation}
For each of the above cases
we define $\{k_i\}_{i \in \I}$ as follows:
\begin{equation}                   \label{eq:defki}
\begin{array}{c|ccccc}
\text{Case} & k_0 & k_1 & k_2 & k_3 
\\ \hline
\rule{0mm}{4.5mm}
\textup{\sf DS}  
 & (abcq^{1-d})^{1/2} & aq^{-d}k_0^{-1} & cq^{-d}k_0^{-1} & bq^{-d}k_0^{-1}
\\
\rule{0mm}{4mm}
\textup{\sf DDa}  & q^{-d} &  a & c & b
\\
\rule{0mm}{4mm}
\textup{\sf DDb}  &  bq & c & a & q^{-d-1}
\\
\rule{0mm}{4mm}
\textup{\sf SSa}  & aq & q^{-d-1} & b & c
\\
\textup{\sf SSb}  & cq & b & q^{-d-1} & a
\end{array}
\end{equation}
In case {\sf DS} we may take either square root as the value of $k_0$;
see Remark \ref{rem:k0} below.
Set $n = d+d'+1$.
So $V \oplus V'$ has dimension $n+1$.
Note that $n$ is even in case {\sf DS}, and odd in the other cases.
The following two lemmas are immediate from \eqref{eq:defki}.

\begin{lemma}  \label{lem:restrictionki}  \samepage
\ifDRAFT {\rm lem:restrictionki}. \fi
The scalars $\{k_i\}_{i \in \I}$ satisfy the following equation:
\begin{equation}    \label{eq:restrictionki}
\begin{array}{c|c}
 \text{\rm Case} &   \text{\rm Equation}
\\ \hline \rule{0mm}{4.3mm}
 \textup{\sf DS} & \quad k_0k_1k_2k_3 = q^{-n-1}
\\ \rule{0mm}{4mm}
 \textup{\sf DDa} &   k_0^2 = q^{-n-1} 
\\  \rule{0mm}{4mm}
 \textup{\sf DDb} &  k_3^2 = q^{-n-1} 
\\ \rule{0mm}{4mm}
 \textup{\sf SSa} &  k_1^2 = q^{-n-1} 
\\ \rule{0mm}{4mm}
 \textup{\sf SSb} &   k_2^2 = q^{-n-1}
\end{array}
\end{equation}
\end{lemma}

\begin{lemma}   \label{lem:abcdinki}
\ifDRAFT {\rm lem:abcdinki}. \fi
The Huang data $(a,b,c,d)$ and $(a',b',c',d')$ are represented in terms of $\{k_i\}_{i \in \I}$
as follows:
\begin{equation}             \label{eq:abcdinki2}
\begin{array}{c|cccc}
 \text{\rm Case} &
  \begin{array}{c}
    a \\
    a'
  \end{array}
&
  \begin{array}{c}
    b \\
    b'
  \end{array}
&
  \begin{array}{c}
   c \\
   c' 
  \end{array}
&
 \begin{array}{c}
   d \\
   d'
 \end{array}
\\ \hline
 \rule{0mm}{6.5mm}
 \textup{\sf DS} 
& 
  \begin{array}{c}
    k_0 k_1 q^{n/2}  \\
    k_0 k_1 q^{(n+2)/2}
 \end{array}
&
  \begin{array}{c}
    k_0 k_3 q^{n/2}  \\
    k_0 k_3 q^{(n+2)/2}
  \end{array}
& 
  \begin{array}{c}
   k_0 k_2 q^{n/2}  \\
   k_0 k_2 q^{(n+2)/2}
 \end{array}
&
  \begin{array}{c}
   n/2 \\
  (n-2)/2
  \end{array}
\\ \hline
 \rule{0mm}{6.5mm}
 \textup{\sf DDa}
& 
 \begin{array}{c}
   k_1 \\
   k_1
 \end{array}
&
 \begin{array}{c}
  k_3 \\
  k_3
 \end{array}
&
 \begin{array}{c}
  k_2 \\
  k_2
 \end{array}
&
 \begin{array}{c}
  (n+1)/2 \\
  (n-3)/2
 \end{array}
\\ \hline
 \rule{0mm}{6.5mm}
 \textup{\sf DDb} 
&
 \begin{array}{c}
   k_2 \\
   k_2
 \end{array}
&
 \begin{array}{c}
   k_0 q^{-1} \\
   k_0 q
 \end{array}
&
 \begin{array}{c}
   k_1 \\
   k_1
 \end{array}
&
 \begin{array}{c}
  (n-1)/2 \\
  (n-1)/2
 \end{array}
\\ \hline
 \rule{0mm}{6.5mm}
 \textup{\sf SSa}
&
 \begin{array}{c}
  k_0 q^{-1} \\
  k_0 q
 \end{array}
&
 \begin{array}{c}
   k_2 \\
   k_2
 \end{array}
&
 \begin{array}{c}
  k_3 \\
  k_3
 \end{array}
&
 \begin{array}{c}
  (n-1)/2 \\
  (n-1)/2
 \end{array}
\\ \hline
  \rule{0mm}{6.5mm}
 \textup{\sf SSb} 
&
 \begin{array}{c}
   k_3 \\ 
   k_3
 \end{array}
&
 \begin{array}{c}
  k_1 \\
  k_1
 \end{array}
&
 \begin{array}{c}
   k_0 q^{-1} \\
   k_0 q
 \end{array}
&
 \begin{array}{c}
  (n-1)/2 \\
  (n-1)/2
 \end{array}
\end{array}
\end{equation}
\end{lemma}

By Lemma \ref{lem:condabc}
the Huang data $(a,b,c,d)$ and $(a',b',c',d')$ satisfy the following inequalities:
\begin{align}
 & \text{Neither of  $a^2$, $b^2$ is among $q^{2d-2}, q^{2d-4},\ldots,q^{2-2d}$}.    \label{eq:condabc1}  \\
 & \text{None of  $abc$, $a^{-1}bc$, $ab^{-1}c$, $abc^{-1}$ is among 
               $q^{d-1},q^{d-3},\ldots,q^{1-d}$}.           \label{eq:condabc2}  \\
 & \text{Neither of  ${a'}^2$, ${b'}^2$ is among  $q^{2d'-2}, q^{2d'-4},\ldots,q^{2-2d'}$}.
                                                              \label{eq:condadbdcd1} \\
 & \text{None of  $a'b'c'$, ${a'}^{-1}b'c'$, $a'{b'}^{-1}c'$, $a'b'{c'}^{-1}$ is among
              $q^{d'-1},q^{d'-3},\ldots,q^{1-d'}$}.          \label{eq:condadbdcd2}
\end{align}
The inequalities \eqref{eq:condabc1}--\eqref{eq:condadbdcd2} have the following consequence.

\begin{lemma}  \label{lem:restrictions}  \samepage
\ifDRAFT {\rm lem:restrictions}. \fi
The scalars $\{k_i\}_{i \in \I}$ satisfy the following inequalities:
\[
\begin{array}{c|l}
\text{\rm Case} & \hspace{4cm} \text{\rm Inequalities}
\\ \hline
\textup{\sf DS} &
  \begin{array}{l}
    \rule{0mm}{4.3mm}
    \text{\rm None of $\pm k_0k_3$, $\pm k_0k_1$ is among $q^{-1},q^{-2},q^{-3},\ldots,q^{-n}$} \\
    \text{\rm None of $\pm k_0$, $\pm k_1$, $\pm k_2$, $\pm k_3$ is among $q^{-1},q^{-2},q^{-3},\ldots,q^{-n/2}$}
  \end{array}
\\ \hline
\textup{\sf DDa} &
 \begin{array}{l}
   \rule{0mm}{4.3mm}
   \text{\rm None of $\pm k_1^{\pm 1}$, $\pm k_3^{\pm 1}$ is among $1,q,q^2,\ldots,q^{(n-1)/2}$} \\
   \text{\rm None of $k_0k_3k_1^{\pm 1}k_2^{\pm 1}$ is among $q^{-1},q^{-3},q^{-5},\ldots,q^{-n}$}
 \end{array}
\\ \hline
\textup{\sf DDb} &
 \begin{array}{l}
   \rule{0mm}{4.3mm}
   \text{\rm None of $\pm k_0^{\pm 1}$, $\pm k_2^{\pm 1}$ is among $1,q,q^2,\ldots,q^{(n-1)/2}$} \\
   \text{\rm None of $k_0k_3k_1^{\pm 1}k_2^{\pm 1}$ is among $q^{-1},q^{-3},q^{-5},\ldots,q^{-n}$}
 \end{array}
\\ \hline 
\textup{\sf SSa} &
 \begin{array}{l}
   \rule{0mm}{4.3mm}
   \text{\rm None of $\pm k_0^{\pm 1}$, $\pm k_2^{\pm 1}$ is among $1,q,q^2,\ldots,q^{(n-1)/2}$} \\
   \text{\rm None of $k_1k_2k_0^{\pm 1}k_3^{\pm 1}$ is among $q^{-1},q^{-3},q^{-5},\ldots,q^{-n}$}
 \end{array}
\\ \hline
\textup{\sf SSb} &
 \begin{array}{l}
   \rule{0mm}{4.3mm}
   \text{\rm None of $\pm k_1^{\pm 1}$, $\pm k_3^{\pm 1}$ is among $1,q,q^2,\ldots,q^{(n-1)/2}$} \\
   \text{\rm None of $k_1k_2k_0^{\pm 1}k_3^{\pm 1}$ is among $q^{-1},q^{-3},q^{-5},\ldots,q^{-n}$}
 \end{array}
\end{array}
\]
\end{lemma}

\begin{proof}
Routine verification using \eqref{eq:cases}, \eqref{eq:defki}, \eqref{eq:condabc1}--\eqref{eq:condadbdcd2}.
\end{proof}

Let {\sf T} be among {\sf DS}, {\sf DDa}, {\sf DDb}, {\sf SSa}, {\sf SSb}.
By Lemmas \ref{lem:restrictionki}, \ref{lem:restrictions} and Proposition \ref{prop:const}
there exists an XD $\Hq$-module $\V$ with dimension $n+1$
that has $X$-type {\sf T} and parameter sequence $\{k_i\}_{i \in \I}$ 
that is consistent with a standard ordering
of the eigenvalues of $X$.

\begin{lemma}  \label{lem:k02not1pre}
\ifDRAFT {\rm lem:k02not1pre}. \fi
$k_0^2 \neq 1$.
\end{lemma}

\begin{proof}
Note that the value of $k_0$ is given in \eqref{eq:defki}.
In case {\sf DS}, $k_0^2 = abcq^{1-d}$ and so $k_0^2 \neq 1$ by \eqref{eq:condabc2}.
In case {\sf DDa}, $k_0^2 = q^{-2d}$ and so $k_0^2 \neq 1$ by $d = d'+2 \geq 2$ and since
$q$ is not a root of unity.
In case {\sf DDb}, $k_0^2 = b^2 q^2$, so $k_0^2 \neq 1$ since
$b^2 \neq q^{-2}$ by \eqref{eq:cases}.
The proof is similar for the cases {\sf SSa}, {\sf SSb}.
\end{proof}

\begin{lemma}  \label{lem:k0not1}
\ifDRAFT {\rm lem:k0not1}. \fi
$t_0$ has two distinct eigenvalues $k_0$, $k_0^{-1}$ on $\V$.
\end{lemma}

\begin{proof}
Note that $t_0$ has two distinct eigenvalues if and only if each of $F^+ \V$ and $F^- \V$
is nonzero.
First assume that the reduced $X$-diagram of $\V$ has a single bond.
Then the result follows from Lemma \ref{lem:WF+WF-W}(ii).
Next assume that the reduced $X$-diagram has no single bond.
Then $n=1$ and the reduced $X$-diagram of $\V$ is a double bond.
So $\V$ has $X$-type \textup{\sf DDa} or \textup{\sf DDb}.
If $\V$ has $X$-type \textup{\sf DDa} then $n \geq 3$ since $d = d'+2 \geq  2$,
contradicting $n=1$.
Thus $\V$ has  $X$-type \textup{\sf DDb}.
By Lemma \ref{lem:typesDD}
the eigenvalues of $t_0$ on $\V_X(\mu_0)$ and $\V_X(\mu_1)$ are reciprocals.
By Lemma \ref{lem:k02not1pre} $k_0 \neq k_0^{-1}$.
By these comments $t_0$ has two distinct eigenvalues on $\V$.
\end{proof}

\begin{lemma}   \label{lem:feasible}    \samepage
\ifDRAFT {\rm lem:feasible}. \fi
The $\Hq$-module $\V$ is feasible.
\end{lemma}

\begin{proof}
By the construction $\V$ is XD.
By Corollary \ref{cor:Ydiagonalizable} and Lemmas \ref{lem:restrictionki},
\ref{lem:restrictions},
$Y$ is diagonalizable on $\V$, so $\V$ is YD.
By Lemma \ref{lem:k0not1} $t_0$ has two distinct eigenvalues on $\V$.
Thus $\V$ is feasible.
\end{proof}

\begin{proofof}{Theorem \ref{thm:main}; ``if'' direction}
By Theorem \ref{thm:main1} and Lemma \ref{lem:feasible} 
the pair $\A,\B$ acts on $\V(k_0)$ (resp.\ $\V(k_0^{-1})$)
as a Leonard pair of $q$-Racah type.
Moreover comparing \eqref{eq:Huangdata+}, \eqref{eq:Huangdata-} with \eqref{eq:abcdinki2}
we find that the Huang data of $\A,\B$ on $\V(k_0)$ (resp.\ $\V(k_0^{-1})$)
coincides with the Huang data of $A,A^*$ (resp.\ $A',A^{*\prime}$).
By this and Lemma \ref{lem:qRacahunique}
the Leonard pair $\A,\B$ on $\V(k_0)$ (resp.\ $\V(k_0^{-1})$) is isomorphic to
the Leonard pair $A,A^*$ on $V$ (resp.\ $A',A^{*\prime}$ on $V'$).
Let $f: \V(k_0) \to V$ (resp.\ $f': \V(k_0^{-1}) \to V'$) denote an isomorphism of Leonard pairs.
Recall that $\V = \V(k_0) + \V(k_0^{-1})$ (direct sum).
So we have the $\F$-linear bijection $f \oplus f' : \V \to V \oplus V'$.
We define the $\Hq$-module structure on $V \oplus V'$ so that $f \oplus f'$
is an isomorphism of $\Hq$-modules.
In this $\Hq$-module the spaces $V$, $V'$ are the eigenspaces of $t_0$,
and \eqref{eq:linked} holds by the construction.
Therefore the Leonard pairs $A,A^*$ and $A',A^{*\prime}$ are linked.
\end{proofof}

\begin{remark}  \label{rem:k0}  \samepage
\ifDRAFT {\rm rem:k0}. \fi
We defined the integer $n$ and the scalars $\{k_i\}_{i \in \I}$ in \eqref{eq:defki}, 
and constructed a feasible $\Hq$-module $\V$ that has dimension $n+1$ and 
parameter sequence $\{k_i\}_{i \in \I}$.
In the definition of $k_0$ in case {\sf DS}, there appears a square root, so the value of $k_0$ is determined up to sign.
By our construction we obtain an $\Hq$-module from each of two values of $k_0$.
These two $\Hq$-modules are not isomorphic; otherwise these two $\Hq$-modules must have the
same parameters up to reciprocal.
\end{remark}

\bigskip\bigskip\noindent
Kazumasa Nomura\\
Tokyo Medical and Dental University\\
Kohnodai, Ichikawa, 272-0827 Japan\\
email: knomura@pop11.odn.ne.jp

\bigskip\noindent
Paul Terwilliger\\
Department of Mathematics\\
University of Wisconsin\\
480 Lincoln Drive\\ 
Madison, Wisconsin, 53706 USA\\
email: terwilli@math.wisc.edu

\medskip\noindent
{\small
{\bf Keywords.} DAHA, Askey-Wilson polynomial, Leonard pair, tridiagonal pair
\\
\noindent
{\bf 2010 Mathematics Subject Classification.} 33D80, 33D45
}

\end{document}